\documentclass[a4paper,11pt]{article}
\usepackage[utf8]{inputenc}
\usepackage[T1]{fontenc}
\usepackage[french,english]{babel}
\usepackage{amssymb}
\usepackage{enumerate}
\usepackage{cancel}
\usepackage[normalem]{ulem}
\usepackage{amsthm}
\usepackage{amsmath}
\usepackage{longtable}
\usepackage{caption}
\usepackage{relsize}
\usepackage{xcolor}
\usepackage{enumitem}
\usepackage[document]{ragged2e}
\usepackage{hyperref}
\usepackage{geometry}
\usepackage{bm}
\usepackage{bbm}
\usepackage[capitalize]{cleveref}
\usepackage{thmtools}
\usepackage{cases}
\usepackage{libertine}
\usepackage{changepage}

\usepackage[title]{appendix}
\usepackage[backend=bibtex,
            giveninits=true,
            style=numeric,
            doi=true,
            isbn=false,
            url=false,
            eprint=false,
            maxbibnames=99,
            citestyle=numeric-comp
            ]{biblatex}
\DeclareFieldFormat[article]{title}{#1}
\DeclareFieldFormat[inbook]{title}{#1}
\DeclareFieldFormat[incollection]{title}{#1}

\renewbibmacro{in:}{%
  \ifentrytype{article}{}{\printtext{\bibstring{in}\intitlepunct}}%
}

\hypersetup{
    colorlinks=true,
    linkcolor=[rgb]{0.2, 0, 0.7},
    filecolor=magenta,
    urlcolor=cyan,
    bookmarks=true,
    citecolor=[rgb]{0,0.4,0},
}

\crefformat{equation}{#2\normalfont{{(\textup{#1})}}#3}

\newcommand{\RR}{\mathbb{R}}
\newcommand{\NN}{\mathbb{N}}

\newcommand{\PP}{\mathbb{P}}

\newcommand{\QQ}{\mathbb{Q}}

\newcommand{\cA}{\mathcal{A}}

\newcommand{\cN}{\mathcal{N}}
\newcommand{\cM}{\mathcal{M}}
\newcommand{\cT}{\mathcal{T}}
\newcommand{\cL}{\mathcal{L}}

\newcommand{\cV}{\mathcal{V}}

\newcommand{\cI}{\mathcal{I}}

\newcommand{\cB}{\mathcal{B}}

\newcommand{\cF}{\mathcal{F}}

\newcommand{\cH}{\mathcal{H}}
\newcommand{\cO}{\mathcal{O}}

\newcommand{\gothC}{\mathfrak{C}}
\newcommand{\gothH}{\mathfrak{H}}

\newcommand{\norm}[2]{ \left \lVert {#1}  \right \rVert_{#2}}
\newcommand{\module}[1]{\left \lvert {#1} \right \rvert}

\DeclareMathOperator{\Lip}{Lip}

\DeclareMathOperator{\loc}{loc}

\DeclareMathOperator{\diver}{div}

\DeclareMathOperator{\gph}{gph}
\DeclareMathOperator{\dom}{dom}

\DeclareMathOperator{\sur}{sur}
\DeclareMathOperator{\co}{co}

\renewcommand{\epsilon}{\varepsilon}

\newcommand{\WOTarrow}{\xrightarrow{\textnormal{WOT}}}

\renewcommand{\eqref}[1]{\cref{#1}}

\let\svthefootnote\thefootnote
\newcommand\freefootnote[1]{%
  \let\thefootnote\relax%
  \footnotetext{#1}%
  \let\thefootnote\svthefootnote%
}

\theoremstyle{plain}
\newtheorem{thm}{Theorem}[section]
\newtheorem{cor}[thm]{Corollary}
\newtheorem{lem}[thm]{Lemma}
\newtheorem{prop}[thm]{Proposition}

\theoremstyle{definition}
\newtheorem{defi}[thm]{Definition}

\theoremstyle{remark}
\newtheorem{rem}[thm]{Remark}

\crefname{thm}{Theorem}{Theorems}
\crefname{cor}{Corollary}{Corollaries}
\crefname{lem}{Lemma}{Lemmata}
\crefname{prop}{Proposition}{Propositions}
\crefname{def}{Definition}{Definitions}
\crefname{rem}{Remark}{Remarks}
\crefname{ex}{Example}{Examples}

\title{A nonsmooth extension of the Brezzi-Rappaz-Raviart approximation theorem via metric regularity techniques and applications to nonlinear PDEs}
\author{Jules Berry  \thanks{Université Paris-Saclay, CNRS, CentraleSupélec, Laboratoire des signaux et systèmes, 91190, Gif-sur-Yvette, France. Email: {\fontfamily{cmtt}\selectfont jules.berry@centralesupelec.fr}.} \and Olivier Ley \thanks{Univ Rennes, INSA, CNRS, IRMAR - UMR 6625, Rennes F-35000, France. Email: {\fontfamily{cmtt}\selectfont olivier.ley@insa-rennes.fr}. Corresponding author.} \and Francisco J. Silva \thanks{Laboratoire XLIM, Université de Limoges, Limoges, 87060, France. Email: {\fontfamily{cmtt}\selectfont francisco.silva@unilim.fr}.}}
\date{\today}

\begin{document}
\freefootnote{This work was partially supported by the ANR (Agence Nationale de la Recherche) through the COSS project ANR-22-CE40-0010 and the Centre Henri Lebesgue ANR-11-LABX-0020-01. The second author was partially supported by the SABOCPR project ANR-25-CE40-3469-01. The third author was partially supported by KAUST through the subaward agreement ORA-2021-CRG10-4674.6 and by the \textit{Ministère de l’Europe et des Affaires étrangères} through the subaward agreement MathAmsud Project 23-MATH-17.}
\maketitle
\justify

\begin{abstract}
  We generalize the Brezzi-Rappaz-Raviart approximation theorem, which allows to obtain existence and a priori error estimates for approximations of solutions to some nonlinear partial differential equations. Our contribution lies in the fact that we typically allow for nonlinearities having merely Lipschitz regularity, while previous results required some form of differentiability. This is achieved by making use of the theory of metrically regular mappings, developed in the context of variational analysis. We apply this generalization to derive quasi-optimal error estimates for finite element approximations to solutions of viscous Hamilton-Jacobi equations and second order mean field game systems.
\end{abstract}
{\small
\begin{adjustwidth}{0.8cm}{0.8cm}
\end{adjustwidth}
}

\section{Introduction}

In the celebrated paper \cite{BRR1980}, Brezzi, Rappaz and Raviart (BRR for short) introduced a general approach to prove existence and a priori error estimates for finite dimensional approximations of solutions to nonlinear partial differential equations.
As an illustrative example, let us consider the following viscous Hamilton-Jacobi equation
\begin{equation}\label{eq:HJ_intro}
 \begin{cases}
  - \Delta u(x) + H(x,Du) + \lambda u(x) = f(x) \quad & \textnormal{in } \Omega, \\
  u(x) = 0 \quad & \textnormal{on } \partial \Omega,
 \end{cases}
\end{equation}
where $\Omega$ is a bounded domain in $\RR^d$ and $H \colon \Omega \times \RR^d \to \RR$ is the Hamiltonian.
Under appropriate assumptions, a solution to \eqref{eq:HJ_intro} can be obtained as a fixed point of the mapping, which, to some function $v$, associates the unique solution $u$ to
\[
 \begin{cases}
  - \Delta u(x) + \lambda u(x) = f(x) - H(x,Dv(x)) \quad & \textnormal{in } \Omega, \\
  u(x) = 0 \quad & \textnormal{on } \partial \Omega.
 \end{cases}
\]
This fact motivates the following reformulation of the problem. Define a linear operator $T$ as the solution mapping of the equation $- \Delta v + \lambda v = g$ with homogeneous Dirichlet boundary condition and the nonlinear mapping
\[
 G(v) := H(\cdot,Dv) - f.
\]
Then, $u$ is a solution to \eqref{eq:HJ_intro} if and only if
\[
 F(u) := (I + T \circ G)(u) = 0.
\]
Moreover, if one wishes to find Galerkin approximations of a solution to \eqref{eq:HJ_intro}, \textit{i.e.,} find $u_h$ in some finite dimensional subspace $V_h \subset X$ such that
\begin{equation}\label{eq:HJ_galerkin_intro}
 \int_\Omega D u_h(x) \cdot D \phi_h(x) + H(x, Du_h(x)) \phi_h(x) + \lambda u_h(x) \phi_h(x) dx = \int_\Omega f(x) \phi_h(x) dx,
\end{equation}
for all test-functions $\phi_h \in V_h$, then the problem can also be expressed in the form $F_h(u_h) = 0$.

More generally, given some Banach spaces $X$ and $Y$, a mapping $F \colon X \to Y$, approximations $F_h \colon X \to Y$ of $F$, and a solution $\bar x \in X$ to
\begin{equation}\label{eq:main}
 F(\bar x) = 0,
\end{equation}
we are interested in finding solutions $\bar x_h \in X$ to
\begin{equation}\label{eq:main_approx}
 F_h(\bar x_h) = 0
\end{equation}
and to quantify the error $\norm{\bar x - \bar x_h}{X}$.

In its simplest form, the Brezzi-Rappaz-Raviart theorem then roughly\footnote{We refer to \cref{thm:reworked_BRR} and \cref{rem:reworked_BRR}-\ref{item:rem_reworked_BRR_simple} for a precise statement.} states that, if the mappings $F$ and $F_h$ are of class $C^1$, with $dF[\bar x]$ invertible, if
\begin{equation}\label{eq:intro_consistency}
 \lim_{h \to 0} F_h(\bar x) = 0
\end{equation}
and if
\begin{equation}\label{eq:intro_stability}
 \lim_{h \to 0} \norm{dF[\bar x] - dF_h[\bar x]}{\cL(X,Y)} = 0,
\end{equation}
then, for all $h$ small enough, there exists $\bar x_h \in X$ solving \eqref{eq:main_approx} and we have the error estimate
\begin{equation}\label{eq:intro_error}
 \norm{\bar x - \bar x_h}{X} \leq 2 \norm{dF[\bar x]^{-1}}{\cL(X,Y)} \norm{F_h(\bar x)}{Y}.
\end{equation}
In other words, under the consistency assumption \eqref{eq:intro_consistency} and the uniform convergence of the linearized operators \eqref{eq:intro_stability}, there exists a solution to the approximate problems and the error estimate \eqref{eq:intro_error} is proportional to the consistency error. This result has been used in the context of Navier-Stokes equations \cite{BRR1980,GR1986}, Von K\'arm\'an equations \cite{BRR1980}, semilinear elliptic equations \cite{CR1990,CR1997,KP2010} and semilinear parabolic equations \cite{CH2002}.

The standard proof of the BRR theorem relies on a quantitative version of the inverse function theorem, so that the differentiability of the mappings $F$ and $F_h$ plays a major role in the argument. This type of result was first obtained in \cite{BRR1980} in the case of finite dimensional approximations and was later extended to more general settings in subsequent works. These more general results can be found in the book by Girault and Raviart \cite[Section IV.3]{GR1986}, where they are attributed to M. Crouzeix, as well as in the monograph of Crouzeix and Rappaz \cite{CR1990} and in the survey paper by Caloz and Rappaz \cite{CR1997}.

Our goal in this paper is to generalize the BRR theorem to a broad class of merely Lipschitz continuous mappings. In particular to mappings that may fail to be differentiable. Some nondifferentiable problems were already considered by Rappaz \cite{R1984}, but assumptions remained close to differentiability. Our approach starts from the idea that the BRR approximation theorem naturally fits in the context of metric regularity. The theory of metrically regular mappings aims at extending the conclusions of the inverse function theorem beyond the case of differentiable mappings defined between Banach spaces. Standard references on this topic are the books by Mordukhovich \cite{M2005}, Dontchev and Rockafellar \cite{DR2014}, Ioffe \cite{I2017} and Thibault \cite{T2023}. This new perspective allows us to deal with a broader class of mappings $F$ while maintaining tractability.

In our version of the BRR theorem, the differentiability and the invertibility of the differential are therefore replaced by the metric regularity of the mapping $F$. This is the main result of this paper and it is stated in \cref{thm:approximation_main}. This generalization allows us to benefit from the wide range of sufficient conditions for metric regularity existing in the literature. One that seemed especially useful to us in the study of nonlinear partial differential equations is a form of generalized strict differentiability, which is recalled in \cref{thm:graves_variation}.

We provide two applications of our generalized BRR theorem. In the first one, whose purpose is mainly to expose the techniques in a simple setting, we obtain quasi-optimal error estimates for finite element approximations of solutions to the Hamilton-Jacobi equation \eqref{eq:HJ_intro}. We refer to Jensen and Smears \cite{JS2013} and Smears and Süli \cite{SS2014} for more results on finite element approximations of Hamilton-Jacobi-Bellman equations. The second and main application, which results in \cref{cor:quasi_optimal_mfg}, deals with the derivation of quasi-optimal error estimates in the context of a stationary second order mean field game system. We recall that the theory of mean field games was introduced independently by Lasry and Lions \cite{LL2006,LL2006a,LL2007} and Huang, Malhamé and Caines \cite{HMC2006,HCM2007} and deals with the study of dynamic games involving large numbers of identical players. The problem can then typically be written in the form of a coupled system of partial differential equations, called the mean field game system, involving an Hamilton-Jacobi equation similar to \eqref{eq:HJ_intro} and an equation governing the distributions of the players. We refer the reader to \cite{BFY2013,GS2014,ACDPS2020} for general references on the theory and to \cite{A2013,AL2020} for surveys about the numerical analysis of mean field game systems. Finite element approximation of mean field game systems were previously considered by Osborne and Smears \cite{OS2024,OS2025,OS2025a,OS2025b}. In \cite{OS2024} they proved the existence and convergence, without error estimates, of finite element approximations of solutions to the stationary mean field game system. Similar results were obtained in \cite{OS2025} in the case of a time-dependent problem. Quasi-optimal error estimates, under the so-called Lasry-Lions monotonicity conditions were then proved in \cite{OS2025a} and extended to a more general setting in \cite{OS2025b}. Simultaneously, in \cite{BLS2025}, we also proved some error estimates for the stationary mean field game system. The main difference with \cite{OS2025}, is that our result does not require the monotonicity condition but assumes instead the stability of the solution, in the sense of Briani and Cardaliaguet \cite{BC2018}. One of the interesting features of the Brezzi-Rappaz-Raviart theorem is that no uniqueness is required on the solution $\bar x$ of the continuous problem \eqref{eq:main}. This aspect motivated us to use this framework to obtain the results in \cite{BLS2025}, since solutions to mean field game systems are not unique in general if the Lasry-Lions condition does not hold. However, our use of the BRR theorem in \cite{BLS2025} led us to assume smoothness of the Hamiltonian in order to satisfy the differentiability assumption. In this paper, we are able to substantially relax this assumption by requiring merely $C^{1,1}$ Hamiltonians, which could be further relaxed to $C^{1,1}_{\textnormal{loc}}$ with some additional technicalities. This type of regularity naturally appears in the context of optimal control problems. We emphasize that the \cref{cor:quasi_optimal_mfg} holds under a generalized stability condition which is implied by the Lasry-Lions condition. Moreover, in the case of smooth Hamiltonians, the results in this paper yield better error estimates than those of our previous work \cite{BLS2025}. In the monotone setting, our results are similar to those of \cite{OS2025a}, with the difference that we drop a stabilization term that was used in \cite{OS2025a} to obtain a discrete maximum principle. However, we leave the question of existence of stable solutions, in our generalized sense, when monotonicity fails, for future research. Our approach also applies to time-dependent MFG systems, for which the first author obtained semidiscrete error estimates in the case of smooth Hamiltonians \cite{B2025}. We finally mention that Osborne, Smears and Wells recently studied a posteriori error estimates in \cite{OSW2025}. Given the general nature of our result and the many applications of the original BRR theorem, we expect that our generalization may be useful beyond the context of mean field games.

In order to use \cref{thm:graves_variation} as a sufficient condition for metric regularity in the applications mentioned above, we had to study a generalized differential for Nemytskii operators defined on Lebesgue spaces. Loosely speaking, it is defined as the set of all measurable selections inside the Clarke generalized Jacobian of the function defining the Nemytskii operator. Our analysis extends previous results obtained by Ulbrich \cite{U2002,U2011} and establishes new results for this object.

The Brezzi-Rappaz-Raviart theorem allows to prove simultaneously the existence of solutions and the error estimates. In future work, we plan to decouple these two aspects and to focus on the derivation of a priori error estimates. The motivation is that the existence may often be proved by other means, for instance by using a fixed-point theorem. We expect that this decoupling will allow to relax some of the assumptions of the BRR theorem. As for the applications to Hamilton-Jacobi equations and mean field games we plan to investigate the applicability of our approach to other numerical schemes and different settings, such as fully nonlinear problems and equations involving nonlocal operators.

The paper is structured as follows. We start in \cref{section:metric_regularity} by recalling some of the main definitions and results about metric regularity. Our generalized BRR theorem is proved in \cref{section:main_BRR}. \cref{section:nemytskii} is dedicated to the study of the generalized differential for Nemytskii operators defined on Lebesgue spaces which will be useful in the applications, which are treated in \cref{section:application}. Finally, \cref{section:set_valued,section:WOT} recall some facts about set-valued maps and the weak operator topology, respectively.

\paragraph{Notations.}
In all of this work, we denote by $B_X(x,r)$ the closed ball in the Banach space $X$, centered at $x \in X$ with radius $r$. Unless otherwise specified, $X$ and $Y$ will always be Banach spaces. Moreover, for a subset $U \subset X$ and a function $F \colon U \to Y$, we recall that
\[
 \Lip_U(F) := \sup_{\substack{x,\, y \in U \\ x \neq y}} \frac{\norm{F(x) - F(y)}{Y}}{\norm{x - y}{X}}.
\]
Finally, for a set $E \subset \RR^n$, we denote by $\co(E)$ the convex hull of $E$.

\section{Metric regularity of single-valued mappings}
\label{section:metric_regularity}

In this section we expose the aspects of the theory of metrically regular mappings that will be useful for our purpose, providing new quantitative versions of some of them. In order to simplify the presentation, we focus on single valued maps between Banach spaces. Note however that one of the strength of this theory is that many aspects still hold for set-valued maps between metric spaces. Our main reference is \cite{I2017} but similar results can be found in \cite{DR2014}.

Let us start by recalling the Banach-Schauder open mapping theorem.
\begin{thm}[Banach-Schauder, {\cite[Theorem 2.6]{B2011}}]\label{thm:banach_schauder}
 Let $T \in \cL(X,Y)$ be a surjective bounded linear operator. Then there exists $\kappa > 0$ such that
 \begin{equation}\label{eq:open_mapping}
  B_Y(0,\kappa) \subset T(B_X(0,1)).
 \end{equation}
\end{thm}

In the context of \cref{thm:banach_schauder}, the supremum of the constants $\kappa > 0$ such that \eqref{eq:open_mapping} holds is called the \emph{Banach constant}\footnote{Here we follow the terminology from \cite{I2017}.} of $T$, which we denote by $\gothC_{X,Y}(T)$, and we simply write $\gothC(T) = \gothC_{X,Y}(T)$ when the context is clear. In the case where $T$ is not surjective, we use the convention that $\gothC(T) = 0$. Since $T(B_X(0,1)) \subset B_Y \left (0,\norm{T}{\cL(X,Y)} \right)$ it is clear that $\gothC(T) \leq \norm{T}{\cL(X,Y)}$. If $T$ is an isomorphism, then $\gothC(T) = \norm{T^{-1}}{\cL(Y,X)}^{-1}$.
Moreover, we have the following useful characterization for the Banach constant (see \cite[Proposition 1.7]{I2017})
\begin{equation}\label{eq:charac_banach_const}
 \gothC(T) = \inf_{\substack{h \in Y' \\ \norm{h}{Y'} = 1}} \norm{T^\star h}{X'},
\end{equation}
where $T^\star$ is the adjoint of $T$.

\begin{cor}[{\cite[Theorem 1.14]{I2017}}]\label{prop:perturbation_banach_schauder}
 Let $T,\, S \in \cL(X,Y)$. Assume that $T$ is surjective with Banach constant $\gothC(T)$ and that $\norm{S}{\cL(X,Y)} < \gothC(T)$. Then $T + S$ is surjective with Banach constant $\gothC(T + S) \geq \gothC(T) - \norm{S}{\cL(X,Y)}$.
\end{cor}
\begin{proof}
 Let $h \in Y'$ with $\norm{h}{Y'} = 1$. Then
 \begin{align*}
  \norm{(T+S)^\star h}{X'} & = \sup_{x \in B_X(0,1)} \frac{1}{\norm{x}{X}} \langle (T+S)^\star h, x \rangle_{X',X}  \\
  & \geq \norm{T^\star h}{X'} - \sup_{x \in B_X(0,1)} \frac{1}{\norm{x}{X}} \langle h, Sx \rangle_{Y',Y} \\
  & \geq \gothC(T) - \norm{S}{\cL(X,Y)}.
 \end{align*}
 Taking the infimum over $h$ and using \cref{eq:charac_banach_const}, we obtain the desired inequality. In particular $\gothC(T + S) > 0$, which implies that $T+S$ is surjective.
\end{proof}

The following fact will be useful when considering applications.
\begin{prop}\label{prop:sufficient_banach_const}
 Let $V$ and $W$ be Banach spaces, let $\cA \subset \cL(V,W)$ be bounded and let $T \in \cL(W,V)$ be compact. Set
 \[
  \cB := \left \{ I + T \circ A :\, A \in \cA \right \}.
 \]
 Assume that each $B \in \cB$ is an isomorphism and that $\cA$ is sequentially compact in the weak operator topology\footnote{See \cref{section:WOT}.}. Then
 \[
  \sup_{B \in \cB} \norm{B^{-1}}{\cL(V)} < + \infty.
 \]
 In particular, there exists $\kappa > 0$ such that $\inf_{B \in \cB} \gothC(B) \geq \kappa$.
\end{prop}

\begin{proof}
 We argue by contradiction. If the conclusion is false, there exists a sequence $(A_n,u_n,v_n) \in \cA \times V \times V$ such that $u_n = \left(I + T \circ A_n \right)^{-1} v_n$, $\norm{v_n}{V} \leq 1$ and $\norm{u_n}{V} \geq n$. Setting $\hat u_n := u_n / \norm{u_n}{V}$ and $\hat v_n := v_n / \norm{u_n}{V}$, we have
 \[
  \hat u_n = \hat v_n - T \circ A_n \hat u_n.
 \]
 Since $\cA$ is bounded, the sequence $A_n \hat u_n$ is bounded in $W$. It then follows from the compactness of $T$ that $T \circ A_n \hat u_n \xrightarrow{n \to \infty} -v$, up to a subsequence, for some $v \in V$. Moreover, since $\norm{\hat v_n}{V} \leq 1/n$, we deduce that $\hat u_n \xrightarrow{n \to \infty} v$ in $V$. In particular $\norm{v}{V} = 1$. From the sequential compactness of $\cA$ in the weak operator topology, there exists $A \in \cA$ such that $A_n \WOTarrow A$ up to another subsequence. Then, for any $h \in B_{V'}(0,1)$, we have
 \begin{align*}
  \langle h, \left(I + T \circ A_n \right)\hat u_n \rangle_{V',V} &= \langle h, \hat u_n \rangle_{V',V} + \langle h, T \circ A_n v \rangle_{V',V} +  \langle h, T \circ A_n (\hat u_n - v) \rangle_{V',V}
 \end{align*}
 From \cref{prop:WOT_composition}, we have $\langle h, T \circ A_n v \rangle_{V',V} \xrightarrow{n \to \infty} \langle h, T \circ A v \rangle_{V',V}$.
Since, in addition,
\[
 \module{\langle h, T \circ A_n (\hat u_n - v) \rangle_{V',V}} \leq \norm{T}{\cL(W,V)} \sup_{A' \in \cA} \norm{A'}{\cL(V,W)} \norm{v - \hat u_n}{V} \xrightarrow{n \to \infty} 0,
\]
we obtain that $\langle h, \left(I + T \circ A \right) v \rangle_{V',V} = 0$.
Since $h$ is arbitrary, we conclude that $v \in \ker_{V} \left(I + T \circ A \right)$. The injectivity of $I + T \circ A$ implies that $v =0$, which contradicts the fact that $\norm{v}{V} =1$. This concludes the proof.
\end{proof}

The notion of metrically regular mappings, which we recall below, extends these ideas to nonlinear mappings.
\begin{defi}[Metric regularity]\label{def:metric_regularity}
 Let $\bar x \in X$. A continuous mapping $F \colon X \to Y$ is \emph{metrically regular}\footnote{Actually, \eqref{eq:linearly_open} defines \emph{openness at linear rate}, a notion that is equivalent to metric regularity. Since metric regularity seems to be a more commonly used terminology, we decided to stick with it. We refer to \cite{I2017,DR2014} for the usual definition of metric regularity.} near $\bar x \in X$ if there exists $\kappa > 0$ and $R > 0$ such that
  \begin{equation}\label{eq:linearly_open}
  B_Y(F(x),\kappa r) \cap B_Y(F(\bar x),R) \subset F \left (B_X(x,r) \right ) \quad \textnormal{for every } 0 < r \leq R \textnormal{ and  } x \in B_X(\bar x, R)
  \end{equation}
  and we define the
  \[
   \sur_{X,Y} (F; \bar x) = \sup \left \{ \kappa > 0 :\textnormal{there exists } R> 0   \textnormal{ such that \eqref{eq:linearly_open} holds} \right \},
  \]
  and we simply write $\sur(F; \bar x)$ for $\sur_{X,Y} (F; \bar x)$.
 Moreover, $F$ is said to be \emph{strongly metrically regular} near $\bar x \in X$ if it is metrically regular near $\bar x$ and there exist neighbourhoods $U$ of $\bar x$ and $V$ of $F(\bar x)$ such that $F^{-1} \cap U$ is single-valued on $V$.
\end{defi}

\begin{prop}[Equivalence of metric regularity and Aubin property, {\cite[Proposition 2.10]{I2017}}]\label{prop:aubin_property}
 Let $F \colon X \to Y$ be continuous. Then $F$ is metrically regular near $\bar x \in X$ if and only if $F^{-1}$ has the Aubin property near $(F(\bar x), \bar x)$, that is, there exist $R, \, K > 0$ such that
 \begin{equation}\label{eq:Aubin_property}
  d(x',F^{-1}(y)) \leq K \norm{y' - y}  \quad \textnormal{whenever } y \in B_Y(F(\bar x),R), \, x' \in B_X(\bar x,R), \, \textnormal{and } x' \in F^{-1}(y').
 \end{equation}
 Moreover, we have $\inf \left \{ K > 0 :\, \textnormal{there exists $R > 0$ such that \eqref{eq:Aubin_property} holds} \right \} = \sur_{X,Y}(F;\bar x)^{-1}$.
\end{prop}

We now recall the theorem of Graves, first established in \cite{G1950}, which provides a sufficient condition for metric regularity.
\begin{thm}[Graves, {\cite[Theorem 1.12 and Corollary 1.13]{I2017}}]\label{thm:graves}
 Let $F$ be a continuous mapping from a neighborhood of $\bar x$ in $X$ into $Y$. Suppose that there exists a surjective linear operator $A \in \cL(X,Y)$ and constants $\delta,\, \kappa,\, R >0$ such that $\gothC(A) > \delta + \kappa$ and
 \[
  \norm{F(x) - F(y) - A(x-y)}{Y} \leq \delta \norm{x-y}{X} \quad \textnormal{for all } x,y \in B_X(\bar x,R).
 \]
 Then
 \[
  B_{Y}(F(\bar x),\kappa r) \subset F(B_X(\bar x,r)) \quad \textnormal{for every } 0 < r \leq R
 \]
 and
 \[
  B_{Y}(F(x),\kappa r) \subset F(B_X(x,r)) \quad \textnormal{for every } x \in B_X(\bar x, R/2) \textnormal{ and } 0 < r \leq R/2.
 \]
 In particular, $F$ is metrically regular near $\bar x$ with $\sur(F;\bar x) \geq \kappa$.
\end{thm}

The following key result, originally established in \cite{DMO1980}, is a nonlinear extension of \cref{prop:perturbation_banach_schauder}.
\begin{thm}[{Milyutin, \cite[Theorems  2.79]{I2017}}]\label{thm:milyutin}
 Let $F \colon X \to Y$ be continuous and metrically regular near $\bar x$, i.e., there exist $\kappa > 0$ and $R >0$ such that
 \[
  B_Y(F(x), \kappa r) \cap B_Y(F(\bar x), R) \subset F(B_X(x,r)) \quad \textnormal{for all }0 < r \leq R, \, x \in B(\bar x, R).
 \]
 Let also $g \colon X \to Y$ be Lipschitz continuous on a neighborhood $U$ of $\bar x \in X$ with $\Lip_U(g) < \kappa$. Then, for every $\mu \in [\Lip_{U}(g), \kappa )$, we have
 \begin{equation}\label{eq:milyutin_conclusion}
  B_Y((F+g)(x),(\kappa - \mu)r) \cap B_Y((F+g)(\bar x), R/2) \subset (F+g)(B_X(x,r))
 \end{equation}
for all $0 < r \leq R/2$ and $x \in  B(\bar x, R/2)$. In other words, if $F$ is metrically regular near $\bar x$ with $\sur(F;\bar x) \geq \kappa$, then $F+g$ is also metrically regular near $\bar x$ with $\sur(F+g;\bar x) \geq \kappa - \mu$.
\end{thm}

In the case where we also assume that $F$ is strongly metrically regular in \cref{thm:milyutin}, it is proved in \cite[Theorem 2.87]{I2017} that the perturbed mapping $F + g$ is also strongly metrically regular. We provide a quantitative version of this fact in the next result.

\begin{thm}[Perturbation of strong metric regularity]\label{thm:milyutin_strong}
  Under the assumptions of \cref{thm:milyutin}, if we also assume that $F$ is strongly metrically regular near $\bar x$, then $F+g$ is also strongly metrically regular near $\bar x$. More precisely, if $\epsilon > 0$ and $0 < \delta \leq \min \{ \kappa R/2, R/2 \}$ are such that $F^{-1} \cap B_X(\bar x, \epsilon)$ is single-valued on $B_Y(F(\bar x), \delta)$, then, $(F +g)^{-1} \cap B_X(\bar x, \eta)$ is single-valued on $B_Y(F(\bar x) + g(\bar x), (1 - \kappa^{-1} \mu) \delta)$, where $\eta := \min \{ \epsilon, \kappa^{-1} \delta \}$.
\end{thm}

\begin{proof}
 Let $\delta$, $\epsilon$ and $\eta$ be as in the statement and set $Q:= F^{-1} \cap B(\bar x, \eta)$ which is single-valued since $\eta \leq \epsilon$. Using \eqref{eq:milyutin_conclusion} with $r = \kappa^{-1}\delta$, we see that
 \[
  B_Y(F(\bar x) + g(\bar x), (1 - \kappa^{-1} \mu) \delta) \subset \dom (F + g)^{-1}.
 \]
 For any $x_1,\, x_2 \in B_X(\bar x, \eta)$ such that
 \[
  F(x_1) + g(x_1) = F(x_2) + g(x_2) =: z \in B_Y \left ( F(\bar x) + g(\bar x), (1 - \kappa^{-1} \mu)\delta \right ),
 \]
 we have to prove that actually $x_1 = x_2$.

 Notice that
 \begin{align*}
  \norm{F(x_i) - F(\bar x)}{Y} & \leq \norm{F(x_i) + g(x_i) - F(\bar x) - g(\bar x)}{Y} + \norm{g(x_i) - g(\bar x)}{Y} \\
  & \leq (1 - \kappa^{-1} \mu)\delta + \mu \norm{x_i - \bar x}{X} \\
  & \leq (1 - \kappa^{-1} \mu)\delta + \kappa^{-1} \mu \delta \\
  & = \delta,
 \end{align*}
 so that $F(x_i) \in B_Y \left (F(\bar x), \delta \right)$ and $Q(F(x_i))$ is well-defined for $i=1,2$.
 Moreover, from the metric regularity of $F$ and \cref{prop:aubin_property}, we have
 \begin{equation*}\label{eq:milyutin_inverse}
  \norm{Q(y) - Q(y')}{X} \leq \kappa^{-1} \norm{y - y'}{Y} \quad \textnormal{for all } y,\, y' \in B_Y\left (F(\bar x), \delta \right ).
 \end{equation*}
 Then, we have
 \begin{align*}
  \norm{x_1 - x_2}{X}  & = \norm{Q \circ F(x_1) - Q \circ F(x_2)}{X} \leq \kappa^{-1} \norm{F(x_1) - F(x_2)}{Y} \\
  & = \kappa^{-1} \norm{\left (z - g(x_1) \right) - \left (z - g(x_2) \right )}{Y} = \kappa^{-1} \norm{g(x_2) - g(x_1)}{Y} \\
  & \leq \mu \kappa^{-1} \norm{x_1 - x_2}{X}.
 \end{align*}
 Since $\mu \kappa^{-1} < 1$, we conclude that $x_1 = x_2$. This proves that $(F + g)^{-1} \cap B_X(\bar x, \eta)$ is single valued on $B_Y(F(\bar x) + g(\bar x), (1 - \mu \kappa^{-1}) \delta)$.
\end{proof}

The following theorem is an extension of \cref{thm:graves}.
\begin{thm}[{\cite[Theorem 3.8]{CFI2015}}]\label{thm:graves_variation}
 Let $X$ and $Y$ be reflexive Banach spaces and let $ \bar x \in X$. Let further $F \colon X \to Y$ be continuous and let $\cA$ be a bounded convex subset of $\cL(X,Y)$ such that
 \begin{enumerate}[label=(\roman*)]
  \item there exist $\mu > 0$ and $r >0$ such that, for any $x,x' \in B_X(\bar x,r)$, there exists $A \in \cA$ such that
  \begin{equation}\label{eq:graves_variation_diff}
   \norm{F(x') - F(x) - A(x' - x)}{Y} < \mu \norm{x' - x}{X},
  \end{equation}

  \item there exists $\kappa > \mu$ such that $\mathfrak{C}(A) \geq \kappa$ for all $A \in \cA$.
 \end{enumerate}
  Then $F$ is metrically regular near $\bar x$ with $\sur(F; \bar x) \geq \kappa - \mu$.
\end{thm}

\begin{rem}
 The condition \eqref{eq:graves_variation_diff} is a generalization of strict differentiability and was called \emph{strict Fréchet pre-differentiability} in \cite{I1981}. We recall that a mapping $F \colon X \to Y$ is strictly differentiable at $\bar x \in X$ if it is Fréchet differentiable at $\bar x$ and if, for every $\mu > 0$, there exists $\delta > 0$ such that
 \[
  \norm{F(x') - F(x) - dF[\bar x](x' - x)}{Y} < \mu \norm{x' - x}{X} \quad \textnormal{for all } x, \, x' \in B_X(\bar x, \delta).
 \]
\end{rem}

We now extend \cref{thm:graves_variation} to obtain a criteria for strong metric regularity.
\begin{cor}\label{cor:strong_metric_regularity}
 Let $X$ and $Y$ be reflexive Banach spaces and let $ \bar x \in X$. Let further $F \colon X \to Y$ be continuous and let $\cA$ be a bounded convex subset of $\cL(X,Y)$ such that
 \begin{enumerate}[label=(\roman*)]
  \item for every $\epsilon > 0$ there is $\delta >0$ such that, for any $x,x' \in B_X(\bar x,\delta)$, there exists $A \in \cA$ such that
  \begin{equation}\label{eq:strong_metric_diff}
   \norm{F(x') - F(x) - A(x' - x)}{Y} < \epsilon \norm{x' - x}{X},
  \end{equation}

  \item each $A \in \cA$ is an isomorphism and there is $\kappa>0$ such that $\norm{A^{-1}}{\cL(Y,X)} \leq \kappa^{-1}$.
  \end{enumerate}
  Then $F$ is strongly metrically regular near $\bar x$.
\end{cor}

\begin{proof}
  The fact that $F$ is metrically regular near $\bar x$ follows from \cref{thm:graves_variation}. In order to establish strong metric regularity, we have to prove that there exist $\delta,\, \eta >0$ such that $F^{-1} \cap B_{Y}(\bar x, \eta)$ is single-valued in $B_Y \left(F(\bar x), \delta \right)$. Assume that it is not the case. Let $\delta > 0$ to be specified later. Then, for every $k$ large enough, there exist $x_k, x_k' \in B_X(\bar x, 2^{-k})$ such that $F(x_k) = F(x_k') \in B_Y \left ( F(\bar x), \delta \right)$ and $x_k \neq x_k'$. Let $\epsilon < \kappa$ and $\delta > 0$ be such that \eqref{eq:strong_metric_diff} holds. Let $k$ be such that $2^{-k} < \delta$. From \eqref{eq:strong_metric_diff}, there is $A \in \cA$ such that
 \begin{align*}
   \norm{A(x_{k} - x_{k}')}{Y} & < \epsilon \norm{x_{k} - x_{k}'}{X} \leq \epsilon \kappa^{-1} \norm{A(x_{k} - x_{k}')}{Y},
 \end{align*}
 which is a contradiction since $\epsilon \kappa^{-1} < 1$.
\end{proof}

\section{Generalized Brezzi-Rappaz-Raviart approximation theorem }
\label{section:main_BRR}
This section contains the main results of the paper. We first provide the main building block of our extensions of the BRR theorem. It contains sufficient conditions for a collection of mapping $F_h \colon X_h \to Y_h$, defined on Banach spaces $X_h \subset X$ for $h>0$, to have zeros close to some reference point $\bar x \in X$. It is the natural generalization of \cite[Theorem 6.1]{CR1997} in the context of metric regularity. We then apply this result, first to revisit the standard forms for the BRR theorem, which can be found in \cite{GR1986,CR1997}, and then to prove extensions of the latter to a large class of nonsmooth problems.

\begin{prop}\label{thm:BRR_most_general}
 Let $X$ and $Y$ be normed vector spaces. For all $h > 0$, let $X_h \subset X$ and $Y_h \subset Y$ be subspaces endowed with the induced norm, let $F_h \colon X_h \to Y_h$ be continuous, and let $\bar x \in X$. Assume that
 \begin{enumerate}[label={\rm (\roman*)}]
  \item there exists a sequence $\hat x_h \in X_h$ such that
  \begin{equation}\label{eq:BRR_most_general_1}
   \lim_{h \to 0} \norm{\bar x - \hat x_h}{X} + \norm{F_h(\hat x_h)}{Y} = 0;
  \end{equation}

  \item there exist constants $h_0, \kappa, R  > 0$ such that, for every $0 < h \leq h_0$, we have
  \begin{equation}\label{eq:BRR_most_general_2}
   B_{Y_h} \left( F_h(\hat x_h), \kappa r \right) \subset F_h \left ( B_{X_h} \left( \hat x_h, r \right) \right) \quad \textnormal{for all } 0 < r \leq R.
  \end{equation}
 \end{enumerate}

 Then there exists $0 < h_1 \leq h_0$ such that, for all $0 < h \leq h_1$, there exists $\bar x_h \in X_h$ satisfying $F_h(\bar x_h) = 0$, and we have the estimate
 \[
  \norm{\bar x - \bar x_h}{X} \leq \left ( 1 + \kappa^{-1} \right) \norm{\bar x - \hat x_h}{X} + \kappa^{-1} \norm{F_h(\hat x_h)}{Y}.
 \]
\end{prop}

\begin{proof}
 From \eqref{eq:BRR_most_general_1}, we may choose $0 < h_1 \leq h_0$ such that
 \[
  \bar r_h := \kappa^{-1} \left(\norm{\bar x - \hat x_h}{X} + \norm{F_h(\hat x_h)}{Y}\right) \leq R \quad \textnormal{for all } 0 < h \leq h_1.
 \]
 Then, from \eqref{eq:BRR_most_general_2}, we have $0 \in B_{Y_h}\left( F_h(\hat x_h), \kappa \bar r_h \right) \subset F_h \left (B_{X_h}(\hat x_h, \bar r_h) \right )$,
 so that there exists $\bar x_h \in B_{X_h}(\hat x_h, \bar r_h)$ such that $F_h(\bar x_h) = 0$. It follows that
 \[
  \norm{\bar x_h - \bar x}{X} \leq \norm{\bar x_h - \hat x_h}{X} + \norm{\hat x_h - \bar x}{X} \leq  \left ( 1 + \kappa^{-1} \right) \norm{\bar x - \hat x_h}{X} + \kappa^{-1} \norm{F_h(\hat x_h)}{Y}. \qedhere
 \]
\end{proof}

\subsection{Standard case revisited}
\label{section:BRR}

The following theorem is the standard BRR theorem in the form \cite[Theorem 3.6 p.312]{GR1986}, where we slightly modified the statement (see \cref{rem:reworked_BRR} below). While the usual proof relies on Banach's fixed point theorem, we provide a proof using metric regularity. Following this idea will lead us to the generalization in \cref{thm:approximation_main} below.

\begin{thm}\label{thm:reworked_BRR}
 Let $X$ and $Y$ be Banach spaces, let $\bar x \in X$ and $A \in \cL(X,Y)$ be surjective. In addition, for every $h > 0$, let $(X_h,\norm{\cdot}{X})$ and $(Y_h,\norm{\cdot}{Y})$ be closed subspaces of $X$ and $Y$, respectively, let $F_h \colon X_h \to Y_h$ be continuous and let $\tilde x_h \in X_h$ be given. We assume that
 \begin{enumerate}[label={\rm(\roman*)}]
  \item we have
  \begin{equation}\label{eq:epsilon_h}
  \epsilon_h := \norm{F_h(\tilde x_h)}{Y} + \norm{\bar x - \tilde x_h}{X} \xrightarrow{h \to 0} 0;
 \end{equation}

 \item \label{item:reworked_BRR_simple_ii} there exist linear operators $A_h \in \cL(X_h,Y_h)$ such that, for some $R >0$ and a function $c \colon \RR^*_+ \times \RR^*_+ \to \RR_+$, increasing with respect to both variables with $ \lim_{(h,r) \to 0} c(h,r) = 0$,  we have
 \begin{equation}\label{eq:F_h_diff}
  \norm{F_h(x) - F_h(y) - A_h(x-y)}{Y} \leq c(h,r) \norm{x-y}{X}
 \end{equation}
 for all $x,y \in B_{X_h}(\tilde x_h, r)$ for every $0 < r < R$ and $h > 0$;

 \item \label{item:reworked_BRR_simple_iii} there exist bounded linear operator $\widehat{A}_h \in \cL(X,Y)$, such that the range of the restriction of $\widehat{A}_h$ to $X_h$ is contained in $Y_h$, and satisfying
  \begin{equation}\label{eq:intermediate_3}
  \widehat{A}_h^{-1}(Y_h) \subset X_h,
 \end{equation}
 \end{enumerate}
 \begin{equation}\label{eq:intermediate_1}
  \lim_{h \to 0} \norm{A - \widehat{A}_h}{\cL(X,Y)} = 0
 \end{equation}
 and
 \begin{equation}\label{eq:intermediate_2}
  \lim_{h \to 0} \norm{A_h - \widehat{A}_h}{\cL(X_h,Y_h)} = 0.
 \end{equation}
  Then, for every $0 < h \leq h_0$, with $h_0$ small enough, there exists $\bar x_h \in X$, such that
  \begin{equation}\label{eq:reworked_BRR_sol}
   F_h(\bar x_h) = 0,
  \end{equation}
  we have the estimates
 \begin{equation}\label{eq:sol_F_h_estimate}
   \norm{\bar x - \bar x_h}{X} \leq \left (1 +  \frac{2}{\gothC_{X,Y}(A)} \right ) \norm{\bar x - \tilde x_h}{X} +\frac{2}{\gothC_{X,Y}(A)} \norm{F_h(\tilde x_h)}{Y}
 \end{equation}
 and $\sur_{X_h,Y_h}(F_h;\bar x_h) \geq \frac{1}{2}\gothC_{X,Y}(A)$. Moreover, if $A$ is an isomorphism, there exists $0 < h_1 \leq h_0$ such that $F_h$ is strongly metrically regular near $\bar x_h$. In particular there is a unique $\bar x_h \in B_{X}(\bar x, r_1) \cap X_h$ satisfying \eqref{eq:reworked_BRR_sol} for some $r_1 >0$ and every $0 < h \leq h_1$.
\end{thm}
\begin{rem}\label{rem:reworked_BRR}
\,
\begin{enumerate}
\item \label{item:rem_reworked_BRR_simple} Usually, the mappings $F_h$ are approximations of some mapping $F \colon X \to Y$ such that $F(\bar x) = 0$. In many cases, \cref{thm:reworked_BRR} is then applied with $X_h = X$, $Y_h = Y$, $\tilde x_h = \bar x$ and $\widehat{A}_h = A_h$. When both $F$ and $F_h$ are differentiable, natural candidates for $A$ and $A_h$ are given by the differentials $dF[\bar x]$ and $dF_h[\bar x]$. This setting corresponds to the simple case described in the introduction. However, in general, it is not always clear how one can identify the operator $A$ and $A_h$.

\item Actually, in \cite{BRR1980}, the authors were considering parametric equations of the form
\begin{equation}\label{eq:parametric}
 F(\sigma, x) = 0 \quad \textnormal{for all } \sigma \in \Sigma,
\end{equation}
for some compact subset $\Sigma \subset \RR$. Hence, they studied nonsingular branches of solutions $(\bar x_\sigma)_{\sigma \in \Sigma}$, that is, such that $dF[\sigma,\bar x_\sigma]$ is an isomorphism for every $\sigma \in \Sigma$. Their investigation was then extended to problems with bifurcations by Brezzi, Rappaz and Raviart \cite{BRR1981,BRR1982}, Rappaz and Raugel \cite{RR1982} and Descloux and Rappaz \cite{DR1982}. A review of the results about bifurcations can also be found in \cite{CR1990,CR1997}.

 \item Actually, as one can check in the proof below, it would be enough to assume (i) and (ii) in \cref{thm:reworked_BRR} and that
 \begin{equation}\label{eq:unif_sujectivity}
  \inf_{0 < h < h_1} \gothC(A_h) \geq \kappa > 0
 \end{equation}
 for some $h_1 > 0$ to obtain the first conclusion. This would make the statement closer to the original \cite[Theorem 3.6 p.312]{GR1986}. The operators $A$ and $\widehat{A}_h$ are used as tools to verify \eqref{eq:unif_sujectivity}.

 \item We provide an alternative proof of \cref{thm:reworked_BRR} at the end of \cref{section:generalized_BRR}.
\end{enumerate}
\end{rem}

\begin{lem}\label{prop:A_h_surjective}
  Under the assumptions of \cref{thm:reworked_BRR}, for every $0 < \epsilon < \gothC_{X,Y}(A)$ there exists $h_0 > 0$ such that for every $0 < h \leq h_0$ the operator $A_h$ is surjective and has Banach constant
  \[
   \gothC_{X_h,Y_h}(A_h) \geq \gothC_{X,Y}(A) - \epsilon.
  \]
\end{lem}

\begin{proof}
Before entering into the details of the proof, let us briefly describe the main idea. Since $A$ is surjective and $\widehat{A}_h$ is an approximation of $A$, we will apply \cref{prop:perturbation_banach_schauder} to deduce that $\widehat{A}_h$ is also surjective. But since $A_h$ approximates $\widehat{A}_h$, we will similarly obtain that $A_h$ is surjective.

Let $0 < \epsilon < \mathfrak{C}_{X,Y}(A)$. Set, for every $h > 0$,
\[
 \delta_h := \norm{A - \widehat{A}_h}{\cL(X,Y)} \quad \textnormal{and} \quad \eta_h := \norm{A_h - \widehat{A}_h}{\cL(X_h,Y_h)}.
\]
 From \eqref{eq:intermediate_1} we may choose $h_0 > 0$ such that $\delta_h < \epsilon/2$ for every $0 < h \leq h_0$. Since $\widehat{A}_h = A + (\widehat{A}_h - A)$ and $A$ is surjective, it follows from \cref{prop:perturbation_banach_schauder} that $\gothC_{X,Y}(\widehat{A}_h) \geq \gothC_{X,Y}(A) - \delta_h$, i.e., for every $\kappa < \gothC_{X,Y}(A) - \delta_h$, the following inclusion holds
 \begin{equation}\label{eq:A_h_surjective_3}
  B_Y(0,\kappa) \subset \widehat{A}_h (B_X(0,1)).
 \end{equation}
 In particular $\widehat{A}_h$ is surjective. Moreover \eqref{eq:A_h_surjective_3} and \eqref{eq:intermediate_3} imply that, for every $\kappa < \gothC_{X,Y}(A) - \epsilon/2$,
 \begin{align*}
  B_{Y_h}(0,\kappa) = B_Y(0,\kappa) \cap Y_h \subset \widehat{A}_h(B_X(0,1)) \cap Y_h \subset \widehat{A}_h(B_{X_h}(0,1)).
 \end{align*}
 so that we also have $\gothC_{X_h,Y_h}(\widehat{A}_h)\geq \gothC_{X,Y}(A) - \epsilon/2$. We can then choose $0 < h_1 \leq h_0$ so that $\eta_h < \epsilon/2 < \mathfrak C_{X,Y}(A) - \epsilon/2$ for every $0 < h \leq h_1$. We deduce from \eqref{eq:intermediate_2} and \cref{prop:perturbation_banach_schauder} that
 \[
  \gothC_{X_h,Y_h}(A_h) \geq \gothC_{X_h,Y_h}(\widehat{A}_h) - \epsilon/2 \geq \gothC_{X,Y}(A) - \epsilon,
 \]
 i.e., we have that for every $\kappa < \gothC_{X,Y}(A) - \epsilon$ we have $ B_{Y_h}(0,\kappa) \subset A_h(B_{X_h}(0,1))$ and $A_h$ is surjective.
\end{proof}

\begin{lem}\label{prop:F_h_metric_regular}
 Let $\kappa < \gothC_{X,Y}(A)$. Under the assumption of \cref{thm:reworked_BRR} there exists $h_0 >0$ and $0 < R_0 \leq R$ such that the following inclusion holds for every $0 < h \leq h_0$,
 \[
  B_{Y_h}(F_h(x_h),\kappa r) \subset F_h(B_{X_h}(x_h, r)) \quad \textnormal{for every } x_h \in B_{X_h}(\tilde x_h,r) \textnormal{ and } 0 < r < R_0.
  \]
  In particular $\sur_{X_h,Y_h}(F_h;\tilde x_h) \geq \kappa$.
\end{lem}

\begin{proof}
 Let $\eta \in (0,1)$ and let $h_0>0$ be such that the conclusion of \cref{prop:A_h_surjective} holds for $\epsilon = (1 - \eta)\gothC_{X,Y}(A)$. Notice that according to \cref{prop:A_h_surjective} we have $\gothC_{X_h,Y_h}(A_h) \geq \eta \gothC_{X,Y}(A) =: K$ for every $0 < h \leq h_0$. Then, since $ c$ is increasing with respect to both $h$ and $r$ and that $\lim_{(h,r )\to 0} c(h,r) = 0$,  for every $\kappa < K$, we can, up to taking a smaller $h_0$, choose $R_0 >0$ such that $\kappa + c(h,r) < K$ for every $0 <h \leq h_0$ and $0 < r < R_0$. Using \eqref{eq:F_h_diff}, we can apply Graves' \cref{thm:graves} to the function $F_h$ and the linear operator $A_h$ to deduce that
 \[
 B_{Y_h}(F_h(x_h),\kappa r) \subset F_h(B_{X_h}(x_h,r)) \quad \textnormal{for every } x_h \in B_{X_h}(\tilde x_h,R_0/2) \textnormal{ and } 0 < r <R_0/2. \qedhere
 \]
\end{proof}

\begin{proof}[Proof of \cref{thm:reworked_BRR}]
   Let $h_0$ and $R_0$ be such that the conclusion of \cref{prop:F_h_metric_regular} for  $\kappa = \mathfrak C_{X,Y} (A)/2$. The existence of $\bar x_h$ and the estimate \eqref{eq:sol_F_h_estimate} follow from \cref{prop:F_h_metric_regular,thm:BRR_most_general}. Moreover, for $h_0$ small enough, we have $\bar x_h \in B_{X_h}(\tilde x_h,R_0/2)$, so that $B_{X_h}(\bar x_h, R_0/2) \subset B_{X_h}(\tilde x_h, R_0)$ and $B_{Y_h}(F_h(x_h),\kappa r) \subset F_h(B_{X_h}(x_h,r))$ for every $x_h \in B_{X_h}(\bar x_h,R_0/2)$ and $0 < r <R_0/2$.
   In particular, $\sur_{X_h,Y_h}(F_h; \bar x_h) \geq \gothC_{X,Y}(A) /2$.

   For strong metric regularity, notice first that, if $A$ is an isomorphism, we can use \cref{eq:intermediate_1} and \eqref{eq:intermediate_2} to obtain that, for every $0 < h \leq h_0$, each $A_h$ is also an isomorphism with $\norm{A_h^{-1}}{\cL(X_h,Y_h)} \leq K$ for some constant $K > 0$ independent of $h$. Recall that $R_0 \leq R$. Let $0 < r \leq R_0/2$ and $0 < h_1 \leq h_0$ to be fixed later. For $0 < h \leq h_1$, let $x_h,\, y_h \in B_{X_h}(\bar x_h, r) \subset  B_{X_h}(\tilde x_h, R)$ such that $F_h(x_h) = F_h(y_h)$. Using \eqref{eq:F_h_diff} we obtain
   \[
    \norm{A_h(x_h - y_h)}{X_h} \leq c(h,r) \norm{x_h - y_h}{X_h} \leq K c(h_1,r) \norm{A_h(x_h - y_h)}{X_h}.
   \]
   Since $\lim_{(r,h) \to 0} \tilde c(h,r) = 0$, using the injectivity of $A_h$, we deduce that $x_h = y_h$ for $r$ and $h_1$ small enough.
    Using \eqref{eq:sol_F_h_estimate} and \eqref{eq:epsilon_h}, we may choose $h_1$ small enough so that $\bar x \in B_X(\bar x_h,r/2)$ for all $0 < h \leq h_1$. Then $F_h$ is injective on $B_X(\bar x, r/2) \cap X_h \subset B_{X_h}(\bar x_h,r)$ and $\bar x_h \in B_X(\bar x, r/2) \cap X_h$. This proves that $\bar x_h$ is the unique zero of $F_h$ on $B_X(\bar x, r/2) \cap X_h$.
\end{proof}

\subsection{Generalized BRR theorem}
\label{section:generalized_BRR}
In this section we generalize \cref{thm:reworked_BRR} by using tools from the theory of metrically regular mappings. The main feature of this extension is that it may be applied when we do not have the linear operators $A_h$ at hand.

\begin{thm}\label{thm:approximation_main}
 Let $X,Y$ be two Banach spaces and, for every $h > 0$, let $(X_h,\norm{\cdot}{X})$ and $(Y_h,\norm{\cdot}{Y})$ be closed subspaces of $X$ and $Y$, respectively. Let also $F \colon X \to Y$ be continuous and, for each $h >0$, let $F_h \colon X_h \to Y_h$ be continuous. Finally, let and $\bar x \in X$ and $\tilde x_h \in X_h$ be given. Assume that
 \begin{enumerate}[label=(\roman*)]
  \item \label{item:approximation_main_1} $F$ is metric regular at $\bar x \in X$;
  \item \label{item:approximation_main_2} we have
  \begin{equation}
   \lim_{h \to 0} \norm{\bar x - \tilde x_h}{X} + \norm{F_h(\tilde x_h)}{Y} = 0;
  \end{equation}

  \item \label{item:approximation_main_3} there are continuous maps $\widehat F_h \colon X \to Y$, with $\widehat F_h(X_h) \subset Y_h$ and $\widehat F_h^{-1}(Y_h) \subset X_h$, a function $L \colon \RR^*_+ \times \RR_+^* \to \RR_+$, increasing with respect to both variables and such that $\lim_{(h,r) \to 0} L(h,r) = 0$, and $R > 0$, such that the mappings $\Phi_h := \widehat F_h - F$ and $\Psi_h = F_h - \widehat F_h$ satisfy
  \begin{equation}\label{eq:approximation_main_Phi}
   \norm{\Phi_h(x) - \Phi_h(x')}{Y} \leq L(h,r) \norm{x-x'}{X} \quad \textnormal{for every } x,x' \in B_{X}(\tilde x_h, r), \, 0 < r < R,
  \end{equation}
  and
  \begin{equation}\label{eq:approximation_main_Psi}
   \norm{\Psi_h(x) - \Psi_h(x')}{Y} \leq L(h,r) \norm{x-x'}{X} \quad \textnormal{for every } x,x' \in B_{X_h}(\tilde x_h, r),\, 0 < r < R,
  \end{equation}
   for every $h>0$.
 \end{enumerate}
 Then, for some $h_1 > 0$, there exists $\bar x_h \in X_h$ such that
 \begin{equation}\label{eq:approximation_1}
  F_h(\bar x_h) = 0
 \end{equation}
 for all $0 < h < h_1$, and we have the estimates
 \begin{equation}\label{eq:approximation_2}
  \norm{\bar x - \bar x_h}{X} \leq \left (1 + \frac{2}{\sur_{X,Y}(F; \bar x)} \right ) \norm{\bar x - \tilde x_h}{X} + \frac{2}{\sur_{X,Y}(F; \bar x)} \norm{F_h(\tilde x_h)}{Y},
 \end{equation}
 and $\sur_{X_h,Y_h}(F_h,\bar x_h) \geq \frac{1}{2} \sur_{X,Y}(F, \bar x)$.
 Moreover, if $F$ is strongly metrically regular near $\bar x$, then $F_h$ is strongly metrically regular near $\bar x_h$ and there exists a neighborhood $\mathcal{O}$ of $\bar x$ in $X$, independent of $h$, such that $\bar x_h \in \mathcal{O}$ and is the only element in $\mathcal{O}$ satisfying \eqref{eq:approximation_1}.
\end{thm}

\begin{proof}
 We first write, for every $x \in X_h$,
 \[
  \widehat F_h(x) = F(x) + \widehat F_h(x) - F(x) = F(x) + \Phi_h(x).
 \]
 By \ref{item:approximation_main_1}, there exist  $R'>0$ and $0 < \kappa \leq \sur_{X,Y}(F;\bar x)$ such that
 \[
  B_Y(F(x),\kappa r) \cap B_Y(F(\bar x),R') \subset F\left ( B_X(x,r) \right ) \quad \textnormal{for every } x \in B_X(\bar x,r) \textnormal{ and } 0 < r < R'.
 \]
 Without loss of generality, we may assume that $R'= R$. Using \ref{item:approximation_main_2}, we may choose $h_1 > 0$ small enough so that $\tilde x_h \in B_X(\bar x, R/2)$ and $F(\tilde x_h) \in B_Y(F(\bar x),R/2)$ for every $0 < h \leq h_1$.
 This implies that $B_Y(F(x),\kappa r) \cap B_Y(F(\tilde x_h),R/2) \subset F \left ( B_X(x,r) \right )$ for every $x \in B(\tilde x_h,r)$ and $0 < r < R/2$, for all $0 < h \leq h_1$. We may also choose $h_1$ small enough, choosing a smaller $R$ if necessary, so that \ref{item:approximation_main_3} holds for $L(h,r) < \kappa/4$ for all $0 < h \leq h_1$ and $0 < r \leq R$. With this choice of $h_1$ we can apply \cref{thm:milyutin} to deduce that we may choose $h_1$ and $0 < R_1 \leq R/4$ so that $B_Y(\widehat F_h(x),3\kappa r/4) \cap B_Y(\widehat F_h(\tilde x_h), R_1) \subset \widehat F_h \left ( B_X(x,r) \right )$ for all $x \in B_X(\tilde x_h,R_1)$, $0 < r < R_1$ and $0 < h \leq h_1$. Since
 \begin{align*}
  B_{Y_h}(\widehat F_h(x),3\kappa r/4) \cap B_{Y_h}(\widehat F_h(\tilde x_h), R_1) & = B_{Y}(\widehat F_h(x),3 \kappa r/4) \cap B_Y(\widehat F_h(\tilde x_h), R_1) \cap Y_h \\
  & \subset  \widehat F_h \left ( B_X(x,r) \right ) \cap Y_h
 \end{align*}
 for every $x \in B_X(\tilde x_h,R_1) \cap X_h$ and $0 < r < R_1$ and $\widehat F_h^{-1}(Y_h) \subset X_h$, we deduce that
 \begin{equation*}
  B_{Y_h}(\widehat F_h(x),3\kappa r/4) \cap B_{Y_h}(\widehat F_h(\tilde x_h), R_1) \subset \widehat F_h \left ( B_{X_h}(x,r) \right )
 \end{equation*}
 for every $x \in B_{X_h}(\tilde x_h,R_1)$ and $ 0 < r < R_1$.
 In other words, $\sur_{X_h,Y_h}(\hat F_{h}; \tilde x_h) \geq \frac{3\kappa}{4}$. We can then write $F_h = \widehat F_h + F_h - \widehat F_h = \widehat F_h + \Psi_h$. Another application of \cref{thm:milyutin} yields
 \begin{equation}\label{eq:approximation_main_4}
  B_{Y_h}( F_h(x_h), \kappa r /2) \cap B_{Y_h}( F_h(\tilde x_h), R_2) \subset  F_h(B_{X_h}(x,r))
 \end{equation}
 for every $x_h \in B_{X_h}(\tilde x_h, R_2)$ and $0 < r < R_2 < R/8$. We may then choose $R_3 = 2R_2/\kappa$ so that $B_{Y_h}( F_h(\tilde x_h), \kappa R_3 /2) \subset B_{Y_h}( F_h(\tilde x_h), R_2)$ and $B_{Y_h}( F_h(\tilde x_h), \kappa r /2) \subset  F_h(B_{X_h}(\tilde x_h,r))$ for all $0 < r \leq \min\{R_2, R_3 \}$.
  The existence of $\bar x_h$ satisfying \eqref{eq:approximation_1} and estimate \eqref{eq:approximation_2} follow from \cref{thm:BRR_most_general}.

  In order to prove that $F_h$ is metrically regular at $\bar x_h$ for all $h > 0$ sufficiently small, notice first, using \cref{thm:BRR_most_general}, that
  \[
   \norm{\bar x_h - \tilde x_h}{X} \leq \norm{\bar x_h - \bar x}{X} + \norm{\bar x - \tilde x_h}{X} \leq C \left( \norm{\bar x - \tilde x_h}{X} + \norm{F_h(\tilde x_h)}{Y} \right) \xrightarrow{h \to 0} 0.
  \]
  Setting $R_4 := \min \{R_2, \,R_3 \}$, it follows that $\bar x_h \in B_{X_h}\left (\tilde x_h, R_4/2 \right)$ and  $B_{X_h}\left(\bar x_h, R_4/2 \right) \subset B_{X_h} \left( \tilde x_h, R_4 \right)$ for all $h > 0$ small enough. We then deduce from \eqref{eq:approximation_main_4} that
  \begin{equation}\label{eq:approximation_main_5}
    B_{Y_h}( F_h(x_h), \kappa r /2) \cap B_{Y_h}( F_h(\tilde x_h), R_4) \subset  F_h(B_{X_h}(x_h,r)) \quad \textnormal{for every $x_h \in B_{X_h}(\bar x_h, R_4/2)$.}
  \end{equation}
  Since $F_h(\tilde x_h) \xrightarrow{h \to 0} 0 = F_h(\bar x_h)$, we also have that $F_h(\bar x_h) \in B_{Y_h}\left(F_h(\tilde x_h), R_4/2 \right)$ for all $h > 0$ sufficiently small, so that $B_{Y_h} \left( F_h(\bar x_h), R_4/2 \right) \subset B_{Y_h}\left(F_h( \tilde x_h), R_4 \right)$. It then follows from \eqref{eq:approximation_main_5} that
  \begin{equation*}
    B_{Y_h}( F_h(x_h), \kappa r /2) \cap B_{Y_h}( F_h(\bar x_h), R_4/2) \subset  F_h(B_{X_h}(x_h,r)) \quad \textnormal{for every $x_h \in B_{X_h}(\bar x_h, R_4/2)$}
  \end{equation*}
  for all $h > 0$ small enough. From \cref{def:metric_regularity}, this proves that $F_h$ is metrically regular at $\bar x_h$ for all $h > 0$ sufficiently small with $\sur_{X_h,Y_h}(F_h,\bar x_h) \geq \frac{1}{2} \sur_{X,Y}(F, \bar x)$.
 As for local uniqueness, an application of \cref{thm:milyutin_strong} shows that $F_h$ is also strongly metrically regular whenever $F$ is.
\end{proof}

We now show how \cref{thm:reworked_BRR} may be deduced from \cref{thm:approximation_main}.
\begin{proof}[Alternative proof of \cref{thm:reworked_BRR}]
  Define $\bar F \colon X \to Y$ by $\bar F(x) = A(x - \bar x)$ so that $\bar F (\bar x) = F(\bar x) = 0$ and $\sur_{X,Y}(\bar F; \bar x) = \gothC_{X,Y}(A) > 0$.  Similarly, define for every $h > 0$ the mapping $\widehat F_h \colon X \to Y$ by $\widehat F_h(x) = \widehat A_h(x - \bar x)$. It is then easy to verify that assumptions \ref{item:reworked_BRR_simple_ii} and \ref{item:reworked_BRR_simple_iii} in \cref{thm:reworked_BRR} imply assumption \ref{item:approximation_main_3} in \cref{thm:approximation_main}. The result then directly follows from \cref{thm:approximation_main}.
\end{proof}

\subsection{The case of weakly nonlinear problems}

The following corollary provides sufficient conditions to apply \cref{thm:approximation_main} that are simple to check. It corresponds to \cite[Theorem IV.3.3]{GR1986} under more general assumptions. This result can, for instance, be applied to the finite element approximation of stationary Navier-Stokes equations \cite{BRR1980,GR1986}, Von K\'arm\'an equations \cite{BRR1980}, semilinear elliptic equations \cite{CR1990,CR1997,KP2010}, semilinear parabolic equations \cite{CH2002} or second order stationary mean field games \cite{BLS2025}.
\begin{cor}\label{cor:first_example}
 Let $V$ and $W$ be Banach spaces, with $V$ reflexive, $T \in \cL(W,V)$ and $G \colon V \to W$ be continuous. Define $F \colon V \to V$ by $F = I + T \circ G$, where $I$ is the identity operator on $V$, and assume that there exists $\bar x \in V$ such that $F(\bar x) = 0$. In addition, for every $h > 0$, let $T_h \in \cL(W,V)$ and set $F_h = I + T_h \circ G$. We assume that
 \begin{enumerate}[label={\rm(\roman*)}]
  \item \label{item:cor:first_example_i} we have
  \[
   \lim_{h \to 0} \norm{(T - T_h)w}{V} = 0 \quad \textnormal{for all } w \in W,
  \]

  \item  \label{item:cor:first_example_ii} there exists a bounded convex subset $\cM \subset \cL(V,W)$ satisfying
  \begin{equation}
   \inf_{M \in \cM} \gothC(I + T\circ M) \geq \kappa > 0,
  \end{equation}
  and such that for every $\epsilon > 0$, there exists $\delta > 0$ such that, for any $x,y \in B_V(\bar x, \delta)$ there is $M \in \cM$ satisfying
  \begin{equation}\label{eq:first_example_prediff}
   \norm{G(x) - G(y) - M(x-y)}{W} \leq \epsilon \norm{x - y}{V};
  \end{equation}

  \item  \label{item:cor:first_example_iii} there exist a function $L \colon \RR^*_+ \times \RR_+^* \to \RR_+$, increasing with respect to both variables and such that $\lim_{(h,r) \to 0} L(h,r) = 0$, and $R > 0$, such that the mapping $\Phi_h := (T - T_h) \circ G$ satisfies
  \[
   \norm{\Phi_h(x) - \Phi_h(y)}{Y} \leq L(h,r) \norm{x-y}{X} \quad \textnormal{for every } x,y \in B_{X}(\bar x, r), \, 0 < r < R.
  \]
 \end{enumerate}
Then, for every $0 < h \leq h_0$, there exists $\bar x_h \in V$ satisfying $ F_h(\bar x_h) = 0$, for some $h_0 > 0$ and we have the estimates
 \[
  \norm{\bar x - \bar x_h}{V} \leq \frac{2}{\kappa} \norm{(T-T_h)G(\bar x)}{V}
 \]
 and $\sur_{V,V}(F_h,\bar x_h) \geq \kappa / 2$.
Moreover, if for each $M \in \cM$, the operator $I + T \circ M$ is an isomorphism on $V$, then $F_h$ is strongly metrically regular near $\bar x_h$ and there is a neighborhood $\cO$ of $\bar x$, independent of $h$, such that this solution is unique in $\cO$.
\end{cor}

\begin{rem}
 The assumption \ref{item:cor:first_example_iii} in \cref{cor:first_example} holds in the following cases:
 \begin{enumerate}
  \item When $G$ is locally Lipschitz continuous and $\norm{T - T_h}{\cL(W,V)} \xrightarrow{h \to 0} 0$. Indeed, in this case we have that
  \[
    \Lip_{B_V(\bar x, r)}(\Phi_h) \leq \norm{T - T_h}{\cL(W,V)} \Lip_{B_V(\bar x,r)}(G) =: L(h,r).
  \]

  \item More generally, when $G$ is locally Lipschitz continuous and there is a subspace $Z \subset W$, such that $G(x) - G(y) \in Z$ for every $x,y \in V$, and if $\norm{T - T_h}{\cL(Z,V)} \xrightarrow{h \to 0} 0$.
  This typically happens when $G(x) = \tilde G(x) + \phi$, where $\tilde G \colon V \to Z$ is locally Lipschitz continuous and $\phi \in W$.
 \end{enumerate}
\end{rem}

\begin{proof}
 We are going to apply \cref{thm:approximation_main} with $X = X_h = Y = Y_h = V$ and $\tilde x_h = \bar x$.

 First, it is clear that \ref{item:cor:first_example_i} implies \cref{thm:approximation_main}-\ref{item:approximation_main_2}. We then check that $F$ is metrically regular near $\bar x$. Set
 \[
  \cA = \left \{ I + T \circ M : M \in \cM \right \}.
 \]
 We claim that $\cA$ satisfies the assumptions of \cref{thm:graves_variation}. Indeed, it is a bounded convex subset of $\cL(V,V)$. Moreover, for $0 < \sigma < \kappa$, using \eqref{eq:first_example_prediff} with $\epsilon = \sigma / (1 + \norm{T}{\cL(W,V)})$, we deduce that there exists $r > 0$ such that, for any $x,y \in B_V(\bar x, r)$, there is a $A \in \cA$ such that
 \begin{equation}
  \norm{F(x) - F(y) - A(x-y)}{V} < \sigma \norm{x - y}{V}.
 \end{equation}
 It then follows from \cref{thm:graves_variation} that $F$ is metrically regular near $\bar x$ with $\sur(F;\bar x) \geq \kappa - \sigma$. Since $\sigma$ is arbitrary, we conclude that $\sur(F;\bar x) \geq \kappa$. Note also that, when each $M \in \cM$ is an isomorphism, strong metric regularity follows from \cref{cor:strong_metric_regularity}. Finally, it is clear that \ref{item:cor:first_example_iii} implies \cref{thm:approximation_main}-\ref{item:approximation_main_3}, with $\Phi_h = \Psi_h$.
\end{proof}

The next result is similar to \cref{cor:first_example} but we now also approximate the mapping $G$ by some $G_h$ and not only $T$ as in the first example. It is a generalization of \cite[Theorem IV.3.8]{GR1986}, where it was used to study finite element approximations of Navier-Stokes equations. A similar result was used in \cite{MQ1982} to study pseudo-spectral approximations to the viscous Burgers' equations.

\begin{cor}\label{cor:second_example}
 Let $V$ and $W$ be Banach spaces, with $V$ reflexive, $T \in \cL(W,V)$ and $G \colon V \to W$ be continuous. Define $F \colon V \to V$ by $F = I + T \circ G$, where $I$ is the identity operator on $V$, and assume that there exists $\bar x \in V$ such that $F(\bar x) = 0$. In addition, for every $h > 0$, let $(V_h, \norm{\cdot}{V})$ be closed subspaces of $V$ and $(W_h, \norm{\cdot}{W_h})$ be such that $W \hookrightarrow W_h$. Let also $T_h \in \cL(W_h,V_h)$ and $G_h \colon V_h \to W_h$ be a continuous map. Define $F_h \colon V_h \to V_h$  by $F_h = I + T_h \circ G_h$. We assume that
 \begin{enumerate}[label={\rm(\roman*)}]
 \item \label{item:second_example_i} there exists $\Lambda > 0$ such that
 \begin{equation*}\label{eq:second_example_1}
  \sup_{h > 0} \norm{T_h}{\cL(W_h,V_h)} \leq \Lambda;
 \end{equation*}

 \item \label{item:second_example_ii} we have
\begin{equation*}\label{eq:second_example_2}
 \lim_{h \to 0} \norm{(T - T_h)w}{V} = 0 \quad \textnormal{for all } w \in W;
 \end{equation*}

 \item \label{item:second_example_iii} for every $h > 0$ there exists $\pi_h \in \cL(V,V_h)$ such that
 \begin{equation*}\label{eq:second_example_3}
  \lim_{h \to 0} \norm{\bar x - \pi_h \bar x}{V} = 0;
 \end{equation*}

 \item \label{item:second_example_iv} we have
 \begin{equation*}\label{eq:second_example_4}
   \lim_{h \to 0} \norm{G(\pi_h (\bar x)) - G_h(\pi_h(\bar x))}{W_h} = 0;
 \end{equation*}

 \item  \label{item:cor:first_example_v} there exists a bounded convex subset $\cM \subset \cL(V,W)$ satisfying
  \begin{equation*}
   \inf_{M \in \cM} \gothC(I + T\circ M) \geq \kappa > 0,
  \end{equation*}
  and such that for every $\epsilon > 0$, there exists $\delta > 0$ such that, for any $x,y \in B_V(\bar x, \delta)$ there is $M \in \cM$ satisfying
  \begin{equation*}
   \norm{G(x) - G(y) - M(x-y)}{W} \leq \epsilon \norm{x - y}{V};
  \end{equation*}

 \item \label{item:second_example_vi} there exists a function $L \colon \RR^*_+ \times \RR_+^* \to \RR_+$, increasing with respect to both variables and such that $\lim_{(h,r) \to 0} L(h,r) = 0$, and $R > 0$, such that the mappings $\Phi_h := (T - T_h) \circ G$ and $Q_h := G - G_h$ satisfy
  \begin{align}
   \norm{\Phi_h(x) - \Phi_h(y)}{V} \leq L(h,r) \norm{x-y}{V} & \quad \textnormal{for every } x,y \in B_{V}(\pi_h \bar x, r), \, 0 < r < R, \label{eq:second_example_Phi} \\
   \norm{Q_h(x) - Q_h(y)}{W_h} \leq L(h,r) \norm{x-y}{V_h} & \quad \textnormal{for every } x,y \in B_{V_h}(\pi_h \bar x, r), \, 0 < r < R. \label{eq:second_example_Psi}
  \end{align}
 \end{enumerate}
Then for every $0 < h \leq h_0$ there exists $\bar x_h \in X$ such that $F_h(\bar x_h) = 0$
for some $h_0 > 0$ and we have the estimates
\[
 \norm{\bar x - \bar x_h}{V} \leq \frac{\kappa + 4}{\kappa} \norm{\bar x - \pi_h \bar x}{X} + \frac{2}{\kappa} \left ( \norm{(T-T_h)G(\bar x)}{V} + \Lambda \norm{G(\bar x) - G_h(\pi_h \bar x))}{W_h} \right )
\]
and $\sur_{V_h,V_h}(F_h,\bar x_h) \geq \kappa / 2$.
 Furthermore, if, for each $M \in \cM$, the operator $I + T \circ M$ is an isomorphism on $V$, then $F_h$ is strongly metrically regular near $\bar x_h$ and there is a neighborhood $\cO$ of $\bar x$, independent of $h$, such that this solution is unique in $\cO$.
\end{cor}

\begin{rem}
 In order to clarify the connection of the result above to \cite[Theorem IV.3.8]{GR1986}, we recall that in \cite{GR1986} it is assumed that $G$ is continuously differentiable, that there are a function $L$, satisfying similar assumptions to the ones in \ref{item:second_example_vi}, and bounded linear operators $D_h \in \cL(V_h,W_h)$ such that
 \begin{equation}
  \norm{G_h(x) - G_h(y) - D_h(x-y)}{W_h} \leq L(h,r)\norm{x-y}{V_h} \textnormal{for every } x,y \in B_{V_h}(\tilde x_h, r),
 \end{equation}
 and that $\norm{dG(\bar x) - D_h}{\cL(V_h,W_h)} \xrightarrow{h \to 0} 0$.
 We then notice that
 \begin{align*}
  \norm{Q_h(x) - Q_h(y)}{W_h} & \leq \norm{G(x) - G(y) - dG[\bar x](x - y)}{W_h} + \norm{G_h(x) - G_h(y) - D_h(x - y)}{W_h} \\ & \quad + \norm{(dG[\bar x] - D_h)(x-y)}{W_h} \\
  & \leq C_h \norm{G(x) - G(y) - dG[\bar x](x - y)}{W} \\ & \quad + \left (L(h,r) + \norm{dG[\bar x] - D_h}{\cL(V_h,W_h)} \right ) \norm{x-y}{V_h} \\
  & \leq  \left (C_hc(r) + L(h,r) + \norm{dG[\bar x] - D_h}{\cL(V_h,W_h)} \right ) \norm{x-y}{V_h},
 \end{align*}
 where $C_h$ is the constant of the continuous embedding $W \hookrightarrow W_h$, $c \colon \RR_+ \to \RR_+$ is a nondecreasing function satisfying $c(0^+) = 0$ and
 \[
  \norm{G(x) - G(y) - dG[\bar x]}{W} \leq c(r) \norm{x - y}{V}
 \]
  for all $x, \, y \in B_V(\bar x, r)$.
 Setting $\tilde L(h,r) :=  \left (C_hc(r) + L(h,r) + \norm{dG[\bar x] - D_h}{\cL(V_h,W_h)} \right)$, we observe that $\lim_{(h,r) \to 0} \tilde L(h,r) = 0$ as soon as $\sup_{h > 0} C_h < + \infty$. Hence, it seems to us that \ref{item:second_example_vi} is a reasonable generalization of the assumptions made in \cite{GR1986}.

\end{rem}

\begin{rem}
 In both \cref{cor:first_example,cor:second_example}, one may drop the reflexivity assumption on $V$ if the set
 \[
  \cA := \left \{ I + T \circ M : M \in \cM \right \}
 \]
 is compact in the topology of $\cL(V)$. In this case \cref{thm:graves_variation} has to be replaced by \cite[Theorem 3.9]{CFI2015}, see also \cite[Theorem 6.65]{I2017}.
\end{rem}

\begin{proof}
 We are going to apply \cref{thm:approximation_main} with $X = Y = V$, $X_h = Y_h = V_h$ and $\tilde x_h = \pi_h \bar x$.

 Notice first that \cref{thm:approximation_main}-\ref{item:approximation_main_2} holds by combining \ref{item:second_example_ii}, \ref{item:second_example_iii}, \ref{item:second_example_iv} and the continuity of $G$. Moreover, \cref{thm:approximation_main}-\ref{item:approximation_main_1} (and strong metric regularity) is proved exactly as in the proof of \cref{cor:first_example}.

 It remains to check that \cref{thm:approximation_main}-\ref{item:approximation_main_3} holds. We set $\widehat F_h = I + T_h \circ G$. Clearly $\widehat F_h(V_h) \subset V_h$. Conversely, if there is $x \in V$ such that $\widehat F_h(x) = y_h \in V_h$, then $x = y_h - T_h \circ G(x) \in V_h$. Hence $\widehat F_h^{-1}(V_h) \subset V_h$.  Notice that \eqref{eq:approximation_main_Phi} is a direct consequence of \eqref{eq:second_example_Phi}. Moreover, for any $x,y \in B_{V_h}(\tilde x_h,r)$, using \ref{item:second_example_i} and \ref{item:second_example_vi}, we deduce
 \begin{align*}
  \norm{\Phi_h(x) - \Phi_h(y)}{V_h} \leq \norm{T_h}{\cL(W_h,W)} \norm{Q_h(x) - Q_h(y)}{W_h} \leq \Lambda L(h,r) \norm{x - y}{V_h}.
 \end{align*}
 So that \eqref{eq:approximation_main_Psi} holds.
\end{proof}

\section{Generalized differential for Nemytskii operators on Lebesgue spaces}
\label{section:nemytskii}

In this section we consider a generalized differential for the Nemytskii operator over $L^p$ spaces. The construction is based on Clarke's generalized Jacobian and we prove that it inherits some of its main properties. It seems that this construction was first introduced in \cite{U2002}, see also \cite[Section 3.3]{U2011}. The main result of this section is \cref{thm:nemytskii_prediff} below, which will be useful to apply our extensions of the BRR theorem in the applications considered in \cref{section:application}. The reader may admit \cref{thm:nemytskii_prediff} and continue directly with the next section. The necessary facts are recalled in \cref{section:set_valued}.

\subsection{Clarke's generalized differentials}
\label{section:clarke}

We first recall the definition of Clarke's subdifferential and its extensions. We refer the reader to \cite[Chapter 2]{C1990} for more information on this topic. For a locally Lipschitzian function $f \colon \RR^d \to \RR$, its \emph{Clarke generalized directional derivative} $f^\circ \colon \RR^d \times \RR^d \to \RR$ is defined by
  \begin{equation}
   f^\circ(z,\upsilon) = \limsup_{\substack{y \to z \\ \tau \searrow 0}} \frac{f(y + \tau \upsilon) - f(y)}{\tau}.
  \end{equation}
The \emph{Clarke's subdifferential} of $f$ at $z \in \RR^d$ is then defined as the set
\begin{equation}\label{eq:subdiff_def}
 \partial^C f(z) = \left \{ \xi \in \RR^d : f^\circ(z,\upsilon) \geq \xi \cdot \upsilon \, \textnormal{ for all } \upsilon \in \RR^d \right \},
\end{equation}
but is also characterized by
\begin{equation}\label{eq:clarke_characterization}
 \partial^C f(z) = \co \left \{ \xi \in \RR^{d} : \exists (z_n)_{n \geq 0} \subset \RR^d \setminus \mathcal{N}, \, \xi = \lim_{n \to \infty} \nabla f(z_n) \right \},
\end{equation}
where $\cN \subset \RR^d$ is a negligible set\footnote{The set $\mathcal{N}$ is known to exist because of Rademarcher's theorem \cite[Theorem 3.2]{EG2015}. Although it may seem like this definition depends on the choice of the set $\mathcal{N}$, it was proved in \cite[Theorem 4]{W1981} that this is not the case.} such that $f$ is differentiable on $\RR^d \setminus \cN$.
The latter has a natural extension to vectorial functions $f \colon \RR^d \to \RR^n$, in this case its \emph{Clarke generalized Jacobian}, also denoted by $\partial^C f$, is defined through
\begin{equation}\label{eq:clarke_jacobian}
 \partial^C f(z) = \co \left \{ M \in \RR^{n \times d} : \exists (z_n)_{n \geq 0} \subset \RR^d \setminus \mathcal{N}, \, M = \lim_{n \to \infty} Jf(z_n) \right \},
\end{equation}
where $Jf$ denotes the Jacobian matrix of $f$.

For a locally Lipschitz function $f \colon \RR^d \to \RR$, it follows from \cite[Proposition 2.1.2]{C1990} that the support function\footnote{See \cref{defi:support_function}.} of $\partial^C f$ satisfies
\[
 \sigma_{\partial^Cf(z)}(\upsilon) = f^\circ(z,\upsilon) \quad \textnormal{for all $z,\, \upsilon \in \RR^d$.}
\]
In the case of vectorial functions, we have the following characterization for the support function of Clarke's generalized Jacobian.
\begin{thm}[{\cite{I2002}}]\label{thm:imbert}
  Let $f \colon \RR^d \to \RR^n$, with $n > 1$, be locally Lipschitz continuous. Then, for every $z_0 \in \RR^d$ and $M \in \RR^{n \times d}$, the support function of $\partial^C f$ is characterized by
  \[
   \sigma_{\partial^C f(z_0)}(M) = \limsup_{\substack{z \to z_0 \\ \epsilon \searrow 0}} \frac{1}{\epsilon^d} \int_{\partial P_\epsilon(z)} f(y) \cdot M \nu(y) \cH^{d-1}(dy),
  \]
  where
  \[
   P_\epsilon(z) := \left \{ z + \epsilon \sum_{i=1}^d  t_i e_i : \, t_i \in [0,1] \quad \textnormal{for all } 1 \leq i \leq d \right \}
  \]
  is a $d$-dimensional hypercube, $\nu$ is the outer normal vector to $P_\epsilon(z)$ and $\cH^{d-1}$ is the $(d-1)$-dimensional Hausdorff measure.
\end{thm}

We also recall the properties of the generalized Jacobian which will be used below.
\begin{prop}[{\cite[Proposition 2.6.2]{C1990}}]\label{prop:clarke_basic_properties}
 Let $f \colon \RR^d \to \RR^n$ be locally Lipschitz continuous. Then
 \begin{enumerate}[label={\rm(\roman*)}]
  \item for every $z \in \RR^d$, the set $\partial^C f(z)$ is non-empty, convex and compact;
  \item \label{item:clarke_semicontinuity} the set-valued map $\RR^d \ni z \mapsto \partial^C f(z) \subset \RR^{n \times d}$ is upper semicontinuous. More precisely, for every $\bar z \in \RR^d$ and $\epsilon > 0$, there is $\delta > 0$ such that
  \[
   \partial^C f(z) \subset \partial^C f(\bar z) + B_{\RR^{n \times d}}(0,\epsilon) \quad \textnormal{whenever } \module{z - \bar z} < \delta.
  \]
 \end{enumerate}
\end{prop}

\begin{prop}[Vectorial mean value theorem, {\cite[Proposition 2.6.5]{C1990}}]\label{prop:clarke_mean_value}
Let $f \colon \RR^d \to \RR^n$ be locally Lipschitz continuous. Then, for any $x,y \in \RR^d$, one has
\[
 f(x) - f(y) \in \co \left ( \partial^C f([x,y]) \right )(x-y).
\]
\end{prop}

\begin{lem}\label{lem:graph_clarke}
 Let $(\Omega, \cF)$ be a measurable space and let $f \colon \Omega \times \RR^d \to \RR^n$ be such that
 \begin{enumerate}[label={\rm (\roman*)}]
  \item the function $f_z \colon \omega \mapsto f(\omega,z)$ is $(\cF, \cB(\RR^d))$-measurable for every $z \in \RR^d$,
  \item the function $f_\omega \colon z \mapsto f(\omega,z)$ is locally Lipschitz continuous for every $\omega \in \Omega$.
 \end{enumerate}
 Then the set-valued map $\Omega \times \RR^d \ni (\omega,z) \mapsto \partial^C f(\omega,z) \subset \RR^{n \times d}$ is $\cF \otimes \cB(\RR^d)$-measurable, where $\partial^C f(\omega,z) := \partial^C f_\omega(z)$ for every $(\omega,z) \in \Omega \times \RR^d$. Moreover, if $\cF$ is complete with respect to some measure $\mu$, then
 \[
  \Omega \ni \omega \mapsto \gph(\partial^C f(\omega,\cdot)) \subset \RR^d \times \RR^{n \times d} \quad \textnormal{is $\cF$-measurable.}
 \]
\end{lem}

\begin{proof}
Let us first prove the measurability of $(\omega,z) \mapsto \partial^C f(\omega,z)$.
Notice that $f$ is a Carathéodory function, so that \cref{prop:caratheodory_measurable} implies that it is ($\cF \otimes \cB(\RR^d), \cB(\RR^{n})$)-measurable.

 Let us start by the case $n=1$. Since
 \[
  f^\circ(\omega,z;\upsilon) = \limsup_{\substack{\module{h} \to 0 \\ \tau \searrow 0}} \frac{f(\omega, z +h + \tau \upsilon) - f(\omega,z +h)}{\tau},
 \]
 we have that, for every $\upsilon \in \RR^d$, the map $(\omega,z) \mapsto f^\circ(\omega,z;\upsilon)$ is $\cF \otimes \cB(\RR^d)$-measurable.
 Since $f^\circ(\omega,z;\upsilon) = \sigma_{\partial^C f(\omega,z)}(\upsilon)$ and $\partial^Cf (\omega,z)$ is compact for every $(\omega,z) \in \Omega \times \RR^d$, the conclusion follows from \cref{prop:support_measurable}.

 We turn to the case $n > 1$. Using the properties of Hausdorff measures \cite[Theorem 2.2 p.84]{EG2015}, we have
 \[
  \frac{1}{\epsilon^d} \int_{\partial P_\epsilon(z)} f(\omega,y) \cdot M \nu(y) \cH^{d-1}(dy) = \frac{1}{\epsilon} \int_{\partial P_1(0)} f(\omega,z + \epsilon y) \cdot M \nu(z + \epsilon y) \cH^{d-1}(dy),
 \]
 where the notations are those of \cref{thm:imbert}. Then, using \cref{thm:imbert}, we get
 \[
  \sigma_{\partial^C f(\omega,z_0)}(M) = \limsup_{\substack{h \to 0 \\ \epsilon \searrow 0}} \frac{1}{\epsilon} \int_{\partial P_1(0)} f(\omega,z_0 + h + \epsilon y) \cdot M \nu(z_0 + h + \epsilon y) \cH^{d-1}(dy)
 \]
 and it follows that $\Omega \times \RR^d \ni (\omega,z) \mapsto \sigma_{\partial^C f(\omega,z)}(M) \in \RR$ is $\cF \otimes \cB(\RR^d)$-measurable for every $M \in \RR^{n \times d}$. Again, the conclusion follows from \cref{prop:support_measurable}.

 We turn to the measurability of $\omega \mapsto \gph(\partial^C f(\omega,\cdot)) \subset \RR^d \times \RR^{n \times d}$. From \cref{thm:graph_measurability} We have that
 \[
  \gph(\partial^Cf(\cdot,\cdot)) = \gph(\omega \mapsto \gph(\partial^C f(\omega,\cdot))) \subset \cF \otimes \cB(\RR^d) \otimes \cB(\RR^{n\times d}).
 \]
 The conclusion then follows from another application of \cref{thm:graph_measurability}, using the completeness of $\cF$.
\end{proof}

\subsection{Generalized differential for Nemytskii operators}
\label{section:Nemytskii}

We now turn to the definition and the analysis of the generalized differential for Nemytskii operators defined on Lebesgue spaces. We fix a complete finite measure space $(\Omega, \cF, \mu)$  and we consider a function $H \colon \Omega \times \RR^d \to \RR^n$. We list below all the assumptions on $H$ that will be used in this section.
\begin{enumerate}[label={\bf(N\arabic*)}]
 \item \label{item:nemytskii_lip} The function $H$ is $(\cF \otimes \cB(\RR^d), \cB(\RR^n))$-measurable and $z \mapsto H(x,z)$ is locally Lipschitz continuous for every $x \in \Omega$.

 \item \label{item:nemytskii_growth} There exist positive constants $C_H$ and $ 1 < q \leq p < \infty$ such that
 \begin{equation}\label{eq:nemytskii_growth}
 \module{H(x,z)} \leq C_H \left ( 1 + \module{z}^{p/q} \right ) \quad \textnormal{for all } (x,z) \in \Omega \times \RR^d
\end{equation}
and
\begin{equation}\label{eq:nemytskii_clarke_growth}
 \left ( M \in \partial^C H(x,z) \right ) \Longrightarrow  \left (\module{M} \leq C_H \left ( 1 + \module{z}^{\frac{p}{q}-1} \right ) \right ) \quad \textnormal{for all } (x,z) \in \Omega \times \RR^d,
\end{equation}
where $\partial^C H$ denotes Clarke's generalized Jacobian with respect to the second variable.
\end{enumerate}

For $1 < q \leq p < \infty$ we define the \emph{Nemytskii operator}\footnote{Also called superposition operator.} $\gothH \colon L^p(\Omega;\RR^d) \to L^q(\Omega;\RR^n)$,  defined by
\begin{equation}\label{eq:nemytskii}
  \gothH [u](x) = H(x,u(x)) \quad \textnormal{for every } x \in \Omega.
\end{equation}
Clearly $\gothH$ is well defined as soon as \eqref{eq:nemytskii_growth} holds.

We also recall the following fact.
\begin{prop}[{\cite[Theorem 3.7]{AZ1990}}]\label{prop:nemytskii_continous}
 If \ref{item:nemytskii_lip} and \eqref{eq:nemytskii_growth} hold, then $\gothH \colon  L^p(\Omega;\RR^d) \to L^q(\Omega;\RR^n)$ is continuous.
\end{prop}

We also have the following result.
\begin{lem}\label{lem:clarke_measurable}
 Assume \ref{item:nemytskii_lip} and let $u \colon \Omega \to \RR^d$ be any measurable mapping. Then $\partial^C H(\cdot, u(\cdot))$ is measurable with non-empty, convex and compact values.
\end{lem}
\begin{proof}
 This follows from \cref{prop:measurability_tools}-\ref{item:measurability_composition} combined with \cref{prop:clarke_basic_properties} and \cref{lem:graph_clarke}.
\end{proof}

Our goal in this section is to define a generalized differential for the Nemytskii operator $\gothH$. For this, we assume that \ref{item:nemytskii_lip} holds, so that Clarke's generalized Jacobian for $H$ with respect to the second variable exists, and we consider the following object, for every $u \in L^p(\Omega; \RR^d)$,
\begin{equation}\label{eq:Nemytskii_subdiff}
 \partial^\circ \gothH[u] = \left \{ \begin{array}{c} A : \Omega \to \RR^{n \times d} \textnormal{ is } (\cF,\cB(\RR^{n \times d}))\textnormal{-measurable and } \\ A(x) \in \partial^C H(x, u(x)) \textnormal{ for }\mu \textnormal{-a.e. } x \in \Omega \end{array}\right \}.
\end{equation}
Notice that $\partial^\circ \gothH [u] = \partial^\circ \gothH [v]$  whenever $u = v$ $\mu$-almost everywhere.

\begin{prop}\label{prop:nemytskii_basic_properties}
 Let $1 < q \leq p < \infty$, assume \ref{item:nemytskii_lip}, \ref{item:nemytskii_growth} and set
 \begin{equation}\label{eq:nemytskii_exponent}
  r = \begin{cases}
            pq/(p-q) & \quad \textnormal{if } q < p ,\\
            \infty & \quad \textnormal{if } p = q.
           \end{cases}
 \end{equation}
 Then, $\partial^\circ \gothH$ defines a set-valued map from $L^p(\Omega;\RR^d)$ to $L^{r}(\Omega;\RR^{n \times d})$ with $\dom \partial^\circ \gothH = L^p(\Omega; \RR^d)$. Moreover, for every $u \in L^p(\Omega,\RR^d)$, the set $\partial^\circ \gothH [u]$ is convex and closed for the strong topology of $L^{r}(\Omega;\RR^{n \times d})$. Moreover, if $q<p$, it is sequentially compact as a subset\footnote{Recall that since $\frac{1}{q} = \frac{1}{p} + \frac{1}{r}$, any $\xi \in L^r(\Omega;\RR^{n \times d})$ may be identified with the linear operator $L^p(\Omega;\RR^d) \ni v \mapsto \xi v \in L^q(\Omega;\RR^n)$, which is bounded as a consequence of Hölder's inequality.} of $\cL(L^p(\Omega,\RR^d), L^q(\Omega,\RR^n))$ endowed with the weak operator topology.
\end{prop}

\begin{proof}
 Let $u \in L^p(\Omega;\RR^d)$. From \cref{lem:clarke_measurable} we know that, for every $x \in \Omega$, the set $\partial^C H(x,u(x))$ is non-empty, convex and compact and that the set-valued mapping $\Omega \ni x \mapsto \partial^C H(x,u(x))$ is measurable. We can apply \cref{thm:KRN} to conclude that $\partial^\circ \gothH [u]$ is non-empty. Since $u$ is arbitrary, we have $\dom \partial^\circ \gothH = L^p(\Omega; \RR^d)$. Let now $A \in \partial^\circ \gothH [u]$. Using the growth condition \eqref{eq:nemytskii_clarke_growth} we have
 \[
  \module{A(x)}^{pq/(p-q)} \leq C \left ( 1 + \module{u(x)}^{p} \right ) \quad \textnormal{for a.e. $x \in \Omega$},
 \]
 so that $\partial^\circ \gothH [u] \subset L^{pq/(p-q)}(\Omega;\RR^{n \times d})$ when $q < p$. Moreover, it is also clear from \eqref{eq:nemytskii_clarke_growth} that $A \in L^\infty(\Omega; \RR^{n \times d})$ when $p =q$. The fact that $\partial^\circ \gothH [u]$ is convex and closed follows from the definition and the fact that $\partial^C H(x,u(x))$ is convex and closed for every $x \in \Omega$.
  The sequential compactness in the weak operator topology is a direct consequence of \cref{prop:WOT_Lp_compact}.
\end{proof}

The following result is a form of mean valued theorem for Nemytskii operators.
\begin{thm}[Mean value theorem]\label{thm:nemytskii_mean_value}
  Let $1 < q \leq p < \infty$, let $r$ be defined by \eqref{eq:nemytskii_exponent} and assume \ref{item:nemytskii_lip}, \ref{item:nemytskii_growth}. For every $u,v \in L^p(\Omega;\RR^d)$ we have
  \begin{equation}\label{eq:nemytskii_mean_value}
   \gothH[u] - \gothH[v] \in \Theta[u,v] (u - v) \quad \textnormal{in } L^{q}(\Omega;\RR^n),
  \end{equation}
  where
  \[
   \Theta[u,v] = \left \{ A \in L^{r}(\Omega;\RR^{n\times d}) : A(x) \in \co \left (\partial^C H(x,[u(x), v(x)]) \right ) \textnormal{ for a.e. } x \in \Omega \right \},
  \]
  with $[u(x), v(x)] = \co \left( \{ u(x), v(x) \} \right)$ and $\partial^C H(x,[u(x), v(x)]) = \bigcup_{z \in [u(x), v(x)]} \left \{\partial^C H(x,z) \right\}$.
  Furthermore, if $A \in \Theta[u,v]$ is such that $\gothH[u] - \gothH[v] = A(u-v)$, then there exist measurable functions $\lambda_1, \dots, \lambda_{nd+1} \colon \Omega \to [0,1]$ and measurable selections $A_1, \dots, A_{nd+1} \colon \Omega \to \RR^{n \times d}$ in $\partial^C H(\cdot,[u(\cdot), v(\cdot)])$ such that that
 \[
  \sum_{i=1}^{nd+1} \lambda_i(x) = 1 \quad \textnormal{and} \quad  A = \sum_{i=1}^{nd+1} \lambda_i A_i \quad \textnormal{a.e. in $\Omega$}.
 \]
\end{thm}

\begin{proof}
 Let $u,\, v \in L^p(\Omega; \RR^d)$. From \cref{prop:clarke_mean_value} we have
 \[
  H(x, u(x)) - H(x, v(x)) \in \co \left ( \partial^C H(x, [u(x),v(x)]) \right ) (u(x) - v(x)) \quad \textnormal{for a.e. } x \in \Omega.
 \]
 Let us consider the  set-valued map $\Phi \colon \Omega \rightrightarrows \RR^{n \times d}$ defined by
 \[
  \Phi(x) = \left \{ M \in \co \left ( \partial^C H(x, [u(x),v(x)]) \right ) : H(x, u(x)) - H(x, v(x)) = M(u(x) - v(x)) \right \}.
 \]
 We claim that the mapping
 \begin{equation}\label{eq:nemytskii_mean_value_1}
  \Omega \ni x \mapsto \co \left ( \partial^C H(x, [u(x),v(x)]) \right )
 \end{equation}
 is measurable and has closed values. Indeed, from properties \ref{item:measurability_reunion} and \ref{item:measurablility_convex_hull} in \cref{prop:measurability_tools}, we know that the set-valued map $\Omega \ni x \mapsto [u(x),v(x)] = \co \left ( \{u(x) \} \cup \{v(x) \} \right)$ is measurable. Since, by \cref{prop:clarke_basic_properties}, the mapping $z \mapsto \partial^C H(x, z)$ is upper semicontinuous for every $x \in \Omega$ and closed valued, we deduce from \cref{lem:graph_clarke} and \cref{prop:measurability_tools}-\ref{item:measurability_composition} that $\Omega \ni x \mapsto \partial^C H(x,[u(x),v(x)])$ is measurable. Another applications of \cref{prop:measurability_tools}-\ref{item:measurablility_convex_hull} yields the measurability of \eqref{eq:nemytskii_mean_value_1}, while the fact that it has closed values follows from \cite[Proposition 3 p.42]{AC1984}. This proves the claim.

 Moreover, the mapping $(x,M) \mapsto  M(u(x) - v(x))$ is measurable with respect to $x$ and continuous with respect to $M$, so that it is a Carathéodory map. We can therefore apply \cref{thm:implicit_measurability} to obtain a measurable function $A \colon \Omega \to \RR^{n \times d}$ such that $A(x) \in \Phi(x)$ for every $x \in \Omega$. The fact that $A \in L^{r}(\Omega;\RR^{n \times d})$ then follows from \cref{eq:nemytskii_clarke_growth}. The last statement is a direct consequence of a version of Carathéodory's theorem for measurable selections \cite[Theorem 8.2.15]{AF1990}.
\end{proof}

We now prove that the set-valued map $\partial^\circ \mathfrak{H}$ is upper-semicontinuous.
\begin{prop}\label{prop:nemytskii_osc}
 Let $1 < q < p < \infty$, set $r = pq/(p-q)$ and assume \ref{item:nemytskii_lip}-\ref{item:nemytskii_growth}. For every $\epsilon >0$ and $\bar u \in L^p(\Omega, \RR^d)$, there exists $\delta > 0$ such that, for all $u,v \in B_{L^p}(\bar u, \delta)$ and for every measurable selection $A$ in $\partial^C H(\cdot, [u(\cdot), v(\cdot)])$, we have
 \[
  A \in  \partial^\circ \gothH [\bar u] + B_{L^r}(0,\epsilon).
 \]
 In particular $\partial^\circ \gothH$ is upper-semicontinuous from $L^p(\Omega,\RR^d)$ to $L^r(\Omega, \RR^{n \times d})$.
\end{prop}

\begin{proof}
 Assume that the result is false. Then, there exists $\epsilon > 0$, such that for every $k \geq 1$, there exist $u_k, v_k \in B_{L^p}(\bar u, k^{-1})$ and measurable selections $A_k$ in $\partial^C H(\cdot, [u_k(\cdot), v_k(\cdot)])$ such that $d(\partial^\circ \gothH[\bar u], A_k) \geq \epsilon$, where the distance is measured in $L^{r}(\Omega; \RR^{n \times d})$. Clearly the sequences $u_k$ and $v_k$ converge to $\bar u$ in $L^p(\Omega; \RR^d)$.

 For each $x \in \Omega$ we define the projection $\bar A_k(x) := \Pi_{\partial^C H(x, \bar u(x))}(A_k(x))$.
 From \cref{prop:measurability_tools}-\ref{item:measurability_projection}, we know that each $\bar A_k$ is measurable. Since, in addition, we have $\bar A_k (x) \in \partial^C H(x, \bar u(x))$ for every $x \in \Omega$, we obtain that $\bar A_k \in \partial^\circ \gothH[\bar u]$.

 From the convergence of $u_k$ and $v_k$ to $\bar u$ in $L^p(\Omega;\RR^d)$ and \cite[Theorem 4.9]{B2011} we may extract subsequences, without change of notations, such that the convergence holds $\mu$-almost everywhere and there exists $g \in L^{p}(\Omega;\RR)$ such that
 \begin{equation}\label{eq:nemytskii_osc_domination}
  \max \left \{ \module{u_k}, \module{v_k} \right \} \leq g \quad \mu \textnormal{-almost everywhere.}
 \end{equation}
 By definition of $A_k$, for $\mu$-a.e. $x \in \Omega$, there exists $t^k_x \in [0,1]$ such that
 \[
  A_k(x) \in \partial^C H(x, t^k_x u_k(x) + (1 - t^k_x) v_k(x)).
 \]
  It follows from \eqref{eq:nemytskii_clarke_growth} and \eqref{eq:nemytskii_osc_domination} that
  \[
   \module{A_k} \leq C_H \left (1 + g^{\frac{p}{q} - 1} \right) \quad \mu \textnormal{-almost everywhere.}
  \]
  Then, since $t^k_x u_k(x) + (1 - t^k_x) v_k(x)$ converges to $\bar u(x)$ for $\mu$-a.e. $x \in \Omega$, and using the upper semicontinuity in \cref{prop:clarke_basic_properties}, for $\mu$-a.e. $x \in \Omega$ and every $\eta > 0$, there exists $K = K(x,\eta)> 0$ such that, for every $k \geq K$, we have
 \[
  A_k(x) \in \partial^C H(x, \bar u(x)) + B_{\RR^{n \times d}}(0,\eta).
 \]
 In particular, since $\bar A_k(x)$ is defined as the projection of $A_k(x)$ on $\partial^C H(x, \bar u(x))$, we have
 \[
  \module{\bar A_k(x) - A_k(x)} \leq \eta \quad \textnormal{for all } k \geq K.
 \]
 Hence, if we define $\varphi_k := \bar A_k - A_k$, we have that $\varphi_k$ converges $\mu$-a.e. to $0$. Moreover, from \eqref{eq:nemytskii_clarke_growth},
 \[
  \module{\varphi_k(x)} \leq \module{\bar A_k(x)} + \module{A_k(x)} \leq C_H \left ( 2 + \module{\bar u(x)}^{\frac{p}{q} - 1} + g(x)^{\frac{p}{q} - 1} \right) \quad \mu \textnormal{-almost everywhere},
 \]
 so that we can apply Lebesgue's convergence theorem to deduce that $\varphi_k$ converges to $0$ in $L^{r}(\Omega;\RR^{n\times d})$. This contradicts $d(\partial^\circ \gothH[\bar u], A_k) \geq \epsilon$ and concludes the proof.
\end{proof}

\begin{rem}\label{rem:nemytskii_osc}
 \cref{prop:nemytskii_osc} does not hold for $1 < p=q < \infty$ and $r = \infty$. Indeed, for $d = n = 1$, assume that $\mu$ is atomless and let $H(x,z) = H(z) = \module{z}$, in this case it is well known that
 \[
  \partial^C H(z) = \partial^C H(z) = \begin{cases}
                                            \{-1\} & \quad \textnormal{for } z < 0, \\
                                            [-1,1] & \quad \textnormal{for } z = 0, \\
                                            \{1\} & \quad \textnormal{for } z > 0.
                                           \end{cases}
 \]
 We build a counterexample to the upper-semicontinuity of $\partial^\circ \gothH$ in this case. For any $0 < \delta < \mu(\Omega)^{1/p}$, we consider $E_\delta \in \cF$ with $\mu(E_\delta) = \delta^p$. Notice that $E_\delta$ exists since we assume that $\mu$ is atomless (see \cite[Corollary 1.12.10]{B2007}). Let also
 \[
  \bar u (x) = 1/2 \quad \textnormal{for all } x \in \Omega \quad \textnormal{and} \quad  u_\delta(x) = \begin{cases}
                                                                                              1/2 & \quad \textnormal{if } x \notin E_\delta, \\
                                                                                              -1/2 & \quad \textnormal{if } x \in E_\delta.
                                                                                             \end{cases}
 \]
 Then $\norm{\bar u - u_\delta}{L^p} < \delta$, $\partial^\circ \gothH[\bar u] = \{ \bar \xi \}$ and $\partial^\circ \gothH[u_\delta] = \{\xi_\delta \}$ in $L^\infty(\Omega;\RR)$, where
 \[
  \bar \xi (x) = 1 \quad \textnormal{for all } x \in \Omega \quad \textnormal{and} \quad \xi_\delta(x) = \begin{cases}
                                                                                                  1 \quad & \textnormal{if } x \notin E_\delta, \\
                                                                                                  -1 \quad & \textnormal{if } x \in E_\delta.
                                                                                                 \end{cases}
 \]
 Hence, we have $d(\xi_\delta, \partial^\circ \gothH[\bar u]) = \norm{\bar \xi - \xi_\delta}{L^\infty} = 2$.

\end{rem}

We are now able to prove the main result of this section.
\begin{thm}\label{thm:nemytskii_prediff}
 Let $1 < q < p < \infty$ and assume \ref{item:nemytskii_lip}-\ref{item:nemytskii_growth}. Then, for every $\epsilon > 0$ and $\bar u \in L^p(\Omega;\RR^d)$, there exists $\delta > 0$ such that for every $u,v \in B_{L^p}(\bar u , \epsilon)$, there exists $M \in \partial^\circ \gothH[\bar u]$ satisfying
 \begin{equation}\label{eq:nemytskii_prediff}
  \norm{\gothH[u] - \gothH[v] - M (u - v)}{L^q} < \epsilon \norm{u - v}{L^p}.
 \end{equation}
\end{thm}

\begin{proof}
  Let $\epsilon >0 $ and set $r = pq/(p-q)$. According to \cref{prop:nemytskii_osc}, there exists $\delta > 0$ such that, for all $w_1,\, w_2 \in B_{L^p}(\bar u, \delta)$ and every measurable selection $A$ in $\partial^C H(\cdot, [w_1(\cdot), w_2(\cdot)])$, we have
  \begin{equation}\label{eq:prederivative_1}
   \ A \in  \partial^\circ \gothH[\bar u] + B_{L^r} \left ( 0, \frac{\epsilon}{nd+1} \right ).
  \end{equation}
  Let $u,v \in B_{L^p}(\bar u, \delta)$. From \cref{thm:nemytskii_mean_value}, there exist measurable functions $\lambda_1, \dots, \lambda_{nd+1} \colon \Omega \to [0,1]$, measurable selections $A_1, \dots, A_{nd+1} \colon \Omega \to \RR^{n \times d}$ in $\partial^C H(\cdot, [u(\cdot), v(\cdot)])$ and $A \in L^r(\Omega, \RR^{n \times d})$ such that
 \[
  \sum_{i=1}^{nd+1} \lambda_i(x) = 1, \quad A = \sum_{i=1}^{nd+1} \lambda_i A_i
 \]
and
 \begin{equation}\label{eq:prederivative_2}
  \gothH[u] - \gothH[v] = A \left (u - v \right).
 \end{equation}
 We can apply \eqref{eq:prederivative_1} to $A_i$ to obtain that
 \[
  A_i \in \partial^\circ \gothH [\bar u] + B_{L^r} \left (0, \frac{\epsilon}{nd+1} \right ) \quad \textnormal{for all } 1 \leq i \leq nd+1.
 \]
 It follows that there exists $M_i \in \partial^\circ \gothH [\bar u]$, such that $\norm{A_i - M_i}{L^{r}} \leq \epsilon/(nd+1)$.
 Notice now that we have
 \[
  \sum_{i = 1}^{nd+1} \lambda_i(x) M_i(x) \in \partial^C H(x, \bar u(x)) \quad \textnormal{for $\mu$-a.e. $x \in \Omega$}
 \]
 because $\partial^C H(x, \bar u(x))$ is convex and $M_i(x) \in \partial^C H(x, \bar u(x))$ for $\mu$-a.e. $x \in \Omega$. Therefore, if we set $M = \sum_{i=1}^{nd+1} \lambda_i M_i$, then we have $M \in \partial^\circ \gothH[\bar u]$ and
 \[
  \norm{A - M}{L^{r}} \leq \sum_{i=1}^{nd+1} \norm{A_i - M_i}{L^{r}}  \leq \epsilon.
 \]
 Finally, recalling \eqref{eq:prederivative_2},
 \begin{align*}
  & \norm{\gothH [u] - \gothH [v] - M(u-v)}{L^q} = \norm{\gothH [u] - \gothH [v] - (A - A + M) (u-v)}{L^q} \\
  & \qquad = \norm{(A - M)(u-v)}{L^q}\leq \norm{A - M}{L^{r}} \norm{u - v}{L^p} \\
  & \qquad \leq \epsilon \norm{u - v}{L^p}. \qedhere
 \end{align*}
\end{proof}

\begin{rem}\label{rem:nemytskii_prediff}
 The proof of \cref{thm:nemytskii_prediff} fails for $p = q$ because \cref{prop:nemytskii_osc} is not true in this case, see \cref{rem:nemytskii_osc}. Nevertheless, we do not expect that its conclusion holds in general for $p=q$. Indeed, in addition to the assumptions in \cref{thm:nemytskii_prediff}, assume that $\mu$ is atomless, that $z \mapsto H(x,z)$ is continuously differentiable for every $x \in \Omega$ and that $d = n = 1$. In this case, we have $\partial^C H(x,z) = \partial^C H(x,z) = \{ \partial_z H(x,z) \}$ for every $(x,z) \in \Omega \times \RR$ (see \cite[Proposition 2.2.4]{C1990}). Then, \eqref{eq:nemytskii_prediff} implies that for every $\epsilon > 0$, there is $\delta > 0$ such that, for every $\norm{h}{L^p} < \delta$, we have
 \[
  \norm{\gothH[\bar u + h] - \gothH[\bar u] - \partial_z H(\cdot,\bar u(\cdot))h}{L^p} < \epsilon \norm{h}{L^p}.
 \]
 The latter implies the Fréchet differentiability of $\gothH \colon L^p(\Omega;\RR) \to L^p(\Omega; \RR)$. However, it then follows from \cite[Theorem 3.12]{AZ1990} that $H$ must be of the following form
 \[
  H(x,z) = a(x) + b(x)z \quad \textnormal{for all } (x,z) \in \Omega \times \RR,
 \]
 where $a \in L^p(\Omega;\RR)$ and $b \in L^\infty(\Omega; \RR)$.

\end{rem}

\section{Applications to finite element approximations of nonlinear elliptic PDEs}
\label{section:application}

In this section we consider two applications of our generalized BRR theorem. We first consider a simple application to a viscous Hamilton-Jacobi equation where the nonlinearity is nonsmooth. The second application concerns a stationary second order mean field game system, extending the results from \cite{BLS2025}.

\subsection{Viscous Hamilton-Jacobi equations}
\label{section:HJ}
In this section we apply our abstract result to the following elliptic Dirichlet problem
\begin{equation}\label{eq:HJ}
 \begin{cases}
  - \Delta u(x) + H(x,Du) + \lambda u(x) = f(x) \quad & \textnormal{in } \Omega, \\
  u(x) = 0 \quad & \textnormal{on } \partial \Omega,
 \end{cases}
\end{equation}
where $\Omega$ is a bounded Lipschitz domain in $\RR^d$. The final result is \cref{thm:approximation_HJB} below. Since our aim is to provide an illustrative example, we do not consider the most general assumptions. We refer to \cite[Lemma 4.6]{OS2024} for a result dealing with the existence and uniqueness of solutions to \eqref{eq:HJ}.

In the context of stochastic optimal control problems, it is well known that \eqref{eq:HJ} is the equation satisfied by the value function of an infinite horizon discounted problem. In this case the Hamiltonian takes the form
\[
  H(x,p) = \sup_{a \in A} \left \{- a \cdot p - \ell(x,a) \right \},
\]
where $A \subset \RR^d$ is the set of controls and $\ell \colon \RR^d \times \RR^d \to \RR$ is a continuous running cost. In the case where $A$ is compact and convex and $\ell$ is strictly convex with respect to $a$, then the supremum is achieved by a unique element in $A$ for every $(x,p) \in \RR^d \times \RR^d$. It then follows from Danskin's theorem \cite{Danskin1966} that $H$ is of class $C^1$ with respect to $p$ and the standard version of the BRR theorem applies. Below we show that our generalization allows to consider cases where Danskin's theorem does not apply.

\subsubsection{Setting of the problem}

The assumptions are the following.
\begin{enumerate}[label={\bf(HJ\arabic*)}]
 \item \label{h:HJ_data} We assume that $\lambda \geq 0$, that $f \in L^2(\Omega)$ and that $H \colon \bar \Omega \times \RR^d \to \RR$ is continuous and satisfies
\begin{equation}\label{eq:Hamiltonian_Lipschitz}
 \module{H(x,z_1) - H(x,z_2)} \leq C_H \module{z_1 - z_2} \quad \textnormal{for all } z_1, z_2 \in \RR^d, \, x \in \Omega
\end{equation}
for some $C_H > 0$.

\item \label{h:HJ_reg} We assume that $\Omega$ is such that, for every $g \in L^2(\Omega)$, the unique weak solution $v \in H^1_0(\Omega)$ to
\begin{equation}\label{eq:poisson}
\begin{cases}
 - \Delta v(x) + \lambda v(x) = g(x) \quad & \textnormal{in } \Omega, \\
 v(x) = 0 \quad & \textnormal{on } \partial \Omega,
 \end{cases}
\end{equation}
belongs to $H^2(\Omega)$ and there exists $K_\Omega > 0$ such that $\norm{v}{H^2} \leq K_\Omega \norm{g}{L^2}$.
\end{enumerate}

\begin{rem}\,
\begin{enumerate}
 \item As a consequence of \eqref{eq:Hamiltonian_Lipschitz}, we get
\begin{equation}\label{eq:Hamiltonian_clarke_growth}
 \xi \in \partial^C H(x,z) \Longrightarrow \module{\xi} \leq C_H,
\end{equation}
where Clarke's subdifferential is considered only with respect to the second variable.

\item Our results may be extended to the case of locally Lipschitzian Hamiltonians under the growth conditions
\begin{gather*}
  \module{H(x,p)} \leq C_H \left(1 + \module{p}^2 \right), \\
  \xi \in \partial^C H(x,p) \Longrightarrow \module{\xi} \leq C_H\left(1 + \module{p} \right).
 \end{gather*}

\item  It is known that \ref{h:HJ_reg} holds for instance when the boundary of $\Omega$ has $C^{1,1}$ regularity \cite[Theorem 2.2.2.3]{G1985}, $\Omega$ is convex \cite[Theorem 3.2.1.2]{G1985} or for $d=2$ with $\Omega$ polygonal \cite[Theorem 4.3.1.4]{G1985}.
\end{enumerate}
\end{rem}

Notice that, up to increasing the size of the constant $C_H$ in \eqref{eq:Hamiltonian_Lipschitz}, we have
\[
 \module{H(x,z)} \leq C_H \left ( 1 + \module{z} \right) \quad \textnormal{for all } (x,z) \in  \Omega \times \RR^d.
\]
 In turn, from \eqref{eq:Hamiltonian_Lipschitz}, \eqref{eq:Hamiltonian_clarke_growth} and the continuity of $H$, we deduce that \ref{item:nemytskii_lip} and \ref{item:nemytskii_growth} hold for any $p \geq q$, where we choose $\Omega$ endowed with Lebesgue $\sigma$-algebra and the $d$-dimensional Lebesgue measure for the measure space. Thus, it follows from \cref{prop:nemytskii_continous} that the Nemytskii operator $\gothH \colon L^p(\Omega; \RR^d) \to L^q(\Omega; \RR)$, defined by \eqref{eq:nemytskii}, is well-defined and continuous for every $1 < q \leq p < \infty$. In particular, using Sobolev's inequality, if we choose $q \in (1,2]$ if $d = 1,2$ and $q \in (\frac{2d}{d+2}, 2]$ for $d \geq 3$, we have that $\gothH \colon L^2(\Omega;\RR^d) \to H^{-1}(\Omega)$ is well defined and continuous, where $H^{-1}(\Omega)$ denotes the dual space of $H^1_0(\Omega)$.

Let us reformulate \eqref{eq:HJ}. We define the continuous mapping $G \colon H_0^1(\Omega) \to H^{-1}(\Omega)$ by
\[
 G(v) = \gothH[Dv] - f
\]
and we consider $T \in \cL(H^{-1}(\Omega), H^1_0(\Omega))$ defined by $Tg = v$ where $v$ solves \eqref{eq:poisson}. Defining $F \colon H_0^1(\Omega) \to H_0^1(\Omega)$ by $F := \left (I + T \circ G \right)$, we have that $u \in H^1_0(\Omega)$ is a weak solution to \eqref{eq:HJ} if and only if $F(u) = 0$.

We now introduce approximations of $F$. We consider a collection $(V_h)_{h>0}$ of finite dimensional $H^1_0(\Omega)$-conformal approximation spaces, i.e. such that $V_h \subset H^1_0(\Omega)$ for every $h>0$. The goal is, given a solution $u \in H^2(\Omega)$ to \eqref{eq:HJ}, to find a finite dimensional approximation $u_h \in V_h$, close to $u$ in $H^1(\Omega)$, and satisfying
\begin{equation}\label{eq:HJ_galerkin}
 \int_\Omega D u_h(x) \cdot D \phi_h(x) + H(x, Du_h(x)) \phi_h(x) + \lambda u_h(x) \phi_h(x) dx = \int_\Omega f(x) \phi_h(x) dx  \quad \textnormal{for all } \phi_h \in V_h.
\end{equation}
We assume that the approximation spaces satisfy the following properties :
\begin{align}
 \lim_{h \to 0} \inf_{v_h \in V_h} \norm{v - v_h}{H^{1}_0} = 0 \quad &  \textnormal{for all } v \in H^1_0(\Omega), \label{eq:approximation_property} \\
\inf_{v_h \in V_h} \norm{v - v_h}{H^{1}_0} \leq \eta (h) \norm{v}{H^2} \quad & \textnormal{for all } v \in H^2(\Omega)\label{eq:strong_approximation_property},
\end{align}
where we assume that $\eta$ is increasing and that $\lim_{h \to 0} \eta(h) = 0$.

\begin{rem}
 Assumptions \eqref{eq:approximation_property} and \eqref{eq:strong_approximation_property} hold for instance when $d \leq 3$, $\Omega$ is polyhedral and $V_h$ is the finite element space induced by $\PP^k$-Lagrange finite elements on a collection $(\cT_h)_{h>0}$ of shape regular\footnote{See \cite[Definition 1.107]{EG2004}.} triangulations of $\Omega$, where each $\cT_h$ is made up of elements of diameter at most $h$. In this case we have $\eta(h) = O(h)$ \cite[Corollary 1.110]{EG2004}.
\end{rem}

For every $h>0$, using the Lax-Milgram theorem, we define a linear operator $T_h$ by setting $T_h g = v_h$, for every $g \in H^{-1}(\Omega)$, where $v_h \in V_h$ is such that
\[
 \int_\Omega D v_h \cdot D \phi_h + \lambda v_h \phi_h \ dx = \langle g, \phi_h \rangle_{H^{-1},H^{1}_0} \quad \textnormal{for all } \phi_h \in V_h.
\]
From \cite[Proposition 2.19]{EG2004} we know that $T_h \in \cL(H^{-1}(\Omega),H^1_0(\Omega))$ with
\[
 \norm{T_h}{\cL(H^{-1}(\Omega),H^1_0(\Omega))} \leq 2 \max \{1, C_P(\Omega) \} \quad \textnormal{for all } h > 0,
\]
where $C_P(\Omega) > 0$ is the Poincaré constant associated to $\Omega$.
Defining $F_h \colon H_0^{1}(\Omega) \to H_0^1(\Omega)$ by $F_h : = I + T_h \circ G$, we notice that finding $u_h \in V_h$ such that \eqref{eq:HJ_galerkin} holds is equivalent to solving $F_h(v) = 0$ in $V_h$. In order to achieve this we are therefore going to apply \cref{cor:first_example} with $V = H^{1}_0(\Omega)$, $W = H^{-1}(\Omega)$, and $F$ and $F_h$ defined as above.

\subsubsection{Application of \cref{cor:first_example}}

First, it follows from \eqref{eq:Hamiltonian_Lipschitz} and the continuous embedding $L^q(\Omega; \RR) \hookrightarrow H^{-1}(\Omega)$ that $G$ is (globally) Lipschitz continuous and we denote by $L_G$ its Lipschitz constant. The same continuous embedding, with $q$ as above but different from $2$, and \cref{thm:nemytskii_prediff}\footnote{The restriction on $q$ comes from the fact that \cref{thm:nemytskii_prediff} fails in general for $p = q = 2$, see \cref{rem:nemytskii_prediff}.} also imply that, for any $\epsilon > 0$, there is $\delta > 0$ such that, for any $v_1, v_2 \in B_{H^1(\Omega)}(u,\delta)$, there is $\xi \in \partial^\circ \gothH[u]$ such that
\begin{equation}\label{eq:G_prediff}
 \norm{G(v_1) - G(v_2) - \xi \cdot (Dv_1 - Dv_2)}{H^{-1}} < \epsilon \norm{v_1 - v_2}{H^1}.
\end{equation}

For a solution $u \in H^2(\Omega) \cap H^1_0(\Omega)$ to \eqref{eq:HJ} and $\xi \in \partial^\circ \gothH [u]$, where $\partial^\circ \gothH[u]$ is defined in \eqref{eq:Nemytskii_subdiff}, the linear problem
\begin{equation}\label{eq:HJ_lin}
 \begin{cases}
  - \Delta v + \xi \cdot Dv + \lambda v = 0 \quad & \textnormal{in } \Omega, \\
  v = 0 \quad & \textnormal{on } \partial \Omega.
 \end{cases}
\end{equation}
will play a major role to prove metric regularity.

\begin{prop}[{\cite[Corollary 8.2]{GT2001}}]\label{prop:hj_lin_injective}
 Assume \ref{h:HJ_data} and let $u \in H^1(\Omega)$. For every $\xi \in \partial^\circ \gothH [u]$, $v = 0$ is the unique weak solution to \eqref{eq:HJ_lin}.
\end{prop}

Notice that any element $\xi \in \partial^\circ \gothH [u] \subset L^\infty(\Omega)$ may be identified with an element on $\cL(H^1(\Omega), L^2(\Omega))$ by setting $\xi(v) = \xi \cdot Dv$.
With this identification we define the collection $\cA \subset \cL(H^1_0(\Omega), H^1_0(\Omega))$ as
\begin{equation}\label{eq:HJ_cA}
 \cA = \left \{ I + T \circ \xi : \xi \in \partial^\circ \gothH [Du] \right \}.
\end{equation}
From \cref{prop:nemytskii_basic_properties} and \eqref{eq:Hamiltonian_clarke_growth}, we deduce that that $\cA$ is bounded and convex.

\begin{lem}\label{lem:HJ_metric_regular}
 Assume \ref{h:HJ_data}, \ref{h:HJ_reg}, let  $u \in H^1(\Omega)$ and let $\cA$ be defined by \eqref{eq:HJ_cA}. Then each $A \in \cA$ is an isomorphism and the exists $\kappa > 0$ such that $\inf_{A \in \cA} \gothC(A) \geq \kappa$.
\end{lem}

\begin{proof}
Notice that for every $A \in \cA$, where $A = I + T \circ \xi$ with $\xi \in \partial^\circ \gothH[Du]$, we have $v \in \ker A$ if and only if $v$ is a weak solution to \eqref{eq:HJ_lin} . It follows from \cref{prop:hj_lin_injective} that we have $\ker A = \{ 0 \}$ for every $A \in \cA$. Since $T$ is a compact operator, a simple application of Fredholm's alternative then shows that $A$ is an isomorphism for every $A \in \cA$. We claim that there exists a positive constant $\kappa > 0 $, depending on $C_H$ and $\lambda$, such that
\[
 \norm{A^{-1}}{\cL(H^1_0,H^1_0)} \leq \kappa^{-1} \quad \textnormal{for all } A \in \cA.
\]
Indeed, let $v,w \in H^1_0(\Omega)$ be such that $w = (I + T \circ \xi)(v)$. This is equivalent to
\[
 -\Delta v + \xi \cdot Dv + \lambda v = (- \Delta + \lambda I)w.
\]
It then follows from \cite[Corollary 8.7]{GT2001} that
\[
 \norm{v}{H^1} \leq C \norm{(- \Delta + \lambda I)w}{H^{-1}} \leq C'\norm{w}{H^1},
\]
where the constants depend on $C_H$ and $\lambda$. The claim is proved by setting $\kappa^{-1} = C'$.
In particular this implies $\gothC(A) = \norm{A^{-1}}{\cL(H^1_0,H^1_0)}^{-1} \geq \kappa$ for all $A \in \cA$.
\end{proof}

We are now in position to apply \cref{cor:first_example}, which yields the following theorem.

\begin{thm}\label{thm:approximation_HJB}
 Assume \ref{h:HJ_data}, \ref{h:HJ_reg} and let $(V_h)_{h > 0}$ be a collection of $H^1_0(\Omega)$-conformal approximation spaces, satisfying \eqref{eq:approximation_property} and \eqref{eq:strong_approximation_property}. Then, for any weak solution $u \in H^2(\Omega) \cap H^1_0(\Omega)$ to \eqref{eq:HJ}, there exists $h_0 > 0$ and a neighborhood $\cO$ of $u$ in $H^1_0(\Omega)$ such that, for every $0 < h \leq h_0$, there is a unique $u_h \in \cO \cap V_h$ satisfying \eqref{eq:HJ_galerkin} and we have the error estimate
 \[
  \norm{u - u_h}{H^1} \leq C \inf_{v_h \in V_h} \norm{u - v_h}{H^1}.
 \]
 In particular
 \[
  \norm{u - u_h}{H^1} = O(\eta(h)).
 \]
\end{thm}

\begin{proof}
  From \cref{lem:HJ_metric_regular} and \eqref{eq:G_prediff} we deduce that \cref{cor:first_example}-\ref{item:cor:first_example_ii} is satisfied. Moreover, using Céa's lemma \cite[Lemma 2.28]{EG2004} we deduce
\begin{equation}\label{eq:Cea}
 \norm{(T - T_h)g}{H^1} \leq C \inf_{v_h \in V_h} \norm{(Tg) - v_h}{H^{1}_0} \quad \textnormal{for all } g \in H^{-1}(\Omega),\, h > 0,
\end{equation}
where the constant $C$ is independent of $h$. Combining \eqref{eq:approximation_property} and \eqref{eq:Cea} we obtain
\begin{equation}
 \lim_{h \to 0} \norm{(T-T_h)g}{H^1} = 0 \quad \textnormal{for all } g \in H^{-1}(\Omega),
\end{equation}
so that \cref{cor:first_example}-\ref{item:cor:first_example_i} holds. Set, for each $h > 0$, $\Phi_h := (T-T_h) \circ G$. Then, for any $w_1 , w_2 \in H^{1}_0(\Omega)$, using the fact that $G(w_1) - G(w_2) \in L^2(\Omega)$, \eqref{eq:Cea}, \eqref{eq:strong_approximation_property} and \ref{h:HJ_reg}, we have
\begin{align*}
 &\norm{\Phi_h(w_1) - \Phi_h (w_2)}{H^1} = \norm{(T - T_h)(G(w_1) - G(w_2))}{H^1} \\
 & \qquad \leq C \inf_{v_h \in V_h} \norm{T(G(w_1) - G(w_2)) - v_h}{H^1}
 \leq C K_\Omega \eta(h) \norm{G(w_1) - G(w_2)}{L^2} \\
 & \qquad  \leq C K_\Omega L_G \eta(h) \norm{w_1 - w_2}{H^1}.
\end{align*}
It follows that \cref{cor:first_example}-\ref{item:cor:first_example_iii} holds with $L(h,r) = C K_\Omega L_G \eta(h)$. Hence, all the assumptions of \cref{cor:first_example} are satisfied. It follows that there exists a neighborhood $\mathcal{O}$ of $u$ in $H^1_0(\Omega)$ and a unique $u_h \in \mathcal{O}$ such that $F_h(u_h) = 0$. Moreover we have the estimate
\[
 \norm{u - u_h}{H^1} \leq \norm{\left(T - T_h \right) G(u)}{H^1}.
\]
The conclusion then follows from Céa's lemma \cite[Lemma 2.28]{EG2004} and \eqref{eq:strong_approximation_property}.
\end{proof}

\subsection{Second order stationary mean field games}

Our second application concerns mean field games, and more precisely the following stationary system
\begin{equation}\label{eq:mfg}
\begin{cases}
 - \Delta u + H(x,Du) + \lambda u = F[m] \quad &  \textnormal{in } \Omega, \\
 - \Delta m - \diver \left (m H_p(x,Du) \right ) + \lambda m = \lambda m_0 \quad & \textnormal{in } \Omega, \\
  u = m = 0 \quad & \textnormal{on } \partial \Omega, \\
\end{cases}
\end{equation}
where $\Omega \subset \RR^d$ is a bounded Lipschitz domain, $H \colon \Omega \times \RR^d \to \RR$, $H_p(x,p) := \partial_p H(x,p)$, $F \colon L^2(\Omega) \to L^2(\Omega)$, $\lambda > 0$ and $m_0 \in L^2(\Omega)$. Mean field games systems similar to \eqref{eq:mfg} were introduced in \cite{LL2007} to study Nash equilibria of large symmetric dynamics games. System \eqref{eq:mfg} was first studied in \cite{BFY2013} and we refer the reader to \cite{BFY2013,BLS2025} for its interpretation. We also refer to \cite[Theorem 3.3]{OS2024} for a result dealing with the existence of solutions to \eqref{eq:mfg}.

As discussed at the beginning of \cref{section:HJ}, for an Hamiltonian of the form
\[
  H(x,p) = \sup_{a \in A} \left \{- a \cdot p - \ell(x,a) \right \},
\]
if $A$ compact and convex and $\ell$ is strictly convex with respect to $a$, then $H$ is of class $C^1$ with respect to $p$. However, we do not expect $C^2$ regularity in general, so that the standard BRR theorem does not apply because the second equation in \eqref{eq:mfg} is not differentiable. However, $C^{1,1}$ regularity may be achieved if $\ell$ is uniformly strongly convex \cite[Example 4.34]{BS2000}, for instance. Below, we show that our generalization may be applied in this setting.

\subsubsection{Setting of the problem}

Let us start by stating our assumptions for this application.

\begin{enumerate}[label={\bf{(MFG\arabic*)}}]
 \item \label{item:MFG_1} We assume that the Hamiltonian $H$ convex with respect to the second variable and of class $C^{1,1}$, \textit{i.e.}, differentiable with respect to the second variable with both $H$ and $H_p$ continuous, and that there exists $C_H > 0$ such that
\begin{align}
 \module{H_p(x,p)} & \leq C_H \quad \textnormal{for all } (x,p) \in \Omega \times \RR^d, \\
 \module{H_p(x,p) - H_p(x,q)} & \leq C_H \module{p - q} \quad \textnormal{for all } x \in \Omega,\, p,q \in \RR^d. \label{eq:mfg_hamiltonian_lipschitz}
\end{align}
In addition, we also assume that $m_0 \in L^2(\Omega)$, with $m_0 \geq 0$ and $m_0 \neq 0$, and that $F$ is Lipschitz continuous.

 \item \label{item:MFG_3} We assume that, for each $r > 2$, the mapping $F$ is continuously differentiable from $L^r(\Omega)$ to $L^2(\Omega)$ and $\norm{dF[m]}{\cL(L^r,L^2)} \leq C_F$ for all $m \in L^r(\Omega)$.
\end{enumerate}

\begin{rem}
 Our results may be extended to the case of Hamiltonians of class $C^{1,1}_{\loc}$ under the growth conditions
 \begin{gather*}
  \module{H(x,p)} \leq C_H \left(1 + \module{p}^2 \right), \\
  \module{H_p(x,p)} \leq C_H\left(1 + \module{p} \right), \\
  \xi \in \partial^C H_p(x,p) \Longrightarrow \module{\xi} \leq C_H.
 \end{gather*}
 In the case of smooth Hamiltonians this setting was considered in the first author's PhD thesis \cite[Chapter 3]{B_thesis} and the special case $H(x,p) = \frac{\module{p}^2}{2}$ was considered in \cite{BLS2025}.
\end{rem}

\begin{rem}
 Assumption \ref{item:MFG_3} holds for instance in the case of a local coupling $F[m](x) = f(x,m(x))$ where $f \colon\Omega \times \RR \to \RR$ is a Carathéodory function wich is $C^1$ with respect to the second variable and satisfies
 \begin{align*}
  \module{\partial_m f(x,m)} & \leq C_F
 \end{align*}
 for every $(x,m) \in \Omega \times \RR$. Similarly, the assumption also holds for a nonlocal coupling
 \[
  F[m](x) = f(x,k*m(x))
 \]
 where $k$ is a smooth convolution kernel and $f$ is as above.
\end{rem}

\begin{defi}
 A pair $(u,m) \in H_0^1(\Omega) \times L^2(\Omega)$ is a weak solution to \eqref{eq:mfg} if
 \begin{equation}
 \int_\Omega Du(x) \cdot D \phi(x) + H(x,Du) \phi(x) +\lambda u(x) \phi(x) \, dx = \int_\Omega F[m](x) \phi(x) \, dx
\end{equation}
for all $\phi \in H_0^1(\Omega)$ and
\begin{equation}
 \int_{\Omega} \left (-\Delta \psi(x) + H_p(x,Du) \cdot D \psi(x) + \lambda \psi(x) \right ) m(x) \, dx = \lambda \int_\Omega m_0(x) \psi(x) \, dx
\end{equation}
for all $\psi \in H^2(\Omega) \cap H^1_0(\Omega)$.
\end{defi}

For each $s \in (0,1]$, we define the Hilbert space
\[
 \cV_s := H^{1+s}(\Omega) \cap H^1_0(\Omega)
\]
endowed with the $H^{1+s}(\Omega)$ norm.
We recall the following fact.
\begin{lem}\label{lem:very_weak_isom}
 Assume \ref{h:HJ_reg}. For every {$f \in \cV_1'$} there exists a unique $u \in L^2(\Omega)$ such that
 \begin{equation}\label{eq:very_weak_formulation}
  \int_{\Omega} \left( - \Delta \psi + \lambda \psi \right ) u \, dx = \langle f, \psi \rangle_{\cV_1',\cV_1} \quad \textnormal{for all } \psi \in \cV_1.
 \end{equation}
 Moreover, there exists $C >0$ such that
 \[
  \norm{u}{L^2} \leq C \norm{f}{\cV_1'}.
 \]
\end{lem}

\begin{proof}
 We first prove uniqueness. Let $u \in L^2(\Omega)$ be such that
 \begin{equation}\label{eq:very_weak}
  \int_{\Omega} \left ( - \Delta \phi(x) + \lambda \phi(x) \right ) u(x) \, dx = 0 \quad \textnormal{for all } \phi \in \cV_1.
 \end{equation}
 Let $\xi \in L^2(\Omega)$ and consider, using \ref{h:HJ_reg}, the unique solution $\phi \in H^2(\Omega) \cap H^1_0(\Omega) = \cV_1$ to $- \Delta \phi + \lambda \phi = \xi$ with homogeneous Dirichlet boundary condition. From \eqref{eq:very_weak} we deduce that
 \[
  \int_{\Omega} \xi(x) u(x) \, dx = 0.
 \]
 Since $\xi$ is arbitrary, we conclude that $u = 0$. This proves uniqueness.

 We now turn to the proof of existence. Since $L^2(\Omega)$ is dense in $\cV_1$, we also have that $L^2(\Omega) = \left(L^2(\Omega)\right)'$ is dense in $\cV_1'$.
 Let $(f_n)_{n \in \NN}$ be a sequence in $L^2(\Omega)$ converging to $f$ in $\cV_1'$. Using \ref{h:HJ_reg}, for each $n \in \NN$, let $u_n$ be the unique element in $H^2(\Omega) \cap H^1_0(\Omega) = \cV_1$ such that $-\Delta u_n + \lambda u_n = f_n$. Let also $\xi \in L^2(\Omega)$ and consider $\phi \in H^2(\Omega) \cap H^1_0(\Omega)$ such that $-\Delta \phi + \lambda \phi = \xi$. From \ref{h:HJ_reg}, we have $\norm{\phi}{\cV_1} \leq C \norm{\xi}{L^2}$. Using $u_n$ as a test-function in the equation satisfied by $\phi$ we deduce that
 \[
  \int_{\Omega} \xi(x) u_n(x) \, dx = \int_{\Omega} f_n(x) \phi(x) \, dx \leq \norm{f_n}{\cV_1'} \norm{\phi}{\cV_1} \leq C \norm{f_n}{\cV_1'} \norm{\xi}{L^2}.
 \]
 By duality we conclude that $\norm{u_n}{L^2} \leq C \norm{f_n}{\cV_1'}$. It follows that $u_n$ converges weakly, up to a subsequence, to a some $u \in L^2(\Omega)$ satisfying \eqref{eq:very_weak_formulation}. This concludes the proof.
\end{proof}

Let $(V_h)_{h>0}$ be a collection of $H^1_0(\Omega) \cap W^{1,\infty}(\Omega)$-conformal approximation spaces and assume that
\begin{align}
 \lim_{h \to 0} \inf_{v_h \in V_h} \norm{v - v_h}{H^1} = 0 \quad &  \textnormal{for all } v \in H^1(\Omega), \label{eq:approximation_property_mfg} \\
\inf_{v_h \in V_h} \norm{v - v_h}{H^1} \leq Ch \norm{v}{H^2} \quad & \textnormal{for all } v \in H^2(\Omega)\label{eq:strong_approximation_property_mfg}, \\
\inf_{v_h \in V_h}\norm{v - v_h}{L^2} \leq C h \norm{v}{H^1} \quad & \textnormal{for all } v \in H^1(\Omega),\label{eq:aubin_nitsche}
\end{align}
for some $C > 0$.
We also assume that $(V_h)_{h > 0}$ satisfy the inverse estimate
\begin{equation}\label{eq:inverse_estimate}
 \norm{v_h}{W^{k,p}} \leq Ch^{l - k + d\min{\{0,1/p - 1/q\}}} \norm{v_h}{W^{l,q}}
\end{equation}
for all $k,l \in \{0 , 1\}$ with $l \leq k$, $1 \leq p,q \leq \infty$ and $v_h \in V_h$. Finally, we assume that, for all $h >0$, there exists a linear operators $\cI_h \colon W^{1,1}(\Omega) \to V_h$ such that
\begin{equation}\label{eq:approximation_interpolation}
 \norm{\cI_h v}{W^{k,r}} \leq C \norm{v}{W^{k,r}} \quad \textnormal{for all } v \in W_0^{1,r}(\Omega)
\end{equation}
for every $k \in \{0,1\}$ and $r \in [0,+\infty)$ and such that $\cI_h v_h = v_h$ for all $v_h \in V_h$.

\begin{rem}
The main example of spaces $(V_h)_{h>0}$ satisfying \eqref{eq:approximation_property_mfg}, \eqref{eq:strong_approximation_property_mfg}, \eqref{eq:aubin_nitsche}, \eqref{eq:inverse_estimate} and \eqref{eq:approximation_interpolation} is provided by $\PP^k$-Lagrange finite element spaces on quasi-uniform meshes. Indeed, let $\cI_h$ denote the Scott-Zhang interpolation operator \cite{SZ1990}, which is a projection operator onto $V_h$ according to \cite[Theorem 2.1]{SZ1990}. The stability property \eqref{eq:approximation_interpolation} is then given in \cite[Corollary 4.1]{SZ1990}. Moreover, \eqref{eq:approximation_property_mfg}, \eqref{eq:strong_approximation_property_mfg} and \eqref{eq:aubin_nitsche} follow from \cite[Theorem 4.1]{SZ1990}. Finally, the inverse estimate \eqref{eq:inverse_estimate} can be found in \cite[Corollary 1.141]{EG2004}.
\end{rem}

 The goal is to find approximations $(u_h,m_h) \in V_h \times V_h$ satisfying
\begin{gather}
 \int_\Omega Du_h \cdot D \phi_h + H(x,Du_h) \phi_h +\lambda u_h \phi_h \, dx = \int_\Omega F[m_h](x) \phi_h(x) \, dx \\
 \int_{\Omega} D m_h \cdot D \psi_h + m_h H_p(x,Du_h) \cdot D \psi_h + \lambda m_h \psi_h \, dx = \int_\Omega m_0(x) \psi_h(x) \, dx
\end{gather}
for all $(\phi_h, \psi_h) \in V_h \times V_h$.
\\

Using the Lax-Milgram theorem, we may define a linear operator $S_1 \in \cL(H^{-1}(\Omega), H_0^1(\Omega))$ by setting $S_1f = u$ where $u \in H_0^1(\Omega)$ is the unique solution to
\begin{equation}\label{eq:S1}
 \int_{\Omega} Du \cdot D\phi + \lambda u \phi\, dx = \langle f, \phi \rangle_{H^{-1}(\Omega), H^1(\Omega)} \quad \textnormal{for all } \phi \in H_0^1(\Omega).
\end{equation}
Similarly, using \cref{lem:very_weak_isom}, we may also define a linear operator $S_2 \in \cL(\cV_1', L^2(\Omega))$ by $S_2 g = v$, with $v \in L^2(\Omega)$ being the unique solution to
\begin{equation}\label{eq:S2}
 \int_{\Omega} \left( - \Delta \psi(x) + \lambda \psi(x) \right) v(x) \, dx = \langle g, \psi \rangle \quad \textnormal{for all } \psi \in \cV_1.
\end{equation}
Notice that $S_1 = S_2$ on $H^{-1}(\Omega)$. We then define $T \in \cL(H^{-1}(\Omega) \times \cV_1', H^1(\Omega) \times L^2(\Omega))$ by $T(f,g) = (S_1f, S_2g)$. We also define $G \colon H^1(\Omega) \times L^2(\Omega) \to L^2(\Omega) \times H^{-1}(\Omega)$ by
\[
 G(v,\rho) := \begin{pmatrix}
               H(\cdot, Dv) - F[\rho] \\
               -\diver \left (\rho H_p(\cdot, Dv) \right ) - \lambda m_0
              \end{pmatrix}.
\]
It is then clear that $(u,m) \in H^1(\Omega) \times L^2(\Omega)$ is a weak solution to \eqref{eq:mfg} if and only if
\begin{equation}
\Upsilon (u,m) := \left (I + T \circ G \right)(u,m) = 0.
\end{equation}

Let $2 < r < \infty$ and let $\mathfrak{H} \colon L^r(\Omega;\RR^d) \to L^2(\Omega)$ be the Nemytskii operator associated to $H$. Notice that $H_p$ is a Carathéodory function and that we can also define the corresponding Nemytskii operator $\mathfrak{H_p} \colon L^r(\Omega, \RR^d) \to L^2(\Omega, \RR^d)$ by setting
\[
 \mathfrak{H_p}[w](x) = H_p(x,w(x)) \quad \textnormal{for } x \in \Omega.
\]
Following \cref{section:Nemytskii}, we denote by $\partial^\circ \mathfrak{H_p}$ its generalized differential.

\begin{lem}\label{lem:mfg_nemytskii}
 Let $r > 2$. The Nemytskii operator $\mathfrak{H}$ is continuously differentiable from $L^r(\Omega,\RR^d)$ to $L^2(\Omega)$ with
\[
 d \mathfrak{H}[w] = \mathfrak{H_p}[w] \quad \textnormal{for every } w \in L^r(\Omega,\RR^d).
\]
Moreover, $\partial^\circ \mathfrak{H_p}[w]$ is bounded in $L^\infty(\Omega, \RR^{d \times d})$ for every $w \in L^r(\Omega,\RR^d)$.
\end{lem}
\begin{proof}
 The differentiability of the Nemytskii operator follows from \cite[Theorem 3.12]{AZ1990} and the boundedness of $\partial^\circ \mathfrak{H_p}[w]$ in $L^\infty(\Omega, \RR^{d \times d})$ is a consequence of \eqref{eq:mfg_hamiltonian_lipschitz}.
\end{proof}

\subsubsection{Application of \cref{cor:first_example}}

Assuming that $(u,m)$ is a weak solution to \eqref{eq:mfg} with $m \in L^\infty(\Omega)$ and let $r > 2$, we may define the subset
\[
 \cA[u,m] \subset \cL \left (W_0^{1,r}(\Omega) \times L^r(\Omega), L^2(\Omega) \times H^{-1}(\Omega) \right )
\]
by
\begin{equation}\label{eq:mfg_def_prediff}
 \cA[u,m] := \left \{ A(v,\rho) = \begin{pmatrix}
                             \mathfrak{H_p}[Du] \cdot Dv - dF[m](\rho)  \\
                            - \diver \left( \rho \mathfrak{H_p}[Du] + m \xi Dv \right)
                         \end{pmatrix}
                         : \, \xi \in \partial^\circ \mathfrak{H_p}[Du] \right \}.
\end{equation}
Notice that $\cA[u,m]$ is convex and bounded.

\begin{lem}\label{lem:mfg_prediff}
 Assume that \ref{item:MFG_1} and \ref{item:MFG_3} hold and consider $4 \leq r \leq 6$ and $(u,m) \in W_0^{1,r}(\Omega) \times L^\infty(\Omega)$. Then, for every $\epsilon > 0$, there exists $\delta > 0$ such that, for all $(u_1,m_1), (u_2,m_2) \in B_{W_0^{1,r} \times L^r}\left ((u,m), \delta \right)$, there exists $A \in \cA[u,m]$ satisfying
 \[
  \norm{G(u_1,m_1) - G(u_2,m_2) - A(u_1 - u_2,m_1 - m_2)}{L^2 \times H^{-1}} \leq \epsilon \norm{(u_1,m_1) - (u_2,m_2)}{W^{1,r} \times L^r}.
 \]
\end{lem}

\begin{proof}
 Let $\epsilon > 0$ and fix $(u,m)$ as in the statement. Then, using \cref{thm:nemytskii_prediff}, there exists $\delta_1 > 0$ such that, for every $u_1,u_2 \in B_{W_0^{1,r}}(u,\delta_1)$, there is $\xi \in \partial^\circ \mathfrak{H_p}[Du]$ such that
 \begin{align}
  \norm{\mathfrak{H_p}[Du_1] - \mathfrak{H_p}[Du_2] - \xi(Du_1 - Du_2)}{L^2} & \leq \frac{\epsilon}{5 (1 + \norm{m}{L^\infty})} \norm{Du_1 - Du_2}{L^r} \nonumber \\
  & \leq \frac{\epsilon}{5 (1 + \norm{m}{L^\infty})} \norm{u_1 - u_2}{W^{1,r}}.\label{eq:mfg_prediff_1}
 \end{align}
 Moreover, from \cref{lem:mfg_nemytskii}, $\mathfrak{H}$ is continuously differentiable with $d\mathfrak{H}[Du] = \mathfrak{H_p}[Du]$, it is in particular strictly differentiable, so that there exists $\delta_2 > 0$ such that
 \begin{equation}\label{eq:mfg_prediff_2}
  \norm{\mathfrak{H}[Du_1] - \mathfrak{H}[Du_2] - \mathfrak{H_p}[Du] \cdot \left (Du_1 - Du_2 \right)}{L^2} \leq \frac{\epsilon}{5} \norm{u_1 - u_2}{W^{1,r}}
 \end{equation}
 for every $u_1, u_2 \in B_{W_0^{1,r}}(u,\delta_2)$. Similarly, using \ref{item:MFG_3}, there is $\delta_3 > 0$ such that
 \begin{equation}\label{eq:mfg_prediff_3}
  \norm{F[m_1] - F[m_2] - dF[m](m_1 - m_2)}{L^2} \leq \frac{\epsilon}{5} \norm{m_1 - m_2}{L^r}
 \end{equation}
 for every $m_1, m_2 \in B_{L^r}(m,\delta_3)$. We have that
 \begin{align}
  & \norm{\diver \left (m_1 \mathfrak{H_p}[Du_1] - m_2 \mathfrak{H_p}[Du_2] - (m_1 - m_2) \mathfrak{H_p}[Du] - m \xi (Du_1 - Du_2) \right )}{H^{-1}} \nonumber \\
  & \qquad \leq \norm{m_1 \mathfrak{H_p}[Du_1] - m_2 \mathfrak{H_p}[Du_2] - (m_1 - m_2) \mathfrak{H_p}[Du] - m \xi (Du_1 - Du_2)}{L^2} \nonumber \\
  & \qquad  \leq \norm{(m_1 - m_2) \left( \mathfrak{H_p}[Du_1] - \mathfrak{H_p}[Du]\right )}{L^2} + \norm{(m_2 - m) \left ( \mathfrak{H_p}[Du_1] - \mathfrak{H_p}[Du_2] \right )}{L^2} \\
  & \qquad \qquad + \norm{m \left(\mathfrak{H_p}[Du_1] - \mathfrak{H_p}[Du_2] - \xi(Du_1 - Du_2) \right )}{L^2}. \label{eq:mfg_prediff_4}
 \end{align}
  From \ref{item:MFG_1}, we get
  \[
    \norm{\mathfrak{H_p}[Dv_1] - \mathfrak{H_p}[Dv_2]}{L^r} \leq C_H \norm{Dv_1 - Dv_2}{L^r}.
  \]
 In particular, using Hölder's inequality and the fact that $r \geq 4$, we have
 \begin{equation}\label{eq:mfg_prediff_5}
 \begin{split}
  \norm{(m_1 - m_2) \left( \mathfrak{H_p}[Du_1] - \mathfrak{H_p}[Du]\right )}{L^2} & \leq \module{\Omega}^{\frac{r-2}{2r}} \norm{m_1 - m_2}{L^r} \norm{ \mathfrak{H_p}[Du_1] - \mathfrak{H_p}[Du]}{L^r} \\
  & \leq C_H \module{\Omega}^{\frac{r-2}{2r}} \norm{m_1 - m_2}{L^r} \norm{Du_1 -Du}{L^r} \\
  & \leq \frac{\epsilon}{5} \norm{m_1 - m_2}{L^r}
  \end{split}
 \end{equation}
 and
 \begin{equation}
 \label{eq:mfg_prediff_6}
 \begin{split}
  \norm{(m_2 - m) \left ( \mathfrak{H_p}[Du_1] - \mathfrak{H_p}[Du_2] \right )}{L^2} & \leq \module{\Omega}^{\frac{r-2}{2r}} \norm{m_2 - m}{L^r} \norm{ \mathfrak{H_p}[Du_1] - \mathfrak{H_p}[Du_2]}{L^r} \\
  & \leq C_H \module{\Omega}^{\frac{r-2}{2r}} \norm{m_2 - m}{L^r} \norm{Du_1 - Du_2}{L^r} \\
  & \leq \frac{\epsilon}{5} \norm{ Du_1 - Du_2}{L^r}
  \end{split}
 \end{equation}
 whenever $(u_1,m_1), (u_2,m_2) \in B_{W_0^{1,r} \times L^r}((u,m),\delta_4)$ with $\delta_4 := \left(C_H \module{\Omega}^{\frac{r-2}{2r}} \right)^{-1}$.
Setting $\delta = \min \left \{\delta_1, \delta_2, \delta_3, \delta_4\right \}$ and combining \eqref{eq:mfg_prediff_1}, \eqref{eq:mfg_prediff_2}, \eqref{eq:mfg_prediff_3}, \eqref{eq:mfg_prediff_4}, \eqref{eq:mfg_prediff_5} and \eqref{eq:mfg_prediff_6}, we conclude that
\begin{equation}
 \norm{G(u_1,m_1) - G(u_2,m_2) - A(u_1 - u_2,m_1 - m_2)}{L^2 \times H^{-1}} \leq \epsilon \norm{(u_1,m_1) - (u_2,m_2)}{W^{1,r} \times L^r}
\end{equation}
for every $(u_1,m_1), (u_2,m_2) \in B_{W_0^{1,r} \times L^r}((u,m),\delta)$.
\end{proof}

In the case where $d \leq 3$ we have the continuous embedding $H^2(\Omega) \times H^1(\Omega) \hookrightarrow W^{1,r}(\Omega) \times L^r(\Omega)$ for every $1 \leq r \leq 6$. Using \ref{h:HJ_reg}, we may define
\[
 S_1 \in \cL \left (L^2(\Omega), W_0^{1,r}(\Omega) \right) \quad \textnormal{and} \quad S_2 \in \cL \left (H^{-1}(\Omega), L^r(\Omega) \right )
\]
according to \eqref{eq:S1}-\eqref{eq:S2} and set
\[
 T \in \cL \left (L^2(\Omega) \times H^{-1}(\Omega), W_0^{1,r}(\Omega) \times L^r(\Omega) \right )
\]
as before.

\begin{lem}\label{lem:mfg_banach_constant}
 Assume that $d \leq 3$, that \ref{h:HJ_reg}, \ref{item:MFG_1} and \ref{item:MFG_3} hold, and choose $2 < r < 6$. Let $(u,m) \in H^2(\Omega) \times L^\infty(\Omega)$. If
 \begin{equation}\label{eq:mfg_kernel_trivial}
  \ker_{W_0^{1,r} \times L^r} \left (I + T \circ A \right) = \{0\} \quad \textnormal{for every } A \in \cA[u,m],
 \end{equation}
 then, for every $A \in \cA[u,m]$, the linear operator $I + T \circ A$ is surjective (hence bijective) from $W_0^{1,r}(\Omega) \times L^r(\Omega)$ onto itself and there exists $\kappa > 0$ such that
 \begin{equation}\label{eq:mfg_banach_const}
  \inf_{A \in \cA[u,m]} \mathfrak{C}(I + T \circ A) \geq \kappa.
 \end{equation}
 In particular
 \[
  \sup_{A \in \cA[u,m]} \norm{\left (I + T \circ A \right)^{-1}}{\cL(W_0^{1,r} \times L^r)} \leq \kappa^{-1}.
 \]
\end{lem}

\begin{proof}
  We recall that since $d \leq 3$, the embedding
  \[
   H^{2}(\Omega) \times H^{1}(\Omega) \hookrightarrow W^{1,r}(\Omega) \times L^r(\Omega)
  \]
  is compact according to the Rellich-Kondrachov theorem. Moreover, each operator $A \in \cA[u,m]$ maps $W^{1,r}(\Omega) \times L^r(\Omega)$ to $L^2(\Omega) \times \left(H^{1}(\Omega) \right)'$. We deduce that
  \[
   R(T \circ A) \subset H^{2}(\Omega) \times H^{1}(\Omega),
  \]
  so that $T\circ A$ is a bounded and compact linear operator from $W_0^{1,r}(\Omega) \times L^r(\Omega)$ into itself.
 The surjectivity of $I + T \circ A$ then follows from \eqref{eq:mfg_kernel_trivial} and Fredholm's alternative.

 For \eqref{eq:mfg_banach_const}, we claim that the set $\cA$ is sequentially compact in the weak operator topology of $\cL(X,Y)$, with
 \[
  X = W_0^{1,r}(\Omega) \times L^r(\Omega) \quad \textnormal{and} \quad Y:= L^2(\Omega) \times H^{-1}(\Omega).
 \]
Since $r > 2$, there exists $1 < q < \infty$ such that $1/r + 1/q = 1/2$. Since $\partial^\circ \mathfrak{H_p}[Du]$ is bounded in $L^q(\Omega;\RR^{d \times d})$, we deduce from \cref{prop:nemytskii_basic_properties} that $\partial^\circ \mathfrak{H_p}[Du]$ is sequentially compact for the weak topology of $\cL(L^r(\Omega,\RR^d), L^{2}(\Omega,\RR^d))$. Using \cref{prop:WOT_composition}, it follows that the linear operators
\[
 W_0^{1,r}(\Omega) \ni v \mapsto m \xi Dv \in L^2(\Omega) \quad \textnormal{for } \xi \in \cA[u,m]
\]
is sequentially compact in the weak operator topology. Another application of \cref{prop:WOT_composition} then implies the sequential compactness of the collection of linear operators
 \[
   W_0^{1,r}(\Omega) \ni v \mapsto \diver \left (m \xi Dv \right ) \in H^{-1}(\Omega) \quad \textnormal{for all } \xi \in \partial^\circ \mathfrak{H_p}[u]
 \]
 in the weak operator topology. The claim then easily follows.
 From \cref{prop:sufficient_banach_const}, we conclude that $\inf_{A \in \cA[u,m]} \gothC(I + T \circ A) \geq \kappa > 0$ which concludes the proof.
\end{proof}

Consider $S_h \in \cL(V_h', V_h)$ be defined by $S_h f = v_h$, where $v_h$ is the unique element in $V_h$ such that
\[
 \int_\Omega D v_h \cdot D \phi_h + \lambda v_h \phi_h \, dx = \left \langle f, \phi_h \right \rangle_{V_h', V_h}.
\]
Set $T_h(f,g) = (S_h f,S_h g)$ for all $f,g \in V_h'$ and $\Upsilon_h(v,\rho) = \left (I + T_h \circ G \right ) (v,\rho)$. We can apply our generalized BRR theorem to obtain the following first result.

\begin{thm}\label{thm:mfg_error}
 Assume that $d \leq 3$, that \ref{h:HJ_reg}, \ref{item:MFG_1} and \ref{item:MFG_3} hold, and choose $4 \leq r < 6$. Let $(V_h)_{h > 0}$ be a collection of $W^{1,\infty}(\Omega)$-conformal approximation spaces, satisfying \eqref{eq:approximation_property_mfg},\eqref{eq:strong_approximation_property_mfg} and \eqref{eq:aubin_nitsche}. Let also $(u,m) \in \left(H^2(\Omega)\cap H^1_0(\Omega)\right) \times L^\infty(\Omega)$ be a weak solution to \eqref{eq:mfg} such that
  \begin{equation}\label{eq:mfg_kernel}
  \ker_{W_0^{1,r} \times L^r} \left (I + T \circ A \right) = \{0\} \quad \textnormal{for every } A \in \cA[u,m].
 \end{equation}
 Then, there exists $h_0 > 0$ and a neighborhood $\cO$ of $(u,m)$ in $W_0^{1,r}(\Omega) \times L^r(\Omega)$ such that, for every $0 < h \leq h_0$, there is a unique $(u_h,m_h) \in \cO \cap (V_h \times V_h)$ satisfying
 \[
  \Upsilon_h(u_h,m_h) = 0
 \]
 and we have the error estimate
 \[
 \norm{(u-u_h, m-m_h)}{W^{1,r} \times L^r} \leq K \norm{\left (T - T_h \right) G(u,m)}{W^{1,r} \times L^r}
\]
for some $K > 0$. In particular,
 \[
  \norm{u - u_h}{W^{1,r}} + \norm{m - m_h}{L^r} \leq K h^\theta,
 \]
 where $\theta \in (0,1/4]$ is defined by
 \[
  \frac{1}{r} = \frac{\theta}{2} + \frac{1-\theta}{6}.
 \]
\end{thm}

\begin{rem}
 More generally, the conclusion of \cref{thm:mfg_error} remains true if we replace \eqref{eq:mfg_kernel} by
 \[
  \textnormal{$\Upsilon$ is strongly metrically regular at $(u,m)$.}
 \]
\end{rem}

\begin{proof}[Proof of \cref{thm:mfg_error}]
 Using Céa's lemma \cite[Lemma 2.28]{EG2004}, \eqref{eq:strong_approximation_property_mfg} and \ref{h:HJ_reg}, we have
 \[
  \norm{\left(S_1 - S_h\right)g}{H^1} \leq C \inf_{v_h \in V_h} \norm{S_1 g - v_h}{H^1} \leq C h \norm{g}{L^2}
 \]
 for all $g \in L^2(\Omega)$.
 Similarly, using the Aubin-Nitsche lemma \cite[Lemma 2.31]{EG2004} and \eqref{eq:aubin_nitsche}, we also have
 \[
  \norm{\left(S_2 - S_h \right)\mu}{L^2} \leq C h \norm{\left(S_2 - S_h \right)\mu}{H^1} \leq C h \norm{\mu}{H^{-1}}
 \]
for all $\mu \in H^{-1}(\Omega)$. Moreover, using \eqref{eq:approximation_interpolation} and the inverse estimates \eqref{eq:inverse_estimate} we have
\begin{align*}
 \norm{S_1 g - S_h g}{W^{1,6}} & \leq \norm{S_1 g - \cI_h \circ S_1 g}{W^{1,6}} + \norm{\cI_h \circ S_1 g - S_h g}{W^{1,6}} \\
 & \leq C \left(\norm{S_1 g}{W^{1,6}} + h^{d\left(1/6 - 1/2 \right)} \norm{\cI_h \circ \left(S_1 - S_h \right)g}{H^1} \right) \\
 & \leq C \left( \norm{S_1 g}{H^2} + h^{-d/3} \norm{\left(S_1 - S_h \right)g}{H^1} \right) \\
 & \leq C \left(1 + h^{1 -d/3} \right) \norm{g}{L^2}.
\end{align*}
Since we assume $d \leq 3$, it follows that for any $h_0 > 0$, there exists $C>0$ such that
\[
  \norm{\left(S_1  - S_h\right) g}{W^{1,6}} \leq C \norm{g}{L^2}
\]
for all $0 < h \leq h_0$. Since $r \in [4,6)$, there exists $\theta \in (0,1/4)$ such that $\frac{1}{r} = \frac{\theta}{2} + \frac{1-\theta}{6}$.
By interpolation \cite[Theorem 2.6]{L2018}, we obtain that $\norm{\left(S_1 - S_h \right)}{\cL(L^2,W^{1,r})} \leq Ch^\theta$. A similar argument yields $\norm{\left(S_2 - S_h \right)}{\cL(H^{-1},L^r)} \leq Ch^\theta$.
We therefore obtain
\begin{equation}\label{eq:mfg_error_1}
 \norm{T-T_h}{\cL(L^2 \times H^{-1},W^{1,r} \times L^r)} \leq C h^\theta
\end{equation}
and in particular $\norm{T-T_h}{\cL(L^2 \times H^{-1},W^{1,r} \times L^r)} \xrightarrow{h \to 0} 0$.
Moreover, setting $\Phi_h = \Upsilon - \Upsilon_h$, we observe that
\begin{align*}
 \norm{\Phi_h(v_1,\rho_1) - \Phi_h(v_2,\rho_2)}{W^{1,r} \times L^r} & \leq \norm{T - T_h}{\cL(L^2 \times H^{-1}, W^{1,r} \times L^r)} \norm{G(v_1,\rho_1) - G(v_2,\rho_2)}{L^2 \times H^{-1}} \\
 & \leq C h^{\theta} \norm{G(v_1,\rho_1) - G(v_2,\rho_2)}{L^2 \times H^{-1}} \\
 & \leq C h^{\theta} \norm{(v_1 - v_2, \rho_1 - \rho_2)}{W^{1,r} \times L^r}.
\end{align*}
Using also \cref{lem:mfg_prediff,lem:mfg_banach_constant}, we conclude that we may apply \cref{cor:first_example} to obtain $h_0 > 0$ and a neighborhood $\mathcal{O}$ of $(u,m)$ in $W_0^{1,r}(\Omega) \times L^r(\Omega)$ such that, for every $0 < h \leq h_0$, there is $(u_h,m_h) \in \mathcal{O} \cap \left (V_h \times V_h \right)$ such that $\Upsilon_h(u_h,m_h) = 0$ and
\[
 \norm{(u-u_h, m-m_h)}{W^{1,r} \times L^r} \leq K \norm{\left (T - T_h \right) G(u,m)}{W^{1,r} \times L^r}.
\]
We conclude using \eqref{eq:mfg_error_1}.
\end{proof}

\begin{rem}[Nonnegative approximations of the density]
 In the case of $\mathbb{P}^1$-Lagrange finite elements, Osborne and Smears proposed a numerical scheme satisfying a discrete maximum principle, ensuring that the density $m_h$ remains nonnegative \cite{OS2024,OS2025a}. In general the scheme that we propose here may fail to satisfy this property. However, we believe that our approach may be adapted to the scheme in \cite{OS2024}. Indeed, the solutions to the scheme proposed in \cite{OS2024} can be seen as zeros of $F_h^s := I + T_h^s \circ G$  where $G$ is as in the paper and $T_h^s = (S_h^s, S_h^s)$ is the linear operator defined by $S^s_h f = v_h$, where $v_h \in V_h$ is the unique solution to
	\[
	 \int_{\Omega} \left(1 + \gamma_h(x) \right) Dv_h(x) \cdot D\phi_h(x) + \lambda v_h(x) \phi_h(x)\, dx  =  \langle f, \phi_h \rangle
	\]
	for all $\phi_h \in V_h$, and $\gamma_h$ is the artificial diffusion introduced in \cite{OS2024}. It is proved in \cite{OS2024} that the resulting scheme preserves the positivity of the density if the underlying mesh is strictly acute. From the Lax-Milgram theorem, we have $T_h^s \in \cL(H^{-1}(\Omega), H^1_0(\Omega))$ with uniformly bounded norm. Then, setting $w_h := \left(T_h^s - T_h\right) f$, with $T_h$ as in the paper, we have that $w_h$ solves
	\[
	 \int_{\Omega} Dw_h(x) \cdot D \phi_h(x) + \lambda w_h(x) \phi_h(x) \,dx = - \int_{\Omega} \gamma_h(x) D v_h(x) \cdot D \phi_h(x)\, dx.
	\]
	Since $\norm{\gamma_h}{L^\infty} = O(h)$, we deduce that $\norm{w_h}{H^1} = O(h \norm{f}{H^{-1}})$. Hence $\norm{T_h^s - T_h}{\cL(H^{-1},H^1)} = O(h)$. From this estimate and the triangle inequality, we expect that it is possible to obtain estimates on $T - T_h^s$ that allow to verify the assumptions of the generalized BRR theorem for $F_h^s$ instead of $F_h$.

	Note also that the artificial viscosity parameter $\gamma_h$ introduced in \cite{OS2024} vanishes for $h$ small enough (depending on $\lambda$ and $C_H$). It follows that our scheme is also positivity preserving for $h$ small enough if the mesh is both quasi-uniform and strictly acute, since in this case the schemes are identical.
\end{rem}

In terms of PDEs, the condition \eqref{eq:mfg_kernel} amounts to
\begin{equation}
 \begin{gathered}
  \textnormal{$(v,\rho) = (0,0)$ is the unique weak solution to}\\
  \begin{cases}
   - \Delta v + H_p(x,Du)\cdot Dv + \lambda v = dF[m](\rho) \quad & \textnormal{in } \Omega, \\
   - \Delta \rho - \diver \left( \rho H_p(x, Du) \right) + \lambda \rho = \diver \left ( m \xi v \right) \quad & \textnormal{in } \Omega, \\
   v = \rho = 0 \quad & \textnormal{on } \partial \Omega,
  \end{cases}\\
  \textnormal{for all $\xi \in \partial^\circ \mathfrak{H_p}[Du]$.}
 \end{gathered}
\end{equation}
This condition is a natural generalization of the notion of stability introduced in \cite{BC2018} for Hamiltonians of class $C^2$.
We provide two sufficient conditions for \eqref{eq:mfg_kernel}.
\begin{prop}\label{lem:mfg_kernel_mon}
 Let $(u,m)$ be a weak solution to \eqref{eq:mfg} with $m \in L^\infty(\Omega)$.
 Assume \ref{h:HJ_reg}, \ref{item:MFG_1}, \ref{item:MFG_3}, and that
 \begin{equation}\label{eq:lasry_lions}
  \int_\Omega dF[m](\rho)(x) \rho(x) \, dx > 0  \quad \textnormal{for all } \rho \in H_0^1(\Omega) \textnormal{ such that } \rho \neq 0.
 \end{equation}
 Then
 \[
  \ker_{H^1 \times L^2} \left (I + T \circ A \right) = \{0\} \quad \textnormal{for every } A \in \cA[u,m].
 \]
\end{prop}

\begin{rem}[Strong Lasry-Lions condition]
 Condition \eqref{eq:lasry_lions} holds, for instance, when $F$ is strongly monotone in the sense that there exists $c_F > 0$ such that, for every $m_1, \, m_2 \in H^{1}_0(\Omega)$,
 \[
    \int_{\Omega} \left( F[m_1](x) - F[m_2](x)\right) \left( m_1(x) - m_2(x) \right)\, dx \geq c_F \norm{m_1 - m_2}{L^2}^2.
 \]
 In particular, combined with the convexity of the Hamiltonian in \ref{item:MFG_1}, this implies the Lasry-Lions monotonicity condition.
 In the case where $F[m](x) = f(x,m(x))$ with $f$ continuous and of class $C^1$ with respect to its second variable, the assumption holds if $\partial_m f(x,m) \geq c_F > 0$.
\end{rem}

\begin{proof}[Proof of \cref{lem:mfg_kernel_mon}]
 Let $(v,\rho) \in H_0^1(\Omega) \times L^2(\Omega)$ be such that
 \begin{equation}\label{eq:mfg_kernel_mon_1}
  (v,\rho) = - T \circ A (v,\rho).
 \end{equation}
 Since $T \circ A$ maps $H_0^1(\Omega) \times L^2(\Omega)$ to $\cV_1 \times H_0^1(\Omega)$, we deduce that actually $(v,\rho) \in \cV_1 \times H_0^1(\Omega)$. Then, \eqref{eq:mfg_kernel_mon_1} is equivalent to
 \begin{equation}\label{eq:mfg_kernel_mon_2}
  \begin{cases}
   - \Delta v + H_p(x,Du)\cdot Dv + \lambda v = dF[m](\rho) \quad & \textnormal{in } \Omega, \\
   - \Delta \rho - \diver \left (\rho H_p(x,Du) \right)  + \lambda \rho = \diver \left (m \xi Dv \right) \quad & \textnormal{in } \Omega, \\
   v = \rho = 0 \quad & \textnormal{on } \partial \Omega.
  \end{cases}
 \end{equation}
 Using $\rho$ as a test function in the first equation of \eqref{eq:mfg_kernel_mon_2} and $v$ as a test function for the second one, we obtain
 \[
  \int_\Omega D v(x) \cdot D\rho(x) + H_p(x,Du) \cdot Dv(x) \rho(x) + \lambda v(x) \rho(x) \, dx = \int_{\Omega} dF[m](\rho)(x) \rho(x) \, dx
 \]
 and
 \[
  \int_\Omega D \rho(x) \cdot Dv(x) + \rho(x) H_p(x,Du) \cdot Dv(x) + \lambda \rho(x) v(x) +  \left(m(x) \xi(x) Dv(x) \right ) \cdot Dv(x) \, dx = 0.
 \]
Subtracting the last two identities, using the convexity of $H$, the fact that $m \geq 0$, and the fact that $\xi \in \partial^\circ \mathfrak{H_p}[Du]$, we obtain
 \[
  0 \geq - \int_\Omega m \left( \xi Dv \right ) \cdot Dv \, dx = \int_{\Omega} dF[m](\rho)(x) \rho(x) \, dx.
 \]
 It then follows from the monotonicity assumption \eqref{eq:lasry_lions} that $\rho = 0$. Using \eqref{eq:mfg_kernel_mon_2}, we then conclude that $(v,\rho) = (0,0)$.
\end{proof}

\begin{prop}\label{lem:mfg_kernel_lambda}
 Assume \ref{h:HJ_reg}, \ref{item:MFG_1} and \ref{item:MFG_3} hold. Let $(u,m)$ be a weak solution to \eqref{eq:mfg} with $\norm{m}{L^\infty}$ bounded independently of $\lambda$ when $\lambda$ is large enough. Then
 \[
  \ker_{H^1_0 \times L^2} \left (I + T \circ A \right) = \{0\} \quad \textnormal{for every } A \in \cA,
 \]
 whenever $\lambda > 0$ is large enough.
\end{prop}

\begin{proof}
 Notice that $(v,\rho) \in \ker_{H^1_0 \times L^2} \left (I + T \circ A \right)$ if and only if $(v,\rho)$ is a weak solution to
 \begin{equation}\label{eq:mfg_kernel_system}
  \begin{cases}
   - \Delta v + H_p(x,Du) \cdot Dv + \lambda v = dF[m](\rho) \quad & \textnormal{in } \Omega, \\
   - \Delta \rho - \diver \left (\rho H_p(x,Du) \right) + \lambda \rho = \diver \left (m \xi Dv \right) \quad & \textnormal{in } \Omega, \\
   v = \rho = 0 \quad & \textnormal{on } \partial \Omega.
  \end{cases}
 \end{equation}
 where $\xi \in \partial^\circ \mathfrak{H_p}[Du]$. Notice that elliptic regularity implies that $\rho$ belongs to $H^1(\Omega)$. Using $\rho$ as a test-function in the second equation of \eqref{eq:mfg_kernel_system} and using Young's inequality, we have
 \begin{align*}
  \int_\Omega \module{D \rho}^2 + \lambda \module{\rho}^2 \, dx & \leq \int_{\Omega} \module{\rho H_p(x,Du)\cdot D\rho} + \module{m(\xi Dv)\cdot D\rho}\, dx \\
  & \leq \int_\Omega \module{D \rho}^2 + KC_H^2 \left (\norm{\rho}{L^2}^2 + \norm{m}{L^\infty}^2 \norm{Dv}{L^2}^2 \right ),
 \end{align*}
 where the constant $K$ is independent of $\lambda$. Similarly, using $v$ as a test-function in the first equation of \eqref{eq:mfg_kernel_system}, one obtains
 \[
  \norm{v}{H^1} \leq C \norm{dF[m](\rho)}{L^2} \leq C C_F \norm{\rho}{L^2}
 \]
 for $\lambda$ large enough, where the constant $C$ does not depend on $\lambda$. It then follows that $(v, \rho) = (0,0)$.
\end{proof}

\subsubsection{Improved error estimates}

Finally, we are going to build upon \cref{thm:mfg_error} in order to derive $H^1$-$H^1$ quasi-optimal error estimates. The main step to obtain this result is the following $H^1$-$L^2$ error estimate.
\begin{thm}[$H^1$-$L^2$ error estimate]\label{cor:mfg_improved_error}
  In addition to the assumptions of \cref{thm:mfg_error}, with \eqref{eq:mfg_kernel} replaced by
  \begin{equation}
  \ker_{H^1_0 \times L^2} \left (I + T \circ A \right) = \{0\} \quad \textnormal{for every } A \in \cA[u,m],
 \end{equation}
  assume that $F \colon L^2(\Omega) \to L^q(\Omega)$ is strictly differentiable at $m$ for every $q \in (1, 2)$ and that $m \in H^1_0(\Omega)$. Then, up to the choice of a smaller $h_0$, we also have
  \[
   \norm{u - u_h}{H^1} + \norm{m - m_h}{L^2} \leq K \left( \inf_{(v_h,\rho_h) \in V_h \times V_h} \norm{u - v_h}{H^1} + h \norm{m - \rho_h}{H^1} \right).
  \]
  In particular
 \[
  \norm{u - u_h}{H^1} + \norm{m - m_h}{L^2} = O(h).
 \]
\end{thm}

The proof of \cref{cor:mfg_improved_error} relies on two technical lemmas. The first one is an extension of \cref{lem:mfg_banach_constant} for the case $r = 2$.
\begin{lem}\label{lem:mfg_banach_constant_improved}
 Assume that $d \leq 3$, that \ref{h:HJ_reg}, \ref{item:MFG_1} and \ref{item:MFG_3} hold and let $(u,m) \in H^1_0(\Omega) \times L^\infty(\Omega)$. If
 \begin{equation}\label{eq:mfg_kernel_trivial_L2}
  \ker_{H^{1}_0 \times L^2} \left (I + T \circ A \right) = \{0\} \quad \textnormal{for every } A \in \cA[u,m],
 \end{equation}
 then, for every $A \in \cA[u,m]$, the linear operator $I + T \circ A$ is surjective (hence bijective) from $H^1_0(\Omega) \times L^2(\Omega)$ onto itself and there exists $\kappa > 0$ such that
 \begin{equation}\label{eq:mfg_banach_const_L2}
  \inf_{A \in \cA[u,m]} \mathfrak{C}(I + T \circ A) \geq \kappa.
 \end{equation}
 In particular
 \[
  \sup_{A \in \cA[u,m]} \norm{\left (I + T \circ A \right)^{-1}}{\cL(H^1_0 \times L^2)} \leq \kappa^{-1}.
 \]
\end{lem}

\begin{proof}
  Let $s \in (0,1)$ and recall also that the embedding
  \[
   H^{2-s}(\Omega) \times H^{1 - s}(\Omega) \hookrightarrow H^1(\Omega) \times L^2(\Omega)
  \]
  is compact (see \cite[Theorem 7.1]{DNPV2012}).
  Recall that we have
  \[
   T \in \cL(L^2(\Omega) \times H^{-1}(\Omega), \cV_1 \times H^{1}_0(\Omega))
  \]
  and
  \[
   T \in \cL(H^{-1}(\Omega) \times \cV_1', H^1_0(\Omega) \times L^2(\Omega)).
  \]
  By interpolation \cite[Theorem 1.6]{L2018}, we deduce that
  \[
   T \in \cL \left(H^{-s}(\Omega) \left(H^{-1}(\Omega), \cV_1' \right)_{s,2}, \left(\cV_1, H^1_{0}(\Omega) \right)_{s,2} \times H_0^s(\Omega) \right).
  \]
  From the definition of the real interpolation functor, we have that
  \[
   \left(\cV_1, H^1_{0}(\Omega) \right)_{s,2} = \left(H^2(\Omega) \cap H^1_0(\Omega), H^1_{0}(\Omega) \right)_{s,2} \hookrightarrow \left( H^2(\Omega), H^1(\Omega) \right)_{s,2} = H^{2- s}(\Omega).
  \]
  Similarly,
  \[
   \left(\cV_1, H^1_{0}(\Omega) \right)_{s,2} \hookrightarrow \left(H^1_0(\Omega), H^1_{0}(\Omega) \right)_{s,2} = H^1_0(\Omega).
  \]
  It follows that
  \[
   \left(\cV_1, H^1_{0}(\Omega) \right)_{s,2} \hookrightarrow H^{2-s}(\Omega) \cap H^1_0(\Omega) = \cV_{1-s}.
  \]
  Moreover, from \cite[Theorem 1.18]{L2018}, we have
  \[
   \left(H^{-1}(\Omega), \cV_1' \right)_{s,2} = \left(H^1_0(\Omega), \cV_1 \right)_{s,2}' \hookleftarrow \cV_{s}'.
  \]
  We conclude that
  \[
   T \in \cL \left(H^{-s}(\Omega) \times \cV_{s}', \cV_{1-s} \times H_0^s(\Omega) \right).
  \]

  Since each operator $A \in \cA[u,m]$ maps $H^1_0(\Omega) \times L^2(\Omega)$ to $L^2(\Omega) \times H^{-1}(\Omega) \hookrightarrow H^{-s}(\Omega) \times \cV_{s}' $, we deduce that
  \[
   R_{H^1_0 \times L^2}(T \circ A) \subset \cV_{1-s} \times H_0^{1 - s}(\Omega),
  \]
  so that $T \circ A$ is compact as a operator from $H^1_0(\Omega) \times L^2(\Omega)$ to itself.
 The surjectivity of $I + T \circ A$ then follows from \eqref{eq:mfg_kernel_trivial_L2} and Fredholm's alternative.

 For \eqref{eq:mfg_banach_const_L2}, we first claim that the set
 \[
  \left \{ I + T \circ A : \, A \in \cA[u,m] \right \}
 \]
 is sequentially compact in the weak operator topology of $\cL(X)$, where $X := H^1_0(\Omega) \times L^2(\Omega)$. From \cref{prop:WOT_composition}, it is enough to prove the sequential compactness of $\cA[u,m]$ in the weak operator topology of $\cL(X,Y)$, where $Y:= H^{-s}(\Omega) \times \cV_s'$. Let $\delta > 2$ be such that $H^{1 + s}(\Omega) \hookrightarrow W^{1,\delta}(\Omega)$ (see \cite[Section 6]{DNPV2012}). Then, there exists $1 < q < \infty$ such that $1/2 + 1/\delta + 1/q = 1$.  Since $\partial^\circ \mathfrak{H_p}[Du]$ is bounded in $L^q(\Omega;\RR^d)$, we deduce from \cref{prop:nemytskii_basic_properties} that $\partial^\circ \mathfrak{H_p}[Du]$ is sequentially compact for the weak topology of $\cL(L^r(\Omega,\RR^d), L^{q}(\Omega,\RR^d))$. Using \cref{prop:WOT_composition}, this implies that the collection of linear operators
 \[
   W^{1,r}(\Omega) \ni v \mapsto \diver \left (m \xi Dv \right ) \in  \cV_s' \quad \textnormal{for } \xi \in \partial^\circ \mathfrak{H_p}[u]
 \]
 is sequentially compact for the weak operator topology. The claim then easily follows.
 From \cref{prop:sufficient_banach_const}, we conclude that
 \[
  \inf_{A \in \cA[u,m]} \gothC_{H^1_0 \times L^2}(I + T \circ A) \geq \kappa > 0. \qedhere
 \]
\end{proof}

The second technical lemma is a variation over \cref{lem:mfg_prediff}.
\begin{lem}\label{lem:mfg_prediff_improved}
 Assume that $d \leq 3$, that \ref{item:MFG_1} holds and let $4 \leq r < 6$. Let also $(u,m) \in W^{1,r}_0(\Omega) \times L^\infty(\Omega)$ and assume that $F \colon L^2(\Omega) \to L^q(\Omega)$ is Fréchet differentiable at $m$ for every $q \in (1, 2)$. Then, for every $\epsilon > 0$, there exists $\delta > 0$ such that, whenever $(v,\rho) \in B_{W^{1,r}_0 \times L^2}\left ((u,m), \delta \right)$, there exists $A \in \cA[u,m]$ satisfying
 \[
  \norm{G(v,\rho) - G(u,m) - A(v - u,\rho - m)}{H^{-1} \times H^{-2}} \leq \epsilon \norm{(v,\rho) - (u,m)}{H^1 \times L^2}.
 \]
\end{lem}
\begin{proof}
 The proof follows the same lines as the one of \cref{lem:mfg_prediff}. Let $1 < q < 2$ be such that $H^1(\Omega) \hookrightarrow L^{q'}(\Omega)$, with $\frac{1}{q} + \frac{1}{q'} = 1$ and $\frac{1}{r} + \frac{1}{2} = \frac{1}{q}$.
Using the differentiability of $\mathfrak{H}$ and $F$, we have
\begin{align*}
 \norm{\mathfrak{H}[Dv] - \mathfrak{H}[Du] - \mathfrak{H_p}[Du] \cdot (Dv - Du)}{H^{-1}} & \leq C \norm{\mathfrak{H}[Dv] - \mathfrak{H}[Du] - \mathfrak{H_p}[Du] \cdot (Dv - Du)}{L^q} \\
 & \leq \frac{\epsilon}{4} \norm{Dv - Du}{L^2}
\end{align*}
and
\begin{align*}
 \norm{F[\rho] - F[m] - dF[m](\rho - m)}{H^{-1}} & \leq C  \norm{F[\rho] - F[m] - dF[m](\rho - m)}{L^q} \\
 & \leq \frac{\epsilon}{4} \norm{\rho - m}{L^2}
\end{align*}
whenever $(v,\rho) \in B_{W^{1,r}_0 \times L^2}((u,m),\delta_1)$ for some $\delta_1 > 0$. From \cref{thm:nemytskii_prediff}, for every $C > 0$, there exists $\delta_2 > 0$ such that for every $(v,\rho) \in  B_{W^{1,r}_0 \times L^2}((u,m),\delta_2)$ there exists $\xi \in \partial^\circ \mathfrak{H_p}[Du]$ such that
\[
 \norm{\mathfrak{H_p}[Dv] - \mathfrak{H_p}[Du] - \xi (Dv - Du)}{L^q} \leq \frac{\epsilon}{4 C \norm{m}{L^\infty}}\norm{Dv - Du}{L^2}.
\]
Then, there exists $0 < \delta \leq \min\{\delta_1, \delta_2 \}$ such that
\begin{align*}
 & \norm{\diver \left (\rho \mathfrak{H_p}[Dv] - m \mathfrak{H_p}[Du] - (\rho - m) \mathfrak{H_p}[Du] - m \xi \left (Dv - Du \right) \right )}{H^{-2}} \\
 & \quad \leq \sup_{\norm{\phi}{H^2} \leq 1} \int_{\Omega} \left (\rho \mathfrak{H_p}[Dv] - m \mathfrak{H_p}[Du] - (\rho - m) \mathfrak{H_p}[Du] - m \xi \left (Dv - Du \right) \right ) \cdot D \phi \, dx \\
 & \quad \leq C \norm{\rho \mathfrak{H_p}[Dv] - m \mathfrak{H_p}[Du] - (\rho - m) \mathfrak{H_p}[Du] - m \xi \left (Dv - Du \right)}{L^q} \\
 & \quad \leq C \left ( \norm{(\rho - m) \left (\mathfrak{H_p}[Dv] - \mathfrak{H_p}[Du] \right )}{L^q} + \norm{m \left ( \mathfrak{H_p}[Dv] - \mathfrak{H_p}[Du] - \xi \left(Dv - Du \right) \right )}{L^q} \right ) \\
 & \quad \leq C \left ( C_H \norm{Dv - Du}{L^r} \norm{\rho - m}{L^2} + \norm{m}{L^\infty} \norm{\mathfrak{H_p}[Dv] - \mathfrak{H_p}[Du] - \xi \left(Dv - Du \right)}{L^q} \right ) \\
 & \quad \leq \frac{\epsilon}{2} \left (\norm{\rho - m}{L^2} + \norm{Dv - Du}{L^2} \right)
\end{align*}
for all $(v,\rho) \in  B_{W^{1,r}_0 \times L^2}((u,m),\delta)$. The result then follows easily.
\end{proof}

\begin{proof}[Proof of \cref{cor:mfg_improved_error}]
 Let $\epsilon > 0$ to be fixed later. From \cref{thm:mfg_error}, we know that
 \[
  \lim_{h \to 0} \norm{u_h - u}{W^{1,r}} + \norm{m_h - m}{L^r} = 0.
 \]
 \cref{lem:mfg_prediff_improved} then implies that there exists $h_0 > 0$ such that, for every $0 < h \leq h_0$, there exists $A_h \in \cA[u,m]$ such that
\[
 \norm{G(u_h,m_h) - G(u,m) - A_h(u_h - u, m_h - m)}{H^{-1} \times \cV_1'} \leq \epsilon \norm{(u_h - u, m_h - m)}{H^1 \times L^2}.
\]
Moreover, using \cref{lem:mfg_banach_constant_improved}, we have
\[
 \norm{(u_h - u, m_h - m)}{H^1 \times L^2} \leq \kappa^{-1} \norm{\left (I + T \circ A_h \right)(u_h - u, m_h - m)}{H^1 \times L^2}.
\]
On the other hand, since $\Upsilon(u,m) = \Upsilon_h(u_h,m_h) = 0$, we also have
\[
 (u_h,m_h) - (u,m) = T_h \circ G(u_h,m_h) - T \circ G(u,m),
\]
and it follows that
\begin{align*}
 & \norm{\left (I + T\circ A_h \right)(u_h - u, m_h - m)}{H^1 \times L^2} \\
  & \qquad  = \norm{T_h \circ G(u_h,m_h) - T \circ G(u,m) - T \circ A_h (u_h - u, m_h - m)}{H^1 \times L^2} \\
  & \qquad \leq \norm{(T_h - T) \circ G(u_h,m_h)}{H^1 \times L^2} + \norm{T \left (G(u_h,m_h) - G(u,m) -  A_h (u_h - u, m_h - m) \right)}{H^1 \times L^2} \\
  & \qquad \leq \norm{(T_h - T) \circ G(u,m)}{H^1 \times L^2} + \norm{T \left (G(u_h,m_h) - G(u,m) -  A_h (u_h - u, m_h - m) \right)}{H^1 \times L^2} \\
  & \qquad \qquad + \norm{(T - T_h) \left(G(u,m) - G(u_h,m_h) \right)}{H^1 \times L^2} \\
  & \qquad \leq  \norm{(T_h - T) \circ G(u,m)}{H^1 \times L^2} + \epsilon \norm{T}{\cL(H^{-1} \times \cV_1', H^1 \times L^2)} \norm{(u_h -u, m_h -m)}{H^1 \times L^2} \\
  & \qquad \qquad + \norm{T - T_h}{\cL(L^2 \times H^{-1}, H^1 \times L^2)} \norm{G(u,m) - G(u_h,m_h)}{L^2 \times H^{-1}}
\end{align*}
for all $0 < h \leq h_0$.
Using \ref{item:MFG_1}, we deduce that there exists $L_G > 0$, depending on $\norm{m}{L^\infty}$, such that
\[
 \norm{G(u,m) - G(u_h,m_h)}{L^2 \times H^{-1}} \leq L_G \norm{(u - u_h, m - m_h)}{H^1 \times L^2}.
\]
Since $\norm{T - T_h}{\cL(L^2 \times H^{-1}, H^1 \times L^2)} = O(h)$, we may choose $h_0$ small enough so that
\[
  \norm{T - T_h}{\cL(L^2 \times H^{-1}, H^1 \times L^2)} \leq \frac{\kappa}{3 L_G} \quad \textnormal{for all } 0 < h \leq h_0.
\]
It follows that
\begin{align*}
 \norm{(u_h - u, m_h - m)}{H^1 \times L^2} & \leq \left (\kappa^{-1} \epsilon \norm{T}{\cL(H^{-1} \times (H^{2})', H^1 \times L^2)} + \frac{1}{3} \right) \norm{(u_h -u, m_h -m)}{H^1 \times L^2} \\ & \qquad + \kappa^{-1} \norm{(T_h - T) \circ G(u,m)}{H^1 \times L^2}.
\end{align*}
Choosing $\epsilon := \kappa/\left(3 \norm{T}{\cL(H^{-1} \times (H^{2})', H^1 \times L^2)} \right)$, we obtain
\[
 \norm{(u_h - u, m_h - m)}{H^1 \times L^2} \leq  \frac{3}{\kappa} \norm{(T_h - T) \circ G(u,m)}{H^1 \times L^2}.
\]
The conclusion then follows from Céa's lemma \cite[Lemma 2.28]{EG2004}, the Aubin-Nitsche lemma \cite[Lemma 2.31]{EG2004}, \eqref{eq:strong_approximation_property_mfg} and \eqref{eq:aubin_nitsche}.
\end{proof}

\begin{cor}[{$H^1$-$H^1$ quasi-optimal error estimate}]\label{cor:quasi_optimal_mfg}
 Under the assumptions of \cref{cor:mfg_improved_error} and up to the choice of a smaller $h_0 > 0$, there exists $C > 0$ such that
 \[
 \norm{m - m_h}{H^1} + \norm{u - u_h}{H^1} \leq C \left( \inf_{(v_h,\rho_h) \in V_h \times V_h}  \norm{u - v_h}{H^1} + \norm{m - \rho_h}{H^1} \right)
\]
for all $0 < h \leq h_0$.
\end{cor}

\begin{rem}
 This error estimate is consistent with the rate of convergence observed in the first author's PhD thesis \cite[Chapter 3]{B_thesis} and in \cite{OS2024}.
\end{rem}

\begin{proof}[Proof of \cref{cor:quasi_optimal_mfg}]
Let $a \colon H^1_0(\Omega) \times H^1_0(\Omega) \to \RR$ be the continuous bilinear form defined by
\[
 a(\mu,\rho) = \int_\Omega D \mu(x) \cdot D \rho(x) + \mu(x) H_p(x,Du(x)) \cdot D\rho(x) + \lambda \mu(x) \rho(x) \, dx
\]
and let $\eta \geq 0$ be such that the bilinear form
\[
 H^1_0(\Omega) \times H^1_0(\Omega) \ni (\mu,\rho) \mapsto a(\mu,\rho) + \eta \langle \mu, \rho \rangle_{L^2} \in \RR
\]
is coercive. Notice that
\[
 a(m,\rho) + \eta \langle m, \rho \rangle_{L^2} = \langle \lambda m_0 + \eta m, \rho \rangle_{L^2} \quad \textnormal{for all } \rho \in H^1_0(\Omega).
\]
Using the Lax-Milgram theorem, there exists a unique $\tilde m_h \in V_h$ such that
\begin{equation}\label{eq:quasi_optimal_1}
 a(\tilde m_h, \rho_h) + \eta \langle \tilde m_h, \rho_h \rangle_{L^2} =  \langle \lambda m_0 + \eta m, \rho_h \rangle_{L^2} \quad \textnormal{for all } \rho_h \in V_h
\end{equation}
and Céa's lemma \cite[Lemma 2.28]{EG2004} implies that
\begin{equation}
 \norm{m - \tilde m_h}{H^1} \leq C \inf_{\rho_h \in V_h} \norm{m - \rho_h}{H^1}.
\label{eq:conseq_Cea}
\end{equation}
Then, using \eqref{eq:quasi_optimal_1} and the fact that
\[
  \int_{\Omega} Dm_h(x) \cdot D\rho_h(x) + m_h(x) H_p(x,Du_h(x)) \cdot D \rho_h(x) + \lambda m_h(x) \rho_h(x) \, dx = \lambda \langle m_0, \rho_h \rangle_{L^2}
\]
for all $\rho_h \in V_h$, we have
\begin{align*}
 a(\tilde m_h - m_h, \rho_h) & + \eta \langle \tilde m_h - m_h, \rho_h \rangle_{L^2} \\ & \quad  = \eta \langle (m - m_h), \rho_h \rangle_{L^2} + \langle m_h (H_p(\cdot,Du) - H_p(\cdot,Du_h)), D\rho_h \rangle_{L^2} \\
 & \quad = \eta \langle (m - m_h), \rho_h \rangle_{L^2} + \langle (m_h - m) (H_p(\cdot,Du) - H_p(\cdot,Du_h)), D\rho_h \rangle_{L^2} \\ & \qquad + \langle m (H_p(\cdot,Du) - H_p(\cdot,Du_h)), D\rho_h \rangle_{L^2}
\end{align*}
for all $\rho_h \in V_h$. With the choice $\rho_h = \tilde m_h - m_h$, we deduce
\begin{multline*}
 \norm{\tilde m_h - m_h}{H^1}^2 \leq C \bigg( \eta \norm{m - m_h}{L^2}\norm{\tilde m_h - m_h}{L^2} + 2 \norm{H_p}{L^\infty} \norm{m - m_h}{L^2} \norm{\tilde m_h - m_h}{H^1} \\ + C_H \norm{m}{L^\infty} \norm{\tilde m_h - m_h}{H^1} \norm{u - u_h}{H^1} \bigg),
\end{multline*}
where we have used the fact that $H_p$ is bounded and Lipschitz continuous. Using Young's inequality, we conclude that
\[
 \norm{\tilde m_h - m_h}{H^1} \leq C \left(\norm{m - m_h}{L^2} + \norm{u - u_h}{H^1} \right).
\]
 It follows from~\cref{cor:mfg_improved_error} and~\eqref{eq:conseq_Cea} that
\begin{align*}
 \norm{m - m_h}{H^1} & \leq \norm{m - \tilde m_h}{H^1} + \norm{\tilde m_h - m_h}{H^1} \\
 & \leq \norm{m - \tilde m_h}{H^1} + C \left(\norm{m - m_h}{L^2} + \norm{u - u_h}{H^1} \right) \\
 & \leq C \left( \inf_{(v_h,\rho_h, \tilde \rho_h) \in V_h \times V_h \times V_h}  \norm{u - v_h}{H^1} + \norm{m - \rho_h}{H^1} + h\norm{m - \tilde \rho_h}{L^2} \right) \\
 & \leq C \left( \inf_{(v_h,\rho_h) \in V_h \times V_h}  \norm{u - v_h}{H^1} + \norm{m - \rho_h}{H^1} \right)
\end{align*}
where the last inequality holds for $h$ small enough. In the end, we have obtained that
\[
 \norm{m - m_h}{H^1} + \norm{u - u_h}{H^1} \leq C \left( \inf_{(v_h,\rho_h) \in V_h \times V_h}  \norm{u - v_h}{H^1} + \norm{m - \rho_h}{H^1} \right).\qedhere
\]
\end{proof}

\appendix

\section{Set-valued mappings}
\label{section:set_valued}

Let $X$ and $Y$ be Banach spaces. For a set-valued mapping $F \colon X \rightrightarrows Y$ we recall that the graph of $F$ is defined by
\[
 \gph F = \left \{ (x,y) \in X \times Y : y \in F(x) \right \}.
\]
The domain of $F$ is defined by $\dom F = \left \{ x \in X : F(x) \neq \varnothing \right \}$.

\begin{defi}\label{defi:continuity_set_valued}
 Let $F \colon X \rightrightarrows Y$ be a set-valued mapping.
 \begin{itemize}
 \item We define the \emph{upper-limit} of $F$ at $\bar x \in X$ by
 \[
  \limsup_{x \to \bar x} F(x) = \left \{ y \in Y : \exists x_n \to \bar x,\, \exists y_n \in F(x_n),\, \textnormal{s.t. } y_n \to y \right \}.
 \]

 \item We say that $F$ is \emph{upper-semicontinuous} at $\bar x \in X$ if, for every $\epsilon > 0$, there exists $\delta > 0$ such that
 \[
  F(x) \subset \left \{ y \in Y : d(F(\bar x), y) < \epsilon \right \} \quad \textnormal{for all } x \in B_X(\bar x, \delta),
 \]
 or, equivalently, if $\limsup_{x \to \bar x} F(x) \subset F(\bar x)$.
 \end{itemize}
\end{defi}

In \cref{section:nemytskii}, we shall also deal with the measurability of set-valued maps and measurable selections. We therefore recall here the necessary definitions and results.
\begin{prop}[{\cite[Lemma 4.50]{AB2006}}]\label{prop:caratheodory_measurable}
 Let $(\Omega, \cF)$ be a measurable space, $X$ and $Y$  metric spaces, with $X$ separable, and $f \colon \Omega \times X \to Y$ be a Carathéodory function\footnote{That is, such that $\Omega \ni \omega \mapsto f(\omega,x) \in Y$ is measurable for every $x \in X$ and $X \ni x \mapsto f(\omega,x) \in Y$ is continuous for every $\omega \in \Omega$.}. Then $f$ is $(\cF \otimes \cB(X), \cB(Y))$-measurable.
\end{prop}

\begin{defi}[{\cite[Definition 14.1]{RW1998}}]
  Consider a measurable space $(\Omega, \cF)$, a complete metric space $X$ and a set-valued map $F \colon \Omega \rightrightarrows X$ with closed values. The map $F$ is called (weakly) $\cF$-measurable, if for every open subset $ O \subset X$, we have
  \[
   F^{-1}(O) = \left \{ \omega \in \Omega : F(\omega) \cap O \neq \varnothing \right \} \in \cF.
  \]
\end{defi}

\begin{thm}[{\cite[Theorem 14.8]{RW1998}}]\label{thm:graph_measurability}
 Let $(\Omega,\cF)$ be a measurable space and let $F \colon \Omega \rightrightarrows \RR^n$ have closed values. If $F$ is $\cF$-measurable, then $\gph F \in \cF \otimes \cB(\RR^n)$. Conversely, if $\cF$ if complete with respect to some measure and $\gph F \in \cF \otimes \cB(\RR^n)$, then $F$ is $\cF$-measurable.
\end{thm}

\begin{defi}
 Let $(\Omega, \cF)$ be a measurable space and $X$ a Polish space. Consider a set-valued mapping $G \colon \Omega \rightrightarrows X$. A $(\cF,\cB(X))$-measurable map $g \colon \Omega \to X$ satisfying $g(\omega) \in G(\omega)$ for every $\omega \in \Omega$ is called a measurable selection of $G$.
\end{defi}

We recall the following measurable selection theorem.
\begin{thm}[{Kuratowski and Ryll-Nardzewski, \cite[Theorem 8.1.3]{AF1990}}]\label{thm:KRN}
 Let $X$ be a Polish space and $(\Omega, \cF)$ a measurable space. Let $F \colon \Omega \rightrightarrows X$ be $\cF$-measurable  with closed values, be such that $\dom F = \Omega$. Then there exists a $(\cF,\cB(X))$-measurable selection of $F$.
\end{thm}

The following proposition collects the facts that we will use to prove measurability of some mappings.
\begin{prop}\label{prop:measurability_tools}
 Let $(\Omega, \cF)$ be a measurable space and let $F \colon \Omega \rightrightarrows \RR^n$ be closed-valued and measurable. Then,
 \begin{enumerate}[label={\rm (\roman*)}]
  \item \label{item:measurability_reunion} if $G \colon \Omega \rightrightarrows \RR^n$ is closed-valued and measurable, then $\Omega \ni \omega \mapsto F(\omega) \cup G(\omega) \subset \RR^m$ is also measurable;

  \item \label{item:measurability_intersection} if $\mathcal{I}$ is a countable set and $F_i \colon \Omega \rightrightarrows \RR^n$ is  closed-valued and measurable for each $i \in \cI$, then $\bigcap_{i \in \mathcal{I}} F_i$ is also measurable;

  \item \label{item:measurability_product} if $G_1, \dots , G_N \colon \Omega \rightrightarrows \RR^n$ are measurable, then $ \Omega \ni \omega \mapsto G_1(\omega) \times  \dots \times  G_N(\omega) \subset \left (\RR^n \right)^N$ is also measurable;

  \item \label{item:measurablility_convex_hull} the set-valued map $\co F$ is measurable;

  \item \label{item:measurability_composition} if $G \colon \Omega \times \RR^n \rightrightarrows \RR^m$ is closed valued, with $G(\omega,\cdot)$ upper semicontinuous for every $\omega \in \Omega$, and is such that the set-valued map $\Omega \ni \omega \mapsto \gph G(\omega, \cdot) \subset \RR^n \times \RR^m$ is measurable, then $\Omega \ni \omega \mapsto G(\omega, F(\omega)) \subset \RR^m$ is also measurable;

  \item \label{item:measurability_projection} if $g \colon \Omega \to \RR^n$ is measurable, then the projection mapping\footnote{For a closed set $C \subset \RR^n$, we define the set-valued projection mapping $\Pi_C \colon \RR^n \rightrightarrows \RR^n$ by
  \[
   \Pi_C(y) = \left \{z \in \RR^d : \module{y - z} = d(y,C)\right \}.
  \]
  } $\Omega \ni \omega \to \Pi_{F(\omega)}(g(\omega)) \in \RR^m$ is measurable.
 \end{enumerate}

\end{prop}
\begin{proof}
 See \cite[Proposition 14.11, Exercise 14.12, Theorem 14.13(b) and Exercise 14.17]{RW1998}.
\end{proof}

\begin{thm}[{\cite[Example 14.15]{RW1998}}]\label{thm:implicit_measurability}
 Let $(\Omega, \cF)$ be a measurable space and $g \colon \Omega \times \RR^d \to \RR^n$ be a Carathéodory function. Let also $\Xi \colon \Omega \rightrightarrows \RR^d$ and $G \colon \Omega \rightrightarrows \RR^n$ be $\cF$-measurable and closed valued. Then the set-valued map
 \[
  F \colon \Omega \ni \omega \mapsto \left \{ \xi \in \Xi(\omega) : g(\omega, \xi) \in G(\omega) \right \} \subset \RR^d \textnormal{  is $\cF$-measurable.}
 \]
\end{thm}

\begin{defi}[Support function]\label{defi:support_function}
 Let $H$ be an Hilbert space and $C \subset H$. The \emph{support function} $\sigma_C \colon H \to \RR$ of $C$ is defined by
\[
 \sigma_C(u) := \sup_{\xi \in C}\, \langle \xi, u \rangle.
\]
\end{defi}

The support function provides a useful characterization of closed convex sets.
\begin{prop}[{\cite[Proposition 7.11]{BC2017}}]\label{prop:support_characterization}
 Let $H$ be an Hilbert space and $C \subset H$ be closed and convex. Then $\xi \in C$ if and only if
 \[
  \langle \xi, u \rangle \leq \sigma_C(u) \quad \textnormal{for all } u \in H.
 \]
\end{prop}

\begin{prop}\label{prop:support_measurable}
 Let $(\Omega, \cF)$ be a measurable space and $F \colon \Omega \rightrightarrows \RR^n$ with convex and compact values. If $\Omega \ni \omega \mapsto \sigma_{F(\omega)}(u) \in \RR$ is measurable for every $u \in \RR^n$, then $F$ is $\cF$-measurable.
\end{prop}
\begin{proof}
 For each $u \in \RR^n$, define the set-valued map $F(\cdot;u) \colon \Omega \rightrightarrows \RR^n$ by
 \[
  F(\omega;u) = \left \{ \xi \in \RR^n : \sigma_{F(\omega)}(u) - \xi \cdot u \geq 0 \right \}.
 \]
 Since $\sigma_{F(\cdot)}(u)$ is measurable and $\xi \mapsto \sigma_{F(\omega)}(u) - \xi \cdot u$ is continuous for every $u \in \RR^n$ and $\omega \in \Omega$, it follows from \cref{thm:implicit_measurability} that each $F(\cdot;u)$ is $\cF$-measurable. From \cref{prop:support_characterization} and the fact that $F$ has convex and compact values, we deduce that
 \[
  F(\omega) = \bigcap_{u \in \RR^n} F(\omega;u) \quad \textnormal{for every $\omega \in \Omega$.}
 \]
 Moreover, for every $\omega \in \Omega$, the fact that $F(\omega)$ is compact also implies the continuity of $u \mapsto \sigma_{F(\omega)}(u)$. Therefore, the mapping
 \[
  \RR^n \ni u \mapsto \sigma_{F(\omega)}(u) - \xi \cdot u \in \RR
 \]
 is continuous for every $\omega \in \Omega$ and $\xi \in \RR^n$. We conclude that
 \[
  F(\omega) = \bigcap_{u \in \QQ^n} F(\omega,u)
 \]
 and that $F$ is $\cF$-measurable by \cref{prop:measurability_tools}-\ref{item:measurability_intersection}.
\end{proof}

\section{Weak operator topology}
\label{section:WOT}
Let $X$ and $Y$ be Banach spaces. In this section we recall the definition and some of the properties of the weak operator topology of $\cL(X,Y)$. We refer to \cite[Section 3.1]{AJP2006} for an exposition in the case where $X = Y$ is an Hilbert space.

For every pair $(x,h) \in X \times Y'$, define the linear form $x \otimes h \colon \cL(X,Y) \to \RR$ by
\[
x \otimes h (T) := \left \langle h, Tx \right \rangle_{Y',Y}.
\]
The weak operator topology (WOT) on $\cL(X,Y)$ is then defined as the coarsest topology\footnote{See \cite[Section 3.1]{B2011}.} making all the maps $\left(x \otimes h \right)_{x \in X,\, h \in Y'}$ continuous.

\begin{prop}\label{prop:WOT_hausdorff}
 The weak operator topology is Hausdorff.
\end{prop}
\begin{proof}
 Let $T_1,\, T_2 \in \cL(X,Y)$ with $T_1 \neq T_2$, we have to find $O_1,\, O_2 \subset \cL(X,Y)$, open in the weak operator topology, such that $T_1 \in O_1$, $T_2 \in O_2$ and $O_1 \cap O_2 = \varnothing$.

 Since $T_1 \neq T_2$, there exists $x \in X$ such that $T_1 x \neq T_2 x$. Then, from the Hahn-Banach theorem, there exists $h \in Y'$ and $\alpha \in \RR$ such that
 \[
  \left \langle h, T_1 x \right \rangle_{Y',Y} < \alpha < \left \langle h, T_2 x \right \rangle_{Y',Y}.
 \]
 It is then enough to set $O_1 := (x \otimes h)^{-1}((-\infty, \alpha))$ and $O_2 := (x \otimes h)^{-1}((\alpha, + \infty))$.
\end{proof}

\cref{prop:WOT_hausdorff} allows to work with sequences in the weak operator topology. By definition, a sequence $(T_n)_{n \in \NN}$ in $\cL(X,Y)$ converges to $T \in \cL(X,Y)$ in the weak operator topology if and only if
\[
 \lim_{n \to \infty} \left \langle h, T_n x \right \rangle_{Y',Y} = \left \langle h, Tx \right \rangle_{Y',Y} \quad \textnormal{for all $x \in X$ and $h \in Y'$}
\]
and we write $T_n \WOTarrow T$.

\begin{prop}\label{prop:WOT_composition}
 Let $Z$ be a Banach space, $(T_n)_{n \in \NN}$ a sequence in $\cL(X,Y)$, $T\in \cL(X,Y)$ and $S \in \cL(Y,Z)$. Assume that $T_n \WOTarrow T$. Then $S \circ T_n \WOTarrow S \circ T$.
\end{prop}
\begin{proof}
 Let $x \in X$ and $h \in Z'$. Then $\left \langle h, S \circ T_n x \right \rangle_{Z', Z} = \left \langle S^\star h, T_n x \right \rangle_{Y',Y}$ where $S^\star \in \cL(Z',Y')$ is the adjoint of $S$. Using the fact that $T_n \WOTarrow T$, we have that
 \[
  \lim_{n \to \infty} \left \langle h, S \circ T_n x \right \rangle_{Z', Z} = \left \langle S^\star h, T x \right \rangle_{Y',Y} =  \left \langle h, S \circ T x \right \rangle_{Z', Z}.
 \]
 This concludes the proof.
\end{proof}

\begin{prop}\label{prop:WOT_Lp_compact}
 Let $(\Omega,\cF,\mu)$ be a complete measure space, $n,d \in \NN^*$ and $1 < p, q ,r < \infty$  be such that $\frac{1}{q} = \frac{1}{p} + \frac{1}{r}$. Let also $\Xi \subset L^r(\Omega,\cF,\mu; \RR^{n \times d})$ be convex, closed and bounded. For each $\xi \in \Xi$, define the linear operator $T_\xi \in \cL(L^p(\Omega,\cF,\mu;\RR^{d}),L^q(\Omega,\cF,\mu;\RR^n))$ by
 \[
  \left(T_\xi u \right)(x) := \xi(x) u(x) \quad \textnormal{for } x \in \Omega
 \]
 and set
 \[
  \cA := \left \{ T_\xi : \, \xi \in \Xi \right \}.
 \]
 Then $\cA$ is sequentially compact in the weak operator topology of $\cL(L^p(\Omega,\cF,\mu;\RR^{d}),L^q(\Omega,\cF,\mu;\RR^n))$.
\end{prop}
\begin{proof}
 Let $\left (T_{\xi_n} \right )_{n \in \NN}$ be a sequence in $\cA$. Since $L^r(\Omega,\cF,\mu; \RR^{n \times d})$ is reflexive and $\Xi$ is bounded, there exists $\xi \in  L^r(\Omega,\cF,\mu; \RR^{n \times d})$ such that $\xi_n$ converges weakly to $\xi$ up to extraction of a subsequence. Moreover, since $\Xi$ is closed and convex, we have $\xi \in \Xi$. Let $q'$ denote conjugate exponent of $q$, so that $\left (L^q(\Omega,\cF,\mu;\RR^n) \right)' = L^{q'}(\Omega,\cF,\mu;\RR^n)$.
 Then, for every $u \in L^p(\Omega,\cF,\mu;\RR^{d})$ and $h \in L^{q'}(\Omega,\cF,\mu;\RR^n)$, we have
 \[
  \left \langle h, T_{\xi_n} u \right \rangle_{L^{q'},L^q} = \int_\Omega \left(\xi_n(x) u(x)\right) \cdot h(x) \mu(dx) \xrightarrow{} \int_\Omega \left( \xi(x) u(x) \right) \cdot h(x) \mu(dx) = \left \langle h, T_{\xi} u \right \rangle_{L^{q'},L^q},
 \]
 so that $T_{\xi_n} \WOTarrow T_\xi \in \cA$.
\end{proof}

\printbibliography

@inbook{A2013,
  title = {Finite {{Difference Methods}} for {{Mean Field Games}}},
  booktitle = {Hamilton-{{Jacobi Equations}}: {{Approximations}}, {{Numerical Analysis}} and {{Applications}}},
  author = {Achdou, Yves},
  year = {2013},
  series = {Lecture {{Notes}} in {{Mathematics}}},
  volume = {2074},
  pages = {1--47},
  publisher = {{Springer Berlin, Heidelberg}},
  OPTlocation = {{Berlin, Heidelberg}},
  OPTDOI = {10.1007/978-3-642-36433-4_1},
  url = {https://link.springer.com/10.1007/978-3-642-36433-4_1},
  urldate = {2023-12-19},
  bookauthor = {Achdou, Yves and Barles, Guy and Ishii, Hitoshi and Litvinov, Grigory L.},
  isbn = {978-3-642-36432-7 978-3-642-36433-4},
  langid = {english}
}

@book {AB2006,
    AUTHOR = {Aliprantis, Charalambos D. and Border, Kim C.},
     TITLE = {Infinite dimensional analysis},
   EDITION = {Third edition},
      NOTE = {A hitchhiker's guide},
 PUBLISHER = {Springer, Berlin},
      YEAR = {2006},
     OPTPAGES = {xxii+703},
      ISBN = {978-3-540-32696-0; 3-540-32696-0},
   MRCLASS = {46-01 (28D05 46N10 47-01 54-01 60B10 60J05)},
  MRNUMBER = {2378491},
}

@book{AC1984,
  title = {Differential Inclusions},
  author = {Aubin, Jean-Pierre and Cellina, Arrigo},
  date = {1984},
  series = {Grundlehren Der Mathematischen {{Wissenschaften}}},
  volume = {264},
  publisher = {Springer-Verlag, Berlin},
  OPTDOI = {10.1007/978-3-642-69512-4},
  isbn = {978-3-540-13105-2},
  mrnumber = {755330},
  OPTpagetotal = {xiii+342},
}

@book{ACDPS2020,
  title = {Mean {{Field Games}}},
  editor = {Cardaliaguet, Pierre and Porretta, Alessio},
  year = {2020},
  series = {Lecture {{Notes}} in {{Mathematics}}},
  volume = {2281},
  publisher = {{Springer, Cham}},
  OPTlocation = {{Cham}},
  OPTDOI = {10.1007/978-3-030-59837-2},
  url = {http://link.springer.com/10.1007/978-3-030-59837-2},
  urldate = {2023-12-19},
  isbn = {978-3-030-59836-5 978-3-030-59837-2},
  langid = {english}
}

@book{AF1990,
  title = {Set-Valued Analysis},
  author = {Aubin, Jean-Pierre and Frankowska, Hélène},
  date = {1990},
  series = {Systems \& {{Control}}: {{Foundations}} \& {{Applications}}},
  volume = {2},
  publisher = {Birkhäuser Boston, Inc., Boston},
  isbn = {978-0-8176-3478-0},
  mrnumber = {1048347},
  OPTpagetotal = {xx+461},
}

@Book{AJP2006,
 Editor = {Attal, St{\'e}phane and Joye, Alain and Pillet, Claude-Alain},
 Title = {Open quantum systems {I}. {The} {Hamiltonian} approach.},
 Series = {Lecture Notes in Mathematics},
 ISSN = {0075-8434},
 Volume = {1880},
 ISBN = {3-540-30991-8},
 Year = {2006},
 Publisher = {Springer, Berlin},
 Language = {english},
 OPTDOI = {10.1007/b128449},
 zbMATH = {5080996},
 Zbl = {1103.81004}
}

@incollection{AL2020,
 author = {Achdou, Yves and Lauri{\`e}re, Mathieu},
 title = {Mean field games and applications: numerical aspects},
 booktitle = {Mean field games.},
  editor = {Cardaliaguet, Pierre and Porretta, Alessio},
 isbn = {978-3-030-59836-5; 978-3-030-59837-2},
 pages = {249--307},
 year = {2020},
 publisher = {Springer, Cham},
 language = {english},
 OPTDOI = {10.1007/978-3-030-59837-2_4},
 zbMATH = {7303689},
 Zbl = {1457.91057}
}

@book {AZ1990,
    AUTHOR = {Appell, J\"urgen and Zabrejko, Petr P.},
     TITLE = {Nonlinear superposition operators},
    SERIES = {Cambridge Tracts in Mathematics},
    VOLUME = {95},
 PUBLISHER = {Cambridge University Press, Cambridge},
      YEAR = {1990},
     OPTPAGES = {viii+311},
      ISBN = {0-521-36102-8},
   MRCLASS = {47H30 (46E30 47-02)},
  MRNUMBER = {1066204},
MRREVIEWER = {Moshe\ Marcus},
       OPTDOI = {10.1017/CBO9780511897450},
       URL = {https://doi.org/10.1017/CBO9780511897450},
}

@book {B2007,
    AUTHOR = {Bogachev, V. I.},
     TITLE = {Measure theory. {V}ol. {I}, {II}},
 PUBLISHER = {Springer-Verlag, Berlin},
      YEAR = {2007},
     OPTPAGES = {Vol. I: xviii+500 pp., Vol. II: xiv+575},
      ISBN = {978-3-540-34513-8; 3-540-34513-2},
   MRCLASS = {28-02 (28Axx 28Cxx 46G12 60G42 60G44)},
  MRNUMBER = {2267655},
MRREVIEWER = {Ren\'e\ L.\ Schilling},
       OPTDOI = {10.1007/978-3-540-34514-5},
       URL = {https://doi.org/10.1007/978-3-540-34514-5},
}

@book{B2011,
  title = {Functional Analysis, {{Sobolev}} Spaces and Partial Differential Equations},
  author = {Brezis, Haim},
  date = {2011},
  series = {Universitext},
  publisher = {Springer, New York},
  isbn = {978-0-387-70913-0},
  mrnumber = {2759829},
  OPTpagetotal = {xiv+599}
}

@misc{B2025,
 author = {Berry, Jules},
 title = {Some error estimates for semidiscrete finite element approximations of stable solutions to mean field game systems},
 year = {2025},
 howpublished = {Preprint, {arXiv}:2511.13352},
 url = {https://arxiv.org/abs/2511.13352},
 arXiv = {arXiv:2511.13352}
}

@phdthesis{B_thesis,
  TITLE = {{Contributions to the study of second order mean field games}},
  AUTHOR = {Berry, Jules},
  URL = {https://theses.hal.science/tel-05393129},
  SCHOOL = {{Institut National des Sciences Appliqu{\'e}es - Rennes}},
  YEAR = {2025},
  OPTMONTH = Oct,
  TYPE = {Theses},
  HAL_ID = {tel-05393129},
  HAL_VERSION = {v1},
}

@Book{BC2017,
 Author = {Bauschke, Heinz H. and Combettes, Patrick L.},
 Title = {Convex analysis and monotone operator theory in {Hilbert} spaces},
 Edition = {2nd edition},
 FSeries = {CMS Books in Mathematics/Ouvrages de Math{\'e}matiques de la SMC},
 Series = {CMS Books Math},
 ISSN = {1613-5237},
 ISBN = {978-3-319-48310-8; 978-3-319-48311-5},
 Year = {2017},
 Publisher = {Springer, Cham},
 Language = {english},
 OPTDOI = {10.1007/978-3-319-48311-5},
 zbMATH = {6644703},
 Zbl = {1359.26003}
}

@article{BC2018,
  title = {Stable Solutions in Potential Mean Field Game Systems},
  author = {Briani, Ariela and Cardaliaguet, Pierre},
  year = {2018},
  Longjournal = {Nonlinear Differential Equations and Applications NoDEA},
  journal = {Nonlinear Differ. Equ. Appl.},
  volume = {25},
  number = {1},
  pages = {1},
  issn = {1021-9722, 1420-9004},
  OPTDOI = {10.1007/s00030-017-0493-3},
  url = {https://link.springer.com/10.1007/s00030-017-0493-3},
  urldate = {2023-12-19},
  langid = {english}
}

@Book{BFY2013,
 Author = {Bensoussan, Alain and Frehse, Jens and Yam, Phillip},
 Title = {Mean field games and mean field type control theory},
 FSeries = {SpringerBriefs in Mathematics},
 Series = {SpringerBriefs Math.},
 ISSN = {2191-8198},
 ISBN = {978-1-4614-8507-0; 978-1-4614-8508-7},
 Year = {2013},
 Publisher = {Springer, New York},
 Language = {english},
 OPTDOI = {10.1007/978-1-4614-8508-7},
 zbMATH = {6195704},
 Zbl = {1287.93002}
}

@Article{BLS2025,
 Author = {Berry, Jules and Ley, Olivier and Silva, Francisco J.},
 Title = {Approximation and perturbations of stable solutions to a stationary mean field game system},
 FJournal = {Journal de Math{\'e}matiques Pures et Appliqu{\'e}es. Neuvi{\`e}me S{\'e}rie},
 Journal = {J. Math. Pures Appl. (9)},
 ISSN = {0021-7824},
 Volume = {194},
 Pages = {28},
 Note = {Id/No 103666},
 Year = {2025},
 Language = {english},
 OPTDOI = {10.1016/j.matpur.2025.103666},
 zbMATH = {7980669}
}

@article {BRR1980,
    AUTHOR = {Brezzi, F. and Rappaz, J. and Raviart, P.-A.},
     TITLE = {Finite-dimensional approximation of nonlinear problems. {I}.
              {B}ranches of nonsingular solutions},
   JOURNAL = {Numer. Math.},
  FJOURNAL = {Numerische Mathematik},
    VOLUME = {36},
      YEAR = {1980/81},
    NUMBER = {1},
     PAGES = {1--25},
      ISSN = {0029-599X,0945-3245},
   MRCLASS = {65J15 (58E07 65N30)},
  MRNUMBER = {595803},
MRREVIEWER = {Erich\ Bohl},
       OPTDOI = {10.1007/BF01395985},
       URL = {https://doi.org/10.1007/BF01395985},
}

@article {BRR1981,
    AUTHOR = {Brezzi, F. and Rappaz, J. and Raviart, P.-A.},
     TITLE = {Finite-dimensional approximation of nonlinear problems. {II}.
              {L}imit points},
   JOURNAL = {Numer. Math.},
  FJOURNAL = {Numerische Mathematik},
    VOLUME = {37},
      YEAR = {1981},
    NUMBER = {1},
     PAGES = {1--28},
      ISSN = {0029-599X,0945-3245},
   MRCLASS = {65J15 (58E07 65N30)},
  MRNUMBER = {615889},
MRREVIEWER = {Erich\ Bohl},
       OPTDOI = {10.1007/BF01396184},
       URL = {https://doi.org/10.1007/BF01396184},
}

@article {BRR1982,
    AUTHOR = {Brezzi, F. and Rappaz, J. and Raviart, P.-A.},
     TITLE = {Finite-dimensional approximation of nonlinear problems. {III}.
              {S}imple bifurcation points},
   JOURNAL = {Numer. Math.},
  FJOURNAL = {Numerische Mathematik},
    VOLUME = {38},
      YEAR = {1981/82},
    NUMBER = {1},
     PAGES = {1--30},
      ISSN = {0029-599X,0945-3245},
   MRCLASS = {65J15 (58E07 65N30)},
  MRNUMBER = {634749},
       OPTDOI = {10.1007/BF01395805},
       URL = {https://doi.org/10.1007/BF01395805},
}

@book{BS2000,
 author = {Bonnans, J. Fr{\'e}d{\'e}ric and Shapiro, Alexander},
 title = {Perturbation analysis of optimization problems},
 fseries = {Springer Series in Operations Research},
 series = {Springer Ser. Oper. Res.},
 isbn = {0-387-98705-3},
 year = {2000},
 publisher = {Springer, New York},
 language = {english},
 zbMATH = {1502618},
 Zbl = {0966.49001}
}

@book{C1990,
  title = {Optimization and Nonsmooth Analysis},
  author = {Clarke, F. H.},
  date = {1990},
  series = {Classics in {{Applied Mathematics}}},
  edition = {Second edition},
  volume = {5},
  publisher = {{Society for Industrial and Applied Mathematics, Philadelphia}},
  OPTDOI = {10.1137/1.9781611971309},
  isbn = {978-0-89871-256-8},
  mrnumber = {1058436},
  pagetotal = {xii+308}
}

@article{CFI2015,
  title = {On Primal Regularity Estimates for Single-Valued Mappings},
  author = {Cibulka, Radek and Fabian, Marián and Ioffe, Alexander D.},
  date = {2015},
  Longjournaltitle = {Journal of Fixed Point Theory and Applications},
  journal = {J. Fixed Point Theory Appl.},
  volume = {17},
  number = {1},
  pages = {187--208},
  issn = {1661-7738,1661-7746},
  OPTDOI = {10.1007/s11784-015-0240-5},
  mrnumber = {3392989}
}

@article {CH2002,
    AUTHOR = {Chrysafinos, K. and Hou, L. S.},
     TITLE = {Error estimates for semidiscrete finite element approximations
              of linear and semilinear parabolic equations under minimal
              regularity assumptions},
   JOURNAL = {SIAM J. Numer. Anal.},
  FJOURNAL = {SIAM Journal on Numerical Analysis},
    VOLUME = {40},
      YEAR = {2002},
    NUMBER = {1},
     PAGES = {282--306},
      ISSN = {0036-1429,1095-7170},
   MRCLASS = {65M60 (65M12 65M15)},
  MRNUMBER = {1921920},
MRREVIEWER = {Erwin\ Schechter},
       OPTDOI = {10.1137/S0036142900377991},
       URL = {https://doi.org/10.1137/S0036142900377991},
}

@book{CR1990,
  title = {On Numerical Approximation in Bifurcation Theory},
  author = {Crouzeix, Michel and Rappaz, Jacques},
  date = {1990},
  series = {Recherches En {{Mathématiques Appliquées}}},
  volume = {13},
  publisher = {Masson, Paris; Springer-Verlag, Berlin},
  isbn = {978-2-225-81794-6},
  mrnumber = {1069945},
  OPTpagetotal = {x+165},
}

@incollection{CR1997,
  title = {Numerical Analysis for Nonlinear and Bifurcation Problems},
  booktitle = {Handbook of Numerical Analysis, {{Vol}}. {{V}}},
  author = {Caloz, Gabriel and Rappaz, Jacques},
  date = {1997},
  series = {Handb. {{Numer}}. {{Anal}}.},
  volume = {V},
  pages = {487--637},
  publisher = {North-Holland, Amsterdam},
  OPTDOI = {10.1016/S1570-8659(97)80004-X},
  isbn = {978-0-444-82278-9},
  mrnumber = {1470227},
}

@article{Danskin1966,
  author  = {J. Danskin},
  title   = {The Theory of Max-Min, with Applications},
  journal = {SIAM Journal on Applied Mathematics},
  volume  = {14},
  number  = {4},
  pages   = {641--664},
  year    = {1966}
}

@article {DMO1980,
    AUTHOR = {Dmitruk, A. V. and Milyutin, A. A. and Osmolovski\u{\i}, N.
              P.},
     TITLE = {Ljusternik's theorem and the theory of the extremum},
   JOURNAL = {Uspekhi Mat. Nauk},
  FJOURNAL = {Akademiya Nauk SSSR i Moskovskoe Matematicheskoe Obshchestvo.
              Uspekhi Matematicheskikh Nauk},
    VOLUME = {35},
      YEAR = {1980},
    NUMBER = {6(216)},
     PAGES = {11--46, 215},
      ISSN = {0042-1316},
   MRCLASS = {58C99 (49B27)},
  MRNUMBER = {601755},
MRREVIEWER = {Ji\v{r}\'{\i}\ V.\ Outrata},
}

@Article{DNPV2012,
 Author = {Di Nezza, Eleonora and Palatucci, Giampiero and Valdinoci, Enrico},
 Title = {Hitchhiker's guide to the fractional {Sobolev} spaces},
 FJournal = {Bulletin des Sciences Math{\'e}matiques},
 Journal = {Bull. Sci. Math.},
 ISSN = {0007-4497},
 Volume = {136},
 Number = {5},
 Pages = {521--573},
 Year = {2012},
 Language = {english},
 OPTDOI = {10.1016/j.bulsci.2011.12.004},
 zbMATH = {6064171},
 Zbl = {1252.46023}
}

@article {DR1982,
    AUTHOR = {Descloux, Jean and Rappaz, Jacques},
     TITLE = {Approximation of solution branches of nonlinear equations},
   JOURNAL = {RAIRO Anal. Num\'er.},
  FJOURNAL = {RAIRO Analyse Num\'erique},
    VOLUME = {16},
      YEAR = {1982},
    NUMBER = {4},
     PAGES = {319--349},
      ISSN = {0399-0516,0516-2777},
   MRCLASS = {58E07 (47H15)},
  MRNUMBER = {684829},
MRREVIEWER = {J\"urgen\ Scheurle},
}

@book {DR2014,
    AUTHOR = {Dontchev, Asen L. and Rockafellar, R. Tyrrell},
     TITLE = {Implicit functions and solution mappings},
    SERIES = {Springer Series in Operations Research and Financial
              Engineering},
   EDITION = {Second edition},
      OPTNOTE = {A view from variational analysis},
 PUBLISHER = {Springer, New York},
      YEAR = {2014},
     OPTPAGES = {xxviii+466},
      ISBN = {978-1-4939-1036-6; 978-1-4939-1037-3},
   MRCLASS = {49J53 (26B10 47J30 47N10 49K40 58C15 90C31)},
  MRNUMBER = {3288139},
MRREVIEWER = {J.\ Borwein and Matthew\ K.\ Tam},
       OPTDOI = {10.1007/978-1-4939-1037-3},
       URL = {https://doi.org/10.1007/978-1-4939-1037-3},
}

@book {EG2015,
    AUTHOR = {Evans, Lawrence C. and Gariepy, Ronald F.},
     TITLE = {Measure theory and fine properties of functions},
    SERIES = {Textbooks in Mathematics},
   EDITION = {Revised edition},
 PUBLISHER = {CRC Press, Boca Raton},
      YEAR = {2015},
     OPTPAGES = {xiv+299},
      ISBN = {978-1-4822-4238-6},
   MRCLASS = {28-01},
  MRNUMBER = {3409135},
}

@book{EG2004,
  title = {Theory and Practice of Finite Elements},
  author = {Ern, Alexandre and Guermond, Jean-Luc},
  date = {2004},
  series = {Applied {{Mathematical Sciences}}},
  volume = {159},
  publisher = {Springer-Verlag, New York},
  OPTDOI = {10.1007/978-1-4757-4355-5},
  isbn = {978-0-387-20574-8},
  mrnumber = {2050138},
  OPTpagetotal = {xiv+524}
}

@article {G1950,
    AUTHOR = {Graves, Lawrence M.},
     TITLE = {Some mapping theorems},
   JOURNAL = {Duke Math. J.},
  FJOURNAL = {Duke Mathematical Journal},
    VOLUME = {17},
      YEAR = {1950},
     PAGES = {111--114},
      ISSN = {0012-7094,1547-7398},
   MRCLASS = {46.0X},
  MRNUMBER = {35398},
MRREVIEWER = {A.\ F.\ Ruston},
       URL = {http://projecteuclid.org/euclid.dmj/1077476002},
}

@book{G1985,
 author = {Grisvard, Pierre},
 title = {Elliptic problems in nonsmooth domains},
 OPTedition = {Reprint of the 1985 hardback ed.},
 fseries = {Classics in Applied Mathematics},
 series = {Class. Appl. Math.},
 volume = {69},
 isbn = {978-1-611972-02-3},
 year = {2011},
 publisher = {Society for Industrial {and} Applied Mathematics, Philadelphia},
 language = {English},
 zbMATH = {5960425},
 Zbl = {1231.35002}
}

@book{GR1986,
  title = {Finite Element Methods for {{Navier-Stokes}} Equations},
  author = {Girault, Vivette and Raviart, Pierre-Arnaud},
  date = {1986},
  series = {Springer {{Series}} in {{Computational Mathematics}}},
  volume = {5},
  publisher = {Springer-Verlag, Berlin},
  OPTDOI = {10.1007/978-3-642-61623-5},
  isbn = {978-3-540-15796-0},
  mrnumber = {851383},
  OPTpagetotal = {x+374},
}

@article{GS2014,
  title = {Mean {{Field Games Models}}—{{A Brief Survey}}},
  author = {Gomes, Diogo A. and Saúde, João},
  year = {2014},
  Longjournal = {Dynamic Games and Applications},
  journal = {Dyn. Games Appl.},
  volume = {4},
  number = {2},
  pages = {110--154},
  issn = {2153-0785, 2153-0793},
  OPTDOI = {10.1007/s13235-013-0099-2},
  url = {http://link.springer.com/10.1007/s13235-013-0099-2},
  urldate = {2023-12-19},
  langid = {english}
}

@book {GT2001,
    AUTHOR = {Gilbarg, David and Trudinger, Neil S.},
     TITLE = {Elliptic partial differential equations of second order},
    SERIES = {Classics in Mathematics},
      OPTNOTE = {Reprint of the 1998 edition},
 PUBLISHER = {Springer-Verlag, Berlin},
      YEAR = {2001},
     OPTPAGES = {xiv+517},
      ISBN = {3-540-41160-7},
   MRCLASS = {35-02 (35Jxx)},
  MRNUMBER = {1814364},
}

@Article{I2002,
 Author = {Imbert, Cyril},
 Title = {Support functions of the {Clarke} generalized {Jacobian} and of its plenary hull},
 FJournal = {Nonlinear Analysis. Theory, Methods \& Applications. Series A: Theory and Methods},
 Journal = {Nonlinear Anal., Theory Methods Appl., Ser. A, Theory Methods},
 ISSN = {0362-546X},
 Volume = {49},
 Number = {8},
 Pages = {1111--1125},
 Year = {2002},
 Language = {english},
 OPTDOI = {10.1016/S0362-546X(01)00730-1},
 zbMATH = {1787988},
 Zbl = {1009.49017}
}

@article{I1981,
 author = {Ioffe, A. D.},
 title = {Nonsmooth analysis: differential calculus of nondifferentiable mappings},
 fjournal = {Transactions of the American Mathematical Society},
 journal = {Trans. Am. Math. Soc.},
 issn = {0002-9947},
 volume = {266},
 pages = {1--56},
 year = {1981},
 language = {english},
 OPTDOI = {10.2307/1998386},
 zbMATH = {4062185},
 Zbl = {0651.58007}
}

@book{I2017,
  title = {Variational Analysis of Regular Mappings},
  author = {Ioffe, Alexander D.},
  date = {2017},
  series = {Springer {{Monographs}} in {{Mathematics}}},
  publisher = {Springer, Cham},
  OPTDOI = {10.1007/978-3-319-64277-2},
  isbn = {978-3-319-64276-5 978-3-319-64277-2},
  mrnumber = {3727108},
  OPTpagetotal = {xxi+495},
}

@article {KP2010,
    AUTHOR = {Kim, Dongho and Park, Eun-Jae},
     TITLE = {A priori and a posteriori analysis of mixed finite element
              methods for nonlinear elliptic equations},
   JOURNAL = {SIAM J. Numer. Anal.},
  FJOURNAL = {SIAM Journal on Numerical Analysis},
    VOLUME = {48},
      YEAR = {2010},
    NUMBER = {3},
     PAGES = {1186--1207},
      ISSN = {0036-1429,1095-7170},
   MRCLASS = {65N30 (65N15)},
  MRNUMBER = {2679577},
MRREVIEWER = {Miloslav\ Feistauer},
       OPTDOI = {10.1137/090747002},
       URL = {https://doi.org/10.1137/090747002},
}

@book{L2018,
 author = {Lunardi, Alessandra},
 title = {Interpolation theory},
 edition = {3rd edition},
 fseries = {Appunti. Scuola Normale Superiore di Pisa (Nuova Serie)},
 series = {Appunti, Sc. Norm. Super. Pisa (N.S.)},
 issn = {2532-991X},
 volume = {16},
 isbn = {978-88-7642-639-1; 978-88-7642-638-4},
 year = {2018},
 publisher = {Edizioni della Normale, Pisa},
 language = {english},
 OPTDOI = {10.1007/978-88-7642-638-4},
 zbMATH = {6861769},
 Zbl = {1394.46001}
}

@article{LL2006,
  title = "{Jeux à champ moyen. I – Le cas stationnaire}",
  author = {Lasry, Jean-Michel and Lions, Pierre-Louis},
  year = {2006},
  journal = {Comptes Rendus Mathematique},
  shortjournal = {Comptes Rendus Mathematique},
  volume = {343},
  number = {9},
  pages = {619--625},
  issn = {1631073X},
  OPTDOI = {10.1016/j.crma.2006.09.019},
  url = {https://linkinghub.elsevier.com/retrieve/pii/S1631073X06003682},
  urldate = {2023-12-19},
  langid = {french}
}

@article{LL2006a,
  title = "{Jeux à champ moyen. II – Horizon fini et contrôle optimal}",
  author = {Lasry, Jean-Michel and Lions, Pierre-Louis},
  year = {2006},
  journal = {Comptes Rendus Mathematique},
  shortjournal = {Comptes Rendus Mathematique},
  volume = {343},
  number = {10},
  pages = {679--684},
  issn = {1631073X},
  OPTDOI = {10.1016/j.crma.2006.09.018},
  url = {https://linkinghub.elsevier.com/retrieve/pii/S1631073X06003670},
  urldate = {2023-12-19},
  langid = {french}
}

@article{LL2007,
  title = {Mean Field Games},
  author = {Lasry, Jean-Michel and Lions, Pierre-Louis},
  year = {2007},
  LONGjournal = {Japanese Journal of Mathematics},
  journal = {Jpn. J. Math.},
  volume = {2},
  number = {1},
  pages = {229--260},
  issn = {0289-2316, 1861-3624},
  OPTDOI = {10.1007/s11537-007-0657-8},
  url = {http://link.springer.com/10.1007/s11537-007-0657-8},
  urldate = {2023-12-19},
  langid = {english}
}

@Article{HMC2006,
 Author = {Huang, Minyi and Malham{\'e}, Roland and Caines, Peter E.},
 Title = {Large population stochastic dynamic games: closed-loop {McKean}-{Vlasov} systems and the {Nash} certainty equivalence principle},
 FJournal = {Communications in Information and Systems},
 Journal = {Commun. Inf. Syst.},
 OPTISSN = {1526-7555},
 Volume = {6},
 Number = {3},
 Pages = {221--252},
 Year = {2006},
 Language = {english},
 OPTDOI = {10.4310/CIS.2006.v6.n3.a5},
 zbMATH = {5244422},
 Zbl = {1136.91349}
}

@article{HCM2007,
  title = {Large-{{Population Cost-Coupled LQG Problems With Nonuniform Agents}}: {{Individual-Mass Behavior}} and {{Decentralized}} $\varepsilon$-{{Nash Equilibria}}},
  shorttitle = {Large-{{Population Cost-Coupled LQG Problems With Nonuniform Agents}}},
  author = {Huang, Minyi and Caines, Peter E. and Malhame, Roland P.},
  year = {2007},
  Longjournal = {IEEE Transactions on Automatic Control},
  journal = {IEEE Trans. Automat. Contr.},
  volume = {52},
  number = {9},
  pages = {1560--1571},
  issn = {0018-9286},
  OPTDOI = {10.1109/TAC.2007.904450},
  url = {http://ieeexplore.ieee.org/document/4303232/},
  urldate = {2023-12-19}
}

@book{M2005,
 author = {Mordukhovich, Boris S.},
 title = {Variational analysis and generalized differentiation. {I}: {Basic} theory. {II}: {Applications}},
 fseries = {Grundlehren der Mathematischen Wissenschaften},
 series = {Grundlehren Math. Wiss.},
 issn = {0072-7830},
 volume = {330/331},
 isbn = {3-540-25437-4; 3-540-25438-2},
 year = {2005},
 publisher = {Springer, Berlin},
 language = {english},
 OPTDOI = {10.1007/3-540-31247-1},
 zbMATH = {2176481},
 Zbl = {1100.49002}
}

@article {MQ1982,
    AUTHOR = {Maday, Y. and Quarteroni, A.},
     TITLE = {Approximation of {B}urgers' equation by pseudospectral
              methods},
   JOURNAL = {RAIRO Anal. Num\'er.},
  FJOURNAL = {RAIRO Analyse Num\'erique},
    VOLUME = {16},
      YEAR = {1982},
    NUMBER = {4},
     PAGES = {375--404},
      ISSN = {0399-0516,0516-2777},
   MRCLASS = {65N35 (35Q20 76-08 76R99)},
  MRNUMBER = {684831},
       OPTDOI = {10.1051/m2an/1982160403751},
       URL = {https://doi.org/10.1051/m2an/1982160403751},
}

@Article{OS2025,
 Author = {Osborne, Yohance A. P. and Smears, Iain},
 Title = {Finite element approximation of time-dependent mean field games with nondifferentiable {Hamiltonians}},
 FJournal = {Numerische Mathematik},
 Journal = {Numer. Math.},
 ISSN = {0029-599X},
 Volume = {157},
 Number = {1},
 Pages = {165--211},
 Year = {2025},
 Language = {english},
 OPTDOI = {10.1007/s00211-024-01447-2},
 zbMATH = {7978171}
}

@article {OS2024,
    AUTHOR = {Osborne, Yohance A. P. and Smears, Iain},
     TITLE = {Analysis and {N}umerical {A}pproximation of {S}tationary
              {S}econd-{O}rder {M}ean {F}ield {G}ame {P}artial
              {D}ifferential {I}nclusions},
   JOURNAL = {SIAM J. Numer. Anal.},
  FJOURNAL = {SIAM Journal on Numerical Analysis},
    VOLUME = {62},
      YEAR = {2024},
    NUMBER = {1},
     PAGES = {138--166},
      ISSN = {0036-1429,1095-7170},
   MRCLASS = {65N30 (65N15)},
  MRNUMBER = {4688199},
       OPTDOI = {10.1137/22M1519274},
       URL = {https://doi.org/10.1137/22M1519274},
}

@article{OS2025a,
  title = {Near and full quasi-optimality of finite element approximations of stationary second-order mean field games},
  ISSN = {1088-6842},
  url = {http://dx.doi.org/10.1090/mcom/4080},
  OPTDOI = {10.1090/mcom/4080},
  journal = {Mathematics of Computation},
  publisher = {American Mathematical Society (AMS)},
  author = {Osborne,  Yohance and Smears,  Iain},
  year = {2025},
  OPTmonth = feb
}

@article{OS2025b,
  title = {Rates of convergence of finite element approximations of second-order mean field games with nondifferentiable {Hamiltonians}},
  ISSN = {0025-5718},
  url = {http://dx.doi.org/10.1090/mcom/4189},
  OPTDOI = {10.1090/mcom/4189},
  journal = {Mathematics of Computation},
  publisher = {American Mathematical Society (AMS)},
  author = {Osborne,  Yohance and Smears,  Iain},
  year = {2026},
}

@article{OSW2025,
    author = {Osborne, Yohance A P and Smears, Iain and Wells, Harry},
    title = {A posteriori error bounds for finite element approximations of steady-state mean field games},
    journal = {IMA Journal of Numerical Analysis},
    OPTpages = {draf107},
    year = {2025},
    issn = {0272-4979},
    DOI = {10.1093/imanum/draf107},
    url = {https://doi.org/10.1093/imanum/draf107},
    eprint = {https://academic.oup.com/imajna/advance-article-pdf/doi/10.1093/imanum/draf107/65613840/draf107.pdf},
}

@article {R1984,
    AUTHOR = {Rappaz, Jacques},
     TITLE = {Approximation of a nondifferentiable nonlinear problem related
              to {MHD} equilibria},
   JOURNAL = {Numer. Math.},
  FJOURNAL = {Numerische Mathematik},
    VOLUME = {45},
      YEAR = {1984},
    NUMBER = {1},
     PAGES = {117--133},
      ISSN = {0029-599X,0945-3245},
   MRCLASS = {35J65 (65N30 76W05)},
  MRNUMBER = {761884},
MRREVIEWER = {Kurt\ Georg},
       OPTDOI = {10.1007/BF01379665},
       URL = {https://doi.org/10.1007/BF01379665},
}

@incollection {RR1982,
    AUTHOR = {Rappaz, J. and Raugel, G.},
     TITLE = {Approximation of double bifurcation points for nonlinear
              eigenvalue problems},
 BOOKTITLE = {The mathematics of finite elements and applications, {IV}},
     PAGES = {453--461},
 PUBLISHER = {Academic Press, London-New York},
      YEAR = {1982},
      ISBN = {0-12-747254-1},
   MRCLASS = {58E07 (47H12)},
  MRNUMBER = {703190},
MRREVIEWER = {Themistocles\ M.\ Rassias},
}

@book {RW1998,
    AUTHOR = {Rockafellar, R. Tyrrell and Wets, Roger J.-B.},
     TITLE = {Variational analysis},
    SERIES = {Grundlehren der mathematischen Wissenschaften},
    VOLUME = {317},
 PUBLISHER = {Springer-Verlag, Berlin},
      YEAR = {1998},
     OPTPAGES = {xiv+733},
      ISBN = {3-540-62772-3},
   MRCLASS = {49-02 (46N10 47N10 49J52 49K40 90C30)},
  MRNUMBER = {1491362},
MRREVIEWER = {Francis\ H.\ Clarke},
       OPTDOI = {10.1007/978-3-642-02431-3},
       URL = {https://doi.org/10.1007/978-3-642-02431-3},
}

@article{JS2013,
 author = {Jensen, Max and Smears, Iain},
 title = {On the convergence of finite element methods for {Hamilton}-{Jacobi}-{Bellman} equations},
 fjournal = {SIAM Journal on Numerical Analysis},
 journal = {SIAM J. Numer. Anal.},
 issn = {0036-1429},
 volume = {51},
 number = {1},
 pages = {137--162},
 year = {2013},
 language = {english},
 OPTDOI = {10.1137/110856198},
 url = {sro.sussex.ac.uk/id/eprint/45507/1/},
 zbMATH = {6160990},
 Zbl = {1266.65166}
}

@article{SS2014,
 author = {Smears, Iain and S{\"u}li, Endre},
 title = {Discontinuous {Galerkin} finite element approximation of {Hamilton}-{Jacobi}-{Bellman} equations with {Cordes} coefficients},
 fjournal = {SIAM Journal on Numerical Analysis},
 journal = {SIAM J. Numer. Anal.},
 issn = {0036-1429},
 volume = {52},
 number = {2},
 pages = {993--1016},
 year = {2014},
 language = {english},
 OPTDOI = {10.1137/130909536},
 url = {ora.ox.ac.uk/objects/uuid:7f2a46f5-f81b-48c3-87c4-eaf9ebc54d02},
 zbMATH = {6323500},
 Zbl = {1304.65252}
}

@article{SZ1990,
 author = {Scott, L. Ridgway and Zhang, Shangyou},
 title = {Finite element interpolation of nonsmooth functions satisfying boundary conditions},
 fjournal = {Mathematics of Computation},
 journal = {Math. Comput.},
 issn = {0025-5718},
 volume = {54},
 number = {190},
 pages = {483--493},
 year = {1990},
 language = {english},
 OPTDOI = {10.2307/2008497},
 zbMATH = {4141375},
 Zbl = {0696.65007}
}

@book{T2023,
 author = {Thibault, Lionel},
 title = {Unilateral variational analysis in {Banach} spaces. {Part} {I}: general theory. {Part} {II}: special classes of functions and sets},
 isbn = {978-981-12-5816-9; 978-981-12-5818-3; 978-981-12-5494-9; 978-981-12-5495-6},
 year = {2023},
 publisher = {World Scientific, Singapore},
 language = {english},
 OPTDOI = {10.1142/12797},
 zbMATH = {7502906},
 Zbl = {1536.49002}
}

@article {U2002,
    AUTHOR = {Ulbrich, Michael},
     TITLE = {Semismooth {N}ewton methods for operator equations in function
              spaces},
   JOURNAL = {SIAM J. Optim.},
  FJOURNAL = {SIAM Journal on Optimization},
    VOLUME = {13},
      YEAR = {2002},
    NUMBER = {3},
     PAGES = {805--842},
      ISSN = {1052-6234,1095-7189},
   MRCLASS = {49M15 (47J25 49J52 90C30)},
  MRNUMBER = {1972217},
MRREVIEWER = {Vaithilingam\ Jeyakumar},
       OPTDOI = {10.1137/S1052623400371569},
       URL = {https://doi.org/10.1137/S1052623400371569},
}

@book {U2011,
    AUTHOR = {Ulbrich, Michael},
     TITLE = {Semismooth {N}ewton methods for variational inequalities and
              constrained optimization problems in function spaces},
    SERIES = {MOS-SIAM Series on Optimization},
    VOLUME = {11},
 PUBLISHER = {Society for Industrial and Applied Mathematics,
              Philadelphia; Mathematical Optimization Society,
              Philadelphia},
      YEAR = {2011},
     OPTPAGES = {xiv+308},
      ISBN = {978-1-611970-68-5},
   MRCLASS = {49-02 (47J20 47J25 49J40 49M15 65-02 65K15)},
  MRNUMBER = {2839219},
MRREVIEWER = {Michael\ Hinterm\"uller},
       OPTDOI = {10.1137/1.9781611970692},
       URL = {https://doi.org/10.1137/1.9781611970692},
}

@article {W1981,
    AUTHOR = {Warga, J.},
     TITLE = {Fat homeomorphisms and unbounded derivate containers},
   JOURNAL = {J. Math. Anal. Appl.},
  FJOURNAL = {Journal of Mathematical Analysis and Applications},
    VOLUME = {81},
      YEAR = {1981},
    NUMBER = {2},
     PAGES = {545--560},
      ISSN = {0022-247X},
   MRCLASS = {58C15 (49F99)},
  MRNUMBER = {622836},
MRREVIEWER = {Sadayuki\ Yamamuro},
       OPTDOI = {10.1016/0022-247X(81)90081-0},
       URL = {https://doi.org/10.1016/0022-247X(81)90081-0},
}

\end{document}